\documentclass[10pt]{article}

\usepackage{amssymb,latexsym,amsmath,enumerate,verbatim,amsfonts,amsthm}
\usepackage{color,bm} % added by TK
\usepackage{graphicx,subcaption}
\usepackage{algorithm,algorithmic}
\usepackage{mathtools}
\usepackage[active]{srcltx}
\usepackage{cite}
\usepackage[colorinlistoftodos,prependcaption,textsize=footnotesize]{todonotes}
\usepackage[colorlinks,
	linkcolor=blue,
	anchorcolor=blue,
	citecolor=blue]{hyperref}
\usepackage[normalem]{ulem} % added by Ying
\usepackage{manyfoot}%
\usepackage{footmisc}
\usepackage{dsfont}
\usepackage{microtype}

\newcommand*\xbar[1]{%
	\hbox{%
		\vbox{%
			\hrule height 0.5pt % The actual bar
			\kern0.5ex%         % Distance between bar and symbol
			\hbox{%
				\kern-0.1em%      % Shortening on the left side
				\ensuremath{#1}%
				\kern-0.1em%      % Shortening on the right side
			}%
		}%
	}%
}
\numberwithin{equation}{section}
\newtheorem{thm}{Theorem}[section]
\newtheorem{lemma}[thm]{Lemma}
\newtheorem{definition}[thm]{Definition}

\newtheorem{theorem}[thm]{Theorem}
\newtheorem{corollary}[thm]{Corollary}
\newtheorem{proposition}[thm]{Proposition}
\newtheorem{remark}[thm]{Remark}

\def\R{{\rm I\!R}}
\def\tr{{\rm \,tr}}
\def\argmin{\mathop{\rm arg\,min}}

\def\stdCone{{\cal K}}
\def\stdFace{{\cal F}}

\def\dist{{\rm dist}}
\def\rank{{\rm rank}}
\def\Span{{\rm span\,}} % Using \Span not \span because of the existence of \span.
\def\ri{{\rm ri\,}}
\def\dim{{\rm dim\,}}
\newcommand{\cl}{{\rm cl\,}}

\def\dpps{d_{{\rm PPS}}}
\newcommand\Fr{\mathcal{F}_{{\rm r}}}
\newcommand\Fd{\mathcal{F}_{{\rm d}}}
\newcommand\Fs{\mathcal{F}_{\#}}
\newcommand\Ff{\mathcal{F}_{\infty}}
\newcommand\Fne{\mathcal{F}_{{\rm ne}}^{\#}}

\newcommand{\psdc}{\mathcal{S}_{+}^d}
\newcommand{\pdc}{\mathcal{S}_{++}^d}
\newcommand{\Kld}{\mathcal{K}_{{\rm logdet}}}
\newcommand{\Kldi}{\mathcal{K}_{{\rm logdet}}^1}
\newcommand{\Kldsi}{\mathcal{K}_{{\rm logdet}}^{*1}}
\newcommand{\Kldii}{\mathcal{K}_{{\rm logdet}}^2}
\newcommand{\Kldie}{\mathcal{K}_{{\rm logdet}}^{1\mathrm{e}}}
\newcommand{\Kldsie}{\mathcal{K}_{{\rm logdet}}^{*1\mathrm{e}}}
\newcommand{\Kldsii}{\mathcal{K}_{{\rm logdet}}^{*2}}
\newcommand{\ldt}{\log \det}

\makeatletter
\DeclareRobustCommand{\properideal}{\mathrel{\text{$\m@th\proper@ideal$}}}
\newcommand{\proper@ideal}{%
	\ooalign{$\lneq$\cr\raise.22ex\hbox{$\lhd$}\cr}%
}
\makeatother

\allowdisplaybreaks

\title{Tight error bounds for log-determinant cones without constraint qualifications}

\date{\today}

\author{
	Ying Lin\thanks{Department of Applied Mathematics, the Hong Kong Polytechnic University, Hong Kong, People's Republic of China. E-mail: \href{ying.lin@connect.polyu.hk}{ying.lin@connect.polyu.hk}.}
	\and 
	Scott B.\ Lindstrom\thanks{
		Centre for Optimisation and Decision Science, Curtin University, Australia.
		E-mail: \href{scott.lindstrom@curtin.edu.au}{scott.lindstrom@curtin.edu.au}.}
	\and
	Bruno F. Louren\c{c}o\thanks{Department of Statistical Inference and Mathematics, Institute of Statistical Mathematics, Japan.
		This author was supported partly by the JSPS  Grant-in-Aid for Early-Career Scientists  23K16844 and the Grant-in-Aid for Scientific Research (B)21H03398.
		Email: \href{bruno@ism.ac.jp}{bruno@ism.ac.jp}.}
	\and
	Ting Kei Pong\thanks{
		Department of Applied Mathematics, the Hong Kong Polytechnic University, Hong Kong, People's Republic of China.
		This author was supported partly by the Hong Kong Research Grants Council PolyU153001/22p.
		E-mail: \href{tk.pong@polyu.edu.hk}{tk.pong@polyu.edu.hk}.
	}
}

\begin{document}
\maketitle

\begin{abstract}
 In this paper, without requiring any constraint qualifications, we establish tight error bounds for the log-determinant cone, which is the closure of the hypograph of the perspective function of the log-determinant function.
  This error bound is obtained using the recently developed framework based on one-step facial residual functions.
\end{abstract}

\noindent\textbf{Keywords:} error bounds, facial residual functions, log-determinant cone

\section{Introduction}\label{sec:introduction}

The \textit{convex conic feasibility problem} has attracted a lot of attention due to its power in modeling convex problems.
Specifically, a convex conic feasibility problem admits the following form:
\begin{equation}\label{eq:conic-feasibility-problem} \tag{Feas}
	\text{Find } \quad \bm{x} \in (\mathcal{L} + \bm{a}) \cap \mathcal{K},
\end{equation}
where \( \mathcal{K} \) is a closed convex cone contained in a finite dimensional Euclidean space \( \mathcal{E} \), \( \mathcal{L} \subseteq \mathcal{E} \) is a subspace and \( \bm{a} \in \mathcal{E} \) is given.
Various aspects of \eqref{eq:conic-feasibility-problem} such as numerical algorithms and applications have been studied in the literature; see e.g., \cite{BB96,HM11}.
Here we focus on the theoretical aspects, particularly \textit{error bounds} for \eqref{eq:conic-feasibility-problem}.
To be more precise, assuming the feasibility of \eqref{eq:conic-feasibility-problem}, we want to establish inequalities that give upper bounds on the distance from an arbitrary point to \( (\mathcal{L} + \bm{a}) \cap \mathcal{K} \) based on the individual distances from the point to \( \mathcal{L} + \bm{a} \) and \( \mathcal{K} \).
As a fundamental topic in optimization \cite{HF52,LP98,LT93,Pang97,ZS17}, error bounds possess a wide range of applications, especially in algorithm design and convergence analysis.

In this paper, we consider \eqref{eq:conic-feasibility-problem} with \( \mathcal{K} = \Kld \) being the \textit{log-determinant cone} defined as
\[
	\Kld \coloneqq \left\{ (x, y, Z) \in \R \times \R_{++} \times \pdc : x \leq y \ldt (Z / y) \right\} \cup ( \R_{-} \times \{0\} \times \psdc ),
\]
where \( d \geq 1 \), \( \R_{++} \) is the positive orthant, \( \psdc \) (resp., $\pdc$) is the set of \( d \times d \) positive semidefinite (resp., positive definite) matrices.
We note that the log-determinant cone is the closure of the hypograph of the perspective function of the log-determinant function.
% It arises naturally as the \textit{spectral extension} of the \textit{logarithm cone}, which is the high dimensional generalization of \textit{exponential cone} \cite{DE21, LiLoPo20}, \cite[Chapter 5]{MC2020} and is defined as
% \[
% 	\mathcal{K}_{\log} = \left\{ (x, y, z) \in \R \times \R_{++} \times \R_{++}^d : x \leq y \sum_{i=1}^d \log (z_i / y) \right\}  \cup ( \R_- \times \{0\} \times \R_+^d ).
% \]
% As an analog of the spectral extension from the nonnegative orthant to the positive semidefinite cone, the extension from the logarithm cone to the log-determinant cone is important since it provides an example for the calculus rules for error bounds.
% It has long been an open problem as to whether there is a calculus rule for error bounds of \eqref{eq:conic-feasibility-problem} with different cones, especially with some cone and its spectral extension.
% For example, the nonnegative orthant enjoys the Lipschitzian error bounds because of its polyhedrality, while the positive semidefinite cone satisfies the H\"olderian error bounds with exponent \( 1 / 2 \) in the worst case \cite{St00}.
% One may then wonder whether there is some resembling relationship between other pairs of cones and their spectral extensions.

The log-determinant function has both theoretical and practical importance.
It is a \textit{self-concordant barrier function} for \( \psdc \), and hence it is useful for defining the \textit{logarithmically homogeneous self-concordant barrier functions (LHSCBs)} for various matrix cones.
LHSCBs are crucial for complexity analysis of the celebrated \textit{primal-dual interior point methods} for solving conic feasibility problems; see, e.g., \cite{NN94, CKV21}.
In practice, the log-determinant function appears frequently in countless real-world applications, especially in the area of machine learning, to name but a few, the sparse inverse covariance estimation \cite{FHT08}, the fused multiple graphical Lasso problem \cite{YLSWY15, ZZST21}, Gaussian process \cite{RW06, RH05}, sparse covariance selection \cite{D72, ABE08}, finding minimum-volume ellipsoids \cite{AST08, T16, VR09}, the determinantal point process \cite{KT12}, kernel learning \cite{BBFG23}, D-optimal design \cite{A69, BV04} and so on.

An elementary observation is that
\[
t \leq \log \det(Z) , Z \in \pdc \Longleftrightarrow (t,1,Z) \in \Kld,
\]
in this way, a problem that has a log-determinant term in its objective can be recast as a problem over the log-determinant cone $\Kld$. In view of the importance and prevalence of the log-determinant function, the cone $\Kld$ can also be used to handle numerous applications.

That said, if one wishes to use conic linear optimization to solve problems involving
log-determinants, it is not strictly necessary to use $\Kld$.
Indeed, it is possible, for example, to consider a reformulation
using positive semidefinite cones and exponential cones, e.g., \cite[Section 6.2.3]{MC2020}.
%To tackle optimization problems involving the log-determinant function, MOSEK formulates it in a specific scheme using the .
%This reformulation of the log-determinant function in \cite[Section 6.2.3]{MC2020} only involves several standard cones,\footnote{By convention, standard cones refer to: nonnegative orthant, second order cone, rotated second order cone, positive semidefinite cone, exponential cone and power cone.} it reveals that the optimization problems involving the log-determinant function can be solved in many other solvers like Alfonso \cite{DY01} and DDS \cite{KT19}.

A natural question then is whether it is more advantageous to use a reformulation or
handle $\Kld$ directly.
Indeed, Hypatia implements the log-determinant cone as a predefined \textit{exotic cone} \cite{CKV21} and their numerical experiments show that the direct use of the log-determinant cone gives numerical advantages compared to the use of reformulations, see \cite{CKV23} and \cite[Sections~8.4.1, 8.4.2]{CKV22}.
One reason that other formulations may be less efficient is that they increase the dimension of the problem. Another drawback is that they do not capture the geometry of the hypograph of the log determinant function as tightly.

%This provides some evidence that reformulation is not efficient as it increases the scale of the problem.
%Furthermore, the log-determinant cone seems to capture more precisely the geometry of the hypograph of  log-determinant function in comparison with an approach that uses positive semidefinite cones and exponential cones.%  does not make use of some of the ``hidden nice structures''\todo{B:I am somewhat unsure about this part...} of the log-determinant cone.

Motivated by these results, we present a study of the facial structure of $\Kld$ and its properties in connection to feasibility problems as in  \eqref{eq:conic-feasibility-problem}.

Specifically, we deduce tight error bounds for \eqref{eq:conic-feasibility-problem} with \( \mathcal{K} = \Kld \) by deploying a recently developed framework \cite{LiLoPo20, LiLoPo21}, which is based on the \textit{facial reduction algorithm} \cite{BW81, Pa13, WM13} and \textit{one-step facial residual functions} \cite[Definition 3.4]{LiLoPo20}.
This framework has been used with success to develop concrete error bounds for symmetric cones \cite{L17}, exponential cones \cite{LiLoPo20}, $p$-cones \cite{LiLoPo21} and power cones \cite{LLLP23}.
Although the log-determinant cone is a high-dimensional generalization of the exponential cone, whose error bounds were studied in depth in \cite{LiLoPo20}, the derivation of the error bounds for the log-determinant cone is not straightforward.
% Several gaps occur in the analysis.
Indeed, the exponential cone is three dimensional and so its facial structure can be visualized explicitly. %, see \cite[Fig. 1]{LiLoPo20}, with faces lower than three dimensions.
In contrast, with a higher dimension, the log-determinant cone has a more involved facial structure. %, which are difficult to detect and handle.
%Moreover, the log-determinant cone involves matrix spectra and operations, leading to higher complexity in the analysis.

This paper is organized as follows.
In Section \ref{sec:preliminaries}, we recall notation and preliminaries.
In Section \ref{sec:log-det-cone}, we develop tight error bounds for \eqref{eq:conic-feasibility-problem} with \( \mathcal{K} = \Kld \).

\section{Notation and preliminaries}\label{sec:preliminaries}

In this paper, we will use lowercase letters to represent real scalars, bold-faced letters to denote vectors (including ``generalized'' vectors such as \( \bm{x} \in \Kld \), which consists of real scalars and a matrix), capital letters to denote matrices and curly capital letters for spaces, subspaces or cones.

Let \( \mathcal{E} \) be a finite dimensional Euclidean space, and \( \R_{+} \) and \( \R_{++} \) be the set of nonnegative and positive real numbers, respectively.
For a real number \( x \), we denote that \( (x)_+ := \max \{x , 0\} \).
The inner product on \( \mathcal{E} \) is denoted by \( \langle \cdot , \cdot  \rangle  \) and the induced norm is denoted by \( \| \cdot \|  \).
With the induced norm, for any \( \bm{x} \in \mathcal{E} \) and a closed convex set \( \Omega \subseteq \mathcal{E} \), we denote the projection of \( \bm{x} \) onto \( \Omega \) by \( P_\Omega(\bm{x}) := \argmin _{\bm{y} \in \Omega} \| \bm{x} - \bm{y} \|  \) and the distance between \( \bm{x} \) and \( \Omega \) by \( \dist(\bm{x}, \Omega) := \inf_{\bm{y} \in \Omega} \| \bm{x} - \bm{y} \| = \| \bm{x} - P_\Omega(\bm{x}) \|  \).
For any \( \eta \geq 0 \), we denote the ball centered at \( \bm{x}_0 \) with radius \( \eta  \) by \( B(\bm{x}_0; \eta) := \{\bm{x} \in \mathcal{E} \,: \, \| \bm{x} - \bm{x}_0\| \leq \eta \} \).
For notational simplicity, we will use \( B(\eta) \) to denote the ball centered at \( \bm{0} \) with radius \( \eta  \).
Meanwhile, we will denote the orthogonal complement of \( \Omega \) by \( \Omega^{\perp } \coloneqq \left\{ \bm{v} \in \mathcal{E} \,: \,\langle \bm{x}, \bm{v} \rangle = 0 \,\,\,\,\forall \,\bm{x} \in \Omega \right\} \).

\subsection{Matrices}\label{subsec:matrices}
We use \( \R^{m\times n} \) to denote the set of all real \( m \times n \) matrices and $\mathcal{S}^d$ to denote the set of symmetric $d\times d$ matrices.
The \( n \times n \) identity matrix will be denoted by \( I_n \).
Let \( \psdc \) and \( \pdc \) be the set of symmetric \( d \times d \) \textit{positive semidefinite matrices} and \( d \times d \) \textit{positive definite matrices} respectively.
%, which are both contained in \( \mathcal{S}^d \), the space of \( d \times d \) real symmetric matrices.
The interior of \( \psdc \) is \( \pdc \).
We write \( X \succ 0 \) (resp., \( X \succeq 0 \)) if \( X \in \pdc \) (resp., \( X \in \psdc \)).
For any \( X \in \mathcal{S}^d \), we let \( \lambda _i(X) \in \R \) denote the \( i \)-th \textit{eigenvalue} of \( X \) such that \( \lambda _d(X) \geq \lambda _{d-1}(X) \geq \dots \geq \lambda _1(X) \).
We will use \( \lambda _{\max }(X) \) and \( \lambda _{\min }(X) \) to denote the maximum and minimum eigenvalues of \( X \), respectively.
The \textit{rank} of \( X \) is defined by the number of non-zero eigenvalues, denoted by \( \rank (X) \).
The \textit{trace} (resp., \textit{determinant}) of \( X \) is defined by \( \tr(X) := \sum_{i=1}^d \lambda _i (X)  \) (resp., \( \det(X) := \prod_{i=1}^d \lambda _i(X)  \)).
With these, we recall that the Frobenius inner product on \( \mathcal{S}^d \) is given by \( \langle X, Y \rangle := \tr(XY) \) for any \( X, Y \in \mathcal{S}^d \), and the Frobenius norm is \( \| X \| _F := \sqrt{\tr(X^2)} \).
For any \( X \in \psdc  \) (resp., \( X \in \pdc  \)), we have \( \lambda _i(X) \geq 0 \) (resp., \( \lambda _i(X) > 0 \)).
We hence also have for any \( X, Y \in \psdc \) that
\begin{equation}\label{eqn:traceproperties}
	\!\tr(XY) \geq \lambda_{\min}(Y)\tr(X) \geq 0 \ \text{and moreover}, \ \tr(XY) = 0 \iff  XY=0.\!\!
\end{equation}
% The nuclear norm of a matrix \( X \in \R^{m\times n} \) is \( \| X \|_{*} := \sum_{i=1}^{\min \{m, n\}} \sigma _i(X) \), where \( \sigma _i(X) = \sqrt{\lambda _i(X^{\top} X)} \) is the \( i \)-th singular value of \( X \).
% Its dual norm is the spectral norm \( \| X \|_{\infty} = \sigma _{\max }(X) \).

For a given non-zero positive semidefinite matrix, the next result connects its determinant with its trace and rank.

\begin{lemma}\label{lemma:rel-det-tr-Z-rank}
	Let \( Z \in \psdc \setminus \{\bm{0}\} \). Then for any \( \eta > 0 \), there exists \( C>0 \) so that
	\begin{equation}\label{eq:rel-det-tr-Z-rank}
		(\det (R))^{\frac{1}{d}}\le C [\tr(RZ)]^{\frac{\rank(Z)}{d}} \quad\quad \forall  R \in B(\eta) \cap \psdc .
	\end{equation}
\end{lemma}

\begin{proof}
	Let \( Z = Q \Sigma Q^{\top}  \) be an eigendecomposition of \( Z \), where \( Q \) is orthogonal and \( \Sigma \) is diagonal, and let \( \mathsf{r} \) be the rank of \( Z \). Then \( \mathsf{r} \geq 1 \) since \( Z \neq 0 \).
	Without loss of generality, we may suppose that the first \( \mathsf{r} \) diagonal entries of \( \Sigma  \), denoted as \( \sigma_1, \sigma _2, \dots ,\sigma _{\mathsf{r}} \), are nonzero and are arranged in descending order.
	Then \( \sigma _{\mathsf{r}} \) is the smallest positive eigenvalue of \( Z \) and we have for any \( R \in B(\eta) \cap \psdc  \) that
	\begin{equation}\label{eq:aux1-proof-lemma-rel-det-tr}
    \begin{split}
      \tr(RZ) &= \tr(RQ\Sigma Q^{\top} ) = \tr(Q^{\top} R Q\Sigma) = \tr([Q^{\top} R Q]_{\mathsf{r}} [\Sigma]_{\mathsf{r}}) \\
      & \overset{(\text{a})}{\geq } \sigma _{\mathsf{r}} \tr([Q^{\top} R Q]_{\mathsf{r}}) \overset{(\text{b})} \geq \sigma _{\mathsf{r}} \sum_{i=1}^{\mathsf{r}} \lambda _i(Q^{\top} RQ),
    \end{split}
	\end{equation}
	where \( [A]_{\mathsf{r}} \) is the submatrix of \( A \) formed by \( A_{ij} \) for \( 1 \leq i,j \leq \mathsf{r} \), (a) holds since \( [Q^{\top} RQ]_{\mathsf{r}} \succeq 0 \) (thanks to \( R \succeq 0 \)), (b) is true because of the interlacing theorem (see~\cite[Theorem 4.3.8]{HJ90}).

	Next, note that we have for any \( R \in B(\eta) \cap \psdc  \) that
	\begin{equation}\label{eqn:detlambda}
    \begin{split}
      \det(R) & = \det(Q^{\top} RQ) = \prod_{i=1}^d\lambda _i(Q^{\top} RQ) \stackrel{\text{(a)}}{\leq} \eta ^{d-\mathsf{r}} \prod_{i=1}^{\mathsf{r}}\lambda _i(Q^{\top} RQ) \\
      &\stackrel{\text{(b)}}{\leq} \eta ^{d-\mathsf{r}}\left( \frac{1}{\mathsf{r}} \sum_{i=1}^{\mathsf{r}}\lambda _i(Q^{\top} RQ)  \right)^{\mathsf{r}},
    \end{split}
	\end{equation}
	where (a) holds because
	\begin{enumerate}[(i)]
		\item \( \forall \,i = 1, 2, \dots, d,\, \lambda _i(Q^{\top} RQ) = \lambda _i (R) \) since \( Q \) is orthogonal.
		\item \( R \in B(\eta) \cap \psdc  \implies \| R \| _F = \sqrt{\tr(R^2)} \leq \eta \implies \forall \, i = 1, 2, \dots, d,\, \lambda _i(R) \leq \eta \).
	\end{enumerate}
	and (b) comes from the AM-GM inequality.
	Combining \eqref{eqn:detlambda} with~\eqref{eq:aux1-proof-lemma-rel-det-tr} gives
	\[ (\det(R))^{\frac{1}{d}} \leq \eta ^{1 - \frac{\mathsf{r}}{d}} \cdot \left( \frac{1}{\mathsf{r}} \sum_{i=1}^{\mathsf{r}} \lambda _i(Q^{\top} RQ)  \right)^{\frac{\mathsf{r}}{d}} \leq \eta ^{1-\frac{\mathsf{r}}{d}} \cdot \left( \frac{1}{\mathsf{r}\sigma _{\mathsf{r}}} \tr(RZ) \right)^{\frac{\mathsf{r}}{d}} \]
  whenever \( R \in B(\eta) \cap \psdc \).
	Hence, we see that~\eqref{eq:rel-det-tr-Z-rank} holds with \( C = \eta ^{1 - \frac{\mathsf{r}}{d}}(\mathsf{r} \sigma_{\mathsf{r}})^{-\frac{\mathsf{r}}{d}} \).
\end{proof}

\subsection{Error bounds for conic feasibility problems}\label{subsec:error-bounds-feas}
We first recall the definition of error bounds.

\begin{definition}[Error bounds~\cite{LL22, Pang97}]\label{def:error-bounds}
Suppose \eqref{eq:conic-feasibility-problem} is feasible. We say that \eqref{eq:conic-feasibility-problem} satisfies an error bound with a residual function \( r:\R_+ \to \R_+ \) if for every bounded set \( B \subseteq \mathcal{E} \), there exists a constant \( c_B > 0 \) such that
	\[
    \dist(\bm{x}, \stdCone \cap (\mathcal{L} + \bm{a})) \leq c_B r \left( \max \left\{\dist(\bm{x}, \stdCone), \dist(\bm{x}, \mathcal{L} + \bm{a})\right\} \right)  \quad\quad \forall \bm{x} \in B.
  \]
\end{definition}
We remark that typically it is required that $r$  satisfy  $r(0) = 0$, be nondecreasing  and be right-continuous at $0$. Under these conditions, the error bound in Definition~\ref{def:error-bounds} can be understood in the context of \textit{consistent error bound functions}, see \cite[Definition~3.1]{LL22} and this footnote.\footnote{For $B_b = B(b)$ (where $B(b)$ is the ball centered at the origin with radius $b$), if Definition~\ref{def:error-bounds} holds, then $c_{B_b}$ can be taken to be a nondecreasing function of $b$ (since considering a larger constant still preserves the error bound inequality). In this way, the function $\Phi: \R_{+}\times \R_+ \to \R_+$ given by
$\Phi(a,b) \coloneqq c_{B_b}r(a)$ satisfies \cite[Definition~3.1]{LL22}, provided that $r$ has the aforementioned properties.}
With different residual functions, we will have different error bounds, among which the Lipschitzian and H\"olderian error bounds are most widely studied in the literature.

Particularly, we say that \eqref{eq:conic-feasibility-problem} satisfies a \textit{uniform H\"olderian error bound} with exponent \( \gamma  \in (0, 1] \) if  Definition~\ref{def:error-bounds} holds with $r = (\cdot)^\gamma$ for every bounded set $B$.
That is, for every bounded set \( B \subseteq \mathcal{E} \), there exists a constant \( \kappa _B > 0 \) such that
\[
\dist(\bm{x}, \stdCone \cap (\mathcal{L} + \bm{a})) \!\leq\! \kappa _B \max\left\{\dist(\bm{x}, \stdCone), \dist(\bm{x}, \mathcal{L} + \bm{a})\right\}^{\gamma },
\]
for all \( \bm{x} \in B \).
If \( \gamma  = 1 \), then the error bound is said to be \textit{Lipschitzian}. H\"olderian error bounds are a particular case of a {consistent error bound}, see \cite[Theorem~3.5]{LL22}.

Let \( \mathcal{K} \) be a closed convex cone contained in \( \mathcal{E} \) and \( \mathcal{K}^{*} \) be its dual cone.
We will denote the boundary, relative interior, linear span, and dimension of \( \mathcal{K} \) by \( \partial \mathcal{K}, \ri\mathcal{K}, \Span \mathcal{K} \) and \( \dim\mathcal{K} \), respectively.
If \( \mathcal{K} \cap -\mathcal{K} = \{\bm{0}\} \), then \( \mathcal{K} \) is said to be \textit{pointed}.
If \( \mathcal{F} \subseteq  \mathcal{K} \) is a face of \( \mathcal{K} \), i.e., for any \( \bm{x}, \bm{y} \in \mathcal{K} \) such that \( \bm{x} + \bm{y} \in \mathcal{F} \), we have \( \bm{x}, \bm{y} \in \mathcal{F} \), then we write \( \mathcal{F} \unlhd \mathcal{K} \).\footnote{By convention, we only consider nonempty faces.}
If further \( \mathcal{F} = \mathcal{K} \cap \{\bm{n}\}^{\perp } \) for some \( \bm{n} \in \mathcal{K}^{*} \), we say that \( \mathcal{F} \) is an \textit{exposed face} of \( \mathcal{K} \).
A face \( \mathcal{F} \) is said to be \textit{proper} if \( \mathcal{F} \neq \mathcal{K} \), and we denote it by \( \mathcal{F} \properideal \mathcal{K} \).
If \( \mathcal{F} \) is proper and \( \mathcal{F} \neq \mathcal{K} \cap - \mathcal{K} \), then \( \mathcal{F} \) is said to be a \textit{nontrivial} face of \( \mathcal{K} \).

The \textit{facial reduction} algorithm~\cite{BW81,Pa13,WM13} and the \textit{FRA-poly} algorithm~\cite{LMT18} play important roles in making full use of the facial structure of a cone; see also~\cite[Section 3]{LiLoPo20}.
More precisely, assuming \eqref{eq:conic-feasibility-problem} is feasible, the facial reduction algorithm aims at finding the minimal face that contains the feasible region and satisfies some \textit{constraint qualification}.
One of the most commonly used constraint qualification is the so-called \textit{partial-polyhedral Slater (PPS) condition} \cite[Definition 3]{L17}.
For \eqref{eq:conic-feasibility-problem}, if \( \mathcal{K} \) and \( \mathcal{L} + \bm{a} \) satisfy the PPS condition, then a Lipschitzian error bound holds for \( \mathcal{K} \) and \( \mathcal{L} + \bm{a} \); see \cite[Corollary~3]{BBL99} and the discussion preceding \cite[Proposition~2.3]{LiLoPo20}.
Thanks to this property, we can apply the facial reduction algorithm to deduce the error bounds based on the \textit{one-step facial residual function} \cite[Definition 3.4]{LiLoPo20} without requiring any constraint qualifications, as in the framework developed recently in \cite{L17}; see also \cite{LiLoPo20, LiLoPo21}.
This framework is highly inspired by the fundamental work of Sturm on error bound for LMIs, see \cite{St00}.
For the convenience of the reader, we recall the definition of the one-step facial residual function as follows.
\begin{definition}[One-step facial residual function (\( \mathds{1} \)-FRF)]\label{def:one-step-facial-residual-functions}
	Let \( \mathcal{K} \) be a closed convex cone and \( \bm{n} \in \mathcal{K}^{*} \).
	Suppose that \( \psi_{\mathcal{F}, \bm{n}} :\R_{+} \times \R_{+} \to \R_{+} \) satisfies the following properties:
	\begin{enumerate}[(i)]
		\item \( \psi _{\mathcal{F}, \bm{n}} \) is nonnegative, nondecreasing in each argument and it holds that \( \psi _{\mathcal{F}, \bm{n}} (0, t) = 0 \) for every \( t \in \R_{+} \).
		\item The following implication holds for any \(\bm{x} \in \Span\mathcal{K} \) and \( \epsilon \geq 0 \):
		      \[ \dist(\bm{x}, \mathcal{K}) \leq \epsilon , \langle \bm{x}, \bm{n} \rangle  \leq \epsilon  \implies \dist(\bm{x}, \mathcal{K} \cap \{\bm{n}\}^{\perp }) \leq \psi _{\mathcal{K}, \bm{n}} (\epsilon, \| \bm{x} \| ). \]
	\end{enumerate}
	Then \( \psi _{\mathcal{F}, \bm{n}} \) is said to be a one-step facial residual function (FRF) for \( \mathcal{K} \) and \( \bm{n} \).
\end{definition}
The one-step facial residual function is used in each step of the facial reduction algorithm to connect a face and its subface until a face \( \mathcal{F} \) is found such that \( \mathcal{F} \) and \( \mathcal{L} + \bm{a} \) satisfy the PPS condition.
Then the error bound for \( \mathcal{K} \) and \( \mathcal{L} + \bm{a} \) can be obtained as a special composition of those one-step facial residual functions.
Due to the importance of the PPS condition in this framework, we shall define the \textit{distance to the PPS condition} of a \textbf{feasible} \eqref{eq:conic-feasibility-problem}, denoted by \( \dpps(\mathcal{K}, \mathcal{L} + \bm{a}) \), as the length \textit{minus one} of the \textit{shortest} chain of faces (among those chains constructed as in \cite[Proposition~5]{L17}) such that the PPS condition holds for the final face in the chain and \( \mathcal{L} + \bm{a} \).
% According to~\cite[Proposition 24]{L17}, we have
% \begin{equation}\label{eq:distance-PPS-upper-bound}
% 	\dpps (\mathcal{K}, \mathcal{L} + \bm{a}) \leq \min \left\{ \sum_{i=1}^s \ell _{{\rm poly}} (\mathcal{K}^i), \dim(\mathcal{L}^{\perp } \cap \{\bm{a}\}^{\perp })  \right\}
% \end{equation}
% and there is a chain of faces as in \cite[Proposition~5]{L17} of length \( \dpps(\mathcal{K}, \mathcal{L} + \bm{a}) + 1 \).

Before ending this subsection, we present a lemma and a proposition that will help simplify our subsequent analysis.
\begin{lemma}[Formula of \( \| \bm{w} - \bm{u} \| \)]\label{lemma:norm-wk-uk}
	Let \( \mathcal{K} \) be a closed convex cone and \( \bm{n} \in \partial \mathcal{K}^{*} \setminus \{\bm 0\} \) be such that \( \mathcal{F}:=\{\bm{n}\}^{\perp } \cap \mathcal{K} \) is a nontrivial exposed face of \( \mathcal{K} \).
	Let \( \eta > 0 \) and let \( \bm{v} \in \partial \mathcal{K} \cap B(\eta) \setminus \mathcal{F} , \bm{w} = P_{\{\bm{n}\}^{\perp } }(\bm{v}), \bm{u} = P_{\mathcal{F}} (\bm{w})\) and \( \bm{w} \neq \bm{u} \).
	Then, we have
	\begin{equation}
		\label{eq:form-square-norm-w-u}
		\| \bm{w} - \bm{u} \|^2 = \| \bm{v} - \bm{u} \|^2 - \| \bm{w} - \bm{v} \|^2,
	\end{equation}
	and,
	\begin{equation}\label{eq:norm-wk-uk}
		\| \bm{w} - \bm{u} \| \leq \| \bm{v} - \bm{u} \| = \dist(\bm{v}, \mathcal{F}).
	\end{equation}
\end{lemma}

\begin{proof}
	Since \( \bm{w} = P_{\{\bm{n}\}^{\perp }}(\bm{v}) \), we have
	\[ \bm{w} = \bm{v} - \frac{\langle \bm{n}, \bm{v} \rangle }{\| \bm{n} \| ^2}\bm{n} \quad \text{ and } \quad \| \bm{w} - \bm{v} \| = \frac{| \langle \bm{n}, \bm{v} \rangle  | }{\| \bm{n} \| }.   \]
	Moreover, we can notice that \( \bm{w} \perp \bm{n} \) and \( \bm{u} \perp \bm{n} \).

	Now, for any \( \widetilde{\bm{u}} \in \{\bm{n}\}^{\perp } \), we have
	\begin{align*}
		                        & \| \bm{w} - \widetilde{\bm{u}} \|^2 = \left\| \bm{v} - \frac{\langle \bm{n}, \bm{v} \rangle }{\| \bm{n} \| ^2} \bm{n} - \widetilde{\bm{u}}  \right\|^2  =  \left\| \bm{v} - \frac{\langle \bm{n}, \bm{v} \rangle }{\| \bm{n} \| ^2}\bm{n}  \right\| ^2 - 2\left\langle \bm{v} - \frac{\langle \bm{n}, \bm{v} \rangle }{\| \bm{n} \| ^2} \bm{n}, \widetilde{\bm{u}}  \right\rangle + \|\widetilde{\bm{u}}\|^2 \\
		\overset{(\text{a})}{=} & \| \bm{v} \| ^2 - 2\frac{\langle \bm{n}, \bm{v} \rangle }{\| \bm{n} \| ^2} \langle \bm{n}, \bm{v} \rangle + \frac{\langle \bm{n}, \bm{v} \rangle ^2}{\| \bm{n} \| ^2}  - 2\langle \bm{v}, \widetilde{\bm{u}} \rangle + \| \widetilde{\bm{u}} \| ^2 = \| \bm{v} \| ^2 - \frac{\langle \bm{n}, \bm{v} \rangle^2 }{\| \bm{n} \| ^2} - 2\langle \bm{v}, \widetilde{\bm{u}} \rangle + \| \widetilde{\bm{u}} \| ^2 \\
		=                       & \| \bm{v} - \widetilde{\bm{u}} \| ^2 - \| \bm{w} - \bm{v} \| ^2,
	\end{align*}
	where (a) comes from the fact that \( \widetilde{\bm{u}} \perp \bm{n} \).
  This proves \eqref{eq:form-square-norm-w-u} upon letting \( \widetilde{\bm{u}} = \bm{u} \).

	The above display implies that for any \( \widetilde{\bm{u}} \in \{\bm{n}\}^{\perp }  \), \( \bm{v}, \bm{w}, \widetilde{\bm{u}} \) are three vertices of a right-angled triangle with \( \angle \bm{w} \) being the right angle.
	This fact also leads us to the observation that \( \bm{u} = P_{\mathcal{F}} (\bm{v}) \).
	Indeed, suppose not, then there exists \( \widehat{\bm{u}} = P_{\mathcal{F}} (\bm{v}) \) with \( \widehat{\bm{u}} \neq \bm{u} \) such that \( \| \widehat{\bm{u}} - \bm{v} \| < \| \bm{u} - \bm{v} \| \).
	Then \( \widehat{\bm{u}} \in \{\bm{n}\}^{\perp } \) and hence \( \bm{v}, \bm{w}, \widehat{\bm{u}} \) form a new right-angled triangle.
	Thus,
	\[ \| \bm{w} - \widehat{\bm{u}} \|^2 = \| \bm{v} - \widehat{\bm{u}} \|^2 - \| \bm{w} - \bm{v} \|^2 < \| \bm{u} - \bm{v} \| ^2 - \| \bm{w} - \bm{v} \| ^2 = \| \bm{w} - \bm{u} \|^2 . \]
	Since \( \widehat{\bm{u}} \in \mathcal{F} \), the above display contradicts the fact that \( \bm{u} = P_{\mathcal{F}} (\bm{w}) \).
	Therefore, \( \bm{u} = P_{\mathcal{F}}(\bm{v}) \) and \( \| \bm{u} - \bm{v} \| = \dist(\bm{v}, \mathcal{F}) \).
\end{proof}

The next proposition states an error bound result related to the positive semidefinite cone.
We present a proof based on the results in~\cite{L17}, although it can also be obtained from Sturm's error bound in \cite{St00}.
\begin{proposition}[Error bound for positive semidefinite cones]\label{prop:psd-error-bound}
	Let \( Z \in \psdc \setminus \{\bm{0}\} \) and \( \eta > 0 \), then there exists \( C_P > 0 \) such that
	\begin{equation}\label{eq:psd-error-bound}
		\dist(Y, \psdc \cap \{Z\}^{\perp }) \leq C_P \tr(YZ)^{\alpha }  \quad\quad \forall Y \in \psdc \cap B(\eta),
	\end{equation}
	where
	\begin{equation}\label{eq:alpha-psd-cone-error-bound}
		\alpha :=
		\begin{cases}
			\frac{1}{2} & \text{if } \rank(Z) < d, \\
			1           & \text{otherwise}.
		\end{cases}
	\end{equation}
\end{proposition}

\begin{proof}
	By~\cite[Proposition 27, Theorem 37]{L17}, there exists \( C_0 > 0 \) such that
	\[ \dist(Y, \psdc \cap \{Z\}^{\perp }) \leq C_0 \max \{\dist(Y, \psdc), \dist(Y, \{Z\}^{\perp })\}^{\alpha } \quad \text{whenever } Y \in B(\eta), \]
	where \( \alpha  \) is defined as in~\eqref{eq:alpha-psd-cone-error-bound}.

	If further \( Y \in \psdc  \), then \( \dist(Y, \psdc) = 0 \); moreover,
	\(\dist(Y, \{Z\}^{\perp }) = \frac{|\tr(YZ) |}{\|Z\|_F} = \frac{\tr(YZ)}{\|Z\|_F}  \).
	Therefore, letting \( C_P := C_0/\|Z\|_F^{\alpha} \), we can obtain~\eqref{eq:psd-error-bound}.
\end{proof}

\section{Error bounds for the log-determinant cones}\label{sec:log-det-cone}
In this section, we will compute the one-step facial residual functions for the log-determinant cones, and obtain error bounds.
Let \( d \) be a positive integer and \( {\rm sd}(d) := \frac{d(d+1)}{2} \) be the dimension of $\mathcal{S}^d$, we consider the \( ({\rm sd}(d) + 2) \)-dimensional space \( \R \times \R \times \mathcal{S}^d \).
We let \( \bm{x} := (\bm{x}_x, \bm{x}_y, \bm{x}_Z) \) denote an element of \( \R \times \R \times \mathcal{S}^d \), where \( \bm{x}_x \in \R, \bm{x}_y \in \R \) and \( \bm{x}_Z \in \mathcal{S}^d \), and equip \( \R \times \R \times \mathcal{S}^d \) with the following inner product:
\[ \langle \bm{x}, \bm{z} \rangle = \bm{x}_x \bm{z}_x + \bm{x}_y \bm{z}_y + \tr (\bm{x}_Z \bm{z}_Z) \quad \text{for any}\quad \bm{x}, \bm{z} \in \R \times \R \times \mathcal{S}^d. \]

Recall that the log-determinant cone is defined as follows.
\begin{align}
	\Kld & := \left\{ (x, y, Z) \!\in\! \R \!\times\! \R_{++} \!\times\! \pdc : x \leq y \ldt (Z / y) \right\} \cup ( \R_{-} \!\times\! \{0\} \!\times\! \psdc ) \label{eq:log-det-cone}                         \\
	     & = \left\{ (x, y, Z) \!\in\! \R \!\times\! \R_{++} \!\times\! \pdc : y^d e^{x / y} \leq \det (Z) \right\} \cup ( \R_{-} \!\times\! \{0\} \!\times\! \psdc ). \label{eq:exponential-form-log-det-cone}
\end{align}
Its dual cone is given by\footnote{Here is a sketch. Let $f:\pdc \to \R$ be such that
$f(Z) = -d -\log \det(Z)$ and let $\stdCone$ be the closed convex cone generated by the set $C \coloneqq \{(1,y,Z) \mid f(Z) \leq y \}$. We have $\stdCone = \cl\{(x,y,Z) \in \R_{++} \times \R \times \pdc \mid xf(Z/x) \leq y  \}$. That is, $\stdCone$ is the closure of $\{(x,y,Z) \in \R_{++} \times \R \times \pdc \mid y \geq x(-\log \det(Z/x) -d) \}$. By \cite[Theorem~14.4]{RT97}, the closed convex cone $\bar{\stdCone}$ generated by $\{(1,v,W) \mid v \geq f^*(W) \}$ satisfies $\bar{\stdCone} = \{(u,v,W) \mid (-v,-u,W) \in \stdCone^\circ \}$,
where $\stdCone^\circ$ is the polar of $\stdCone$. The conjugate of $f$ is $-\log \det(-W)$ for $W \in -\pdc$.  Overall, we conclude that $(x,y,Z) \in \stdCone^*$ iff $(-x,-y,-Z) \in \stdCone^\circ $ iff  $(y,x,-Z)$ is in the closure of  $\{(u,v,W) \in \R_{++} \times \R \times -\pdc \mid v \geq -u\log \det(-W/u) \} $.
Finally, this implies that $(x,y,Z) \in \stdCone^*$ if and only if
$(x,y,Z)$ is in the closure of
 $\{(x,y,Z) \in \R \times \R_{++} \times \pdc \mid -x \leq y\log \det(Z/y) \} $. This means $(x,y,Z) \in \stdCone^*$ iff $(-x,y,Z) \in \Kld$. Thus, we conclude that the cones in \eqref{eq:log-det-cone} and \eqref{eq:log-det-cone-dual} are dual to each other.}
\begin{align}
	\!\!\Kld^{*} & \!\!:=\!\! \left\{\! (x, y, Z) \!\in\! \R_{--}\! \!\times\! \R \!\times\! \pdc\!\! :\! y \!\geq\! x (\ldt (-Z / x)\! + d) \!\right\} \!\cup\! ( \{0\} \!\times\! \R_{+} \!\times\! \psdc )\!\!\!  \label{eq:log-det-cone-dual}                         \\
	         & \!\!=\!\! \left\{\! (x, y, Z) \!\in\! \R_{--}\! \!\times\! \R \!\times\! \pdc\! : (-x)^d e^{y / x} \leq e^d \det (Z) \!\right\} \!\cup\! ( \{0\} \!\times\! \R_{+} \!\times\! \psdc ) . \label{eq:exponential-form-log-det-cone-dual}
\end{align}

One should notice that if \( d = 1 \), then the log-determinant cone reduces to the exponential cone, whose corresponding error bound results were discussed in~\cite{LiLoPo20}.
Hence, without loss of generality, {\bf we assume that \( d > 1 \) in the rest of this paper}.
% Also, to make use of the techniques developed in~\cite{LiLoPo20}, we write the exponential forms of the log-determinant cones and the dual cones in~\eqref{eq:exponential-form-log-det-cone} and~\eqref{eq:exponential-form-log-det-cone-dual}, respectively.
% However, in this paper, we will try to exploit properties by mainly analysing its logarithm form.
Notice from~\eqref{eq:exponential-form-log-det-cone} and~\eqref{eq:exponential-form-log-det-cone-dual} that \( \Kld^{*} \) is a scaled and rotated version of \( \Kld  \).

For convenience, we further define
\begin{align}\label{eq:def-parts-cone-dual}
  \Kldi & := \left\{ (x, y, Z) \in \R \times \R_{++} \times \pdc : x \leq y \ldt (Z / y) \right\};           \nonumber\\[0.15cm]
  \Kldie & := \left\{ (x, y, Z) \in \R \times \R_{++} \times \pdc : x = y \ldt (Z / y) \right\};             \nonumber\\[0.15cm]
  \Kldii & := \R_{-} \times \{0\} \times \psdc;                                                               \nonumber\\[0.15cm]
  \Kldsi & := \left\{ (x, y, Z) \in \R_{--} \times \R \times \pdc : y \geq x (\ldt (-Z / x) + d) \right\}; \nonumber\\[0.15cm]
  \Kldsie & := \left\{ (x, y, Z) \in \R_{--} \times \R \times \pdc : y = x (\ldt (-Z / x) + d) \right\};   \nonumber\\[0.15cm]
  \Kldsii & := \{0\} \times \R_{+} \times \psdc.
\end{align}
With that, we have \begin{equation}\label{eq:kldboundary}
\partial \Kld = \Kldie \cup \Kldii \end{equation} and \[ \partial \Kld^{*} = \Kldsie \cup \Kldsii. \]

Before moving on, we present several inequalities, which will be useful for our subsequent analysis.
\begin{enumerate}[1.]
	\item Let \( \eta > 0 \), and let \( \bm{x} = (\bm{x}_y \ldt (\bm{x}_Z / \bm{x}_y), \bm{x}_y, \bm{x}_Z) \in \Kldie \cap B(\eta) \) with \( \bm{x}_y > 0 \) and \( \bm{x}_Z \succ 0 \) and satisfy \( \bm{x}_y \ldt (\bm{x}_Z / \bm{x}_y) \geq 0 \).
	      Then, we have
	      \begin{equation}\label{eq:vky-ldt-vkz/vky-upper-bound-by-eta}
		      0 \leq \bm{x}_y \ldt (\bm{x}_Z / \bm{x}_y) \leq \bm{x}_y \ldt (\eta I_d / \bm{x}_y) \leq d\bm{x}_y |\log (\eta )| - d\bm{x}_y \log (\bm{x}_y).
	      \end{equation}
	\item Let \( \alpha > 0 \) and \( s > 0 \). The following inequalities hold for all sufficiently small \( t > 0 \),
	      \begin{equation}\label{eq:relationships-t-a-tlogt-1/logt}
		      t \leq \sqrt{t}, \quad -t^{\alpha } \log (t) \leq t^{\alpha/2},  \quad t^{\alpha } \leq -\frac{1}{\log (t)},  \quad -t^{\alpha } \log (t) \leq -\frac{1}{\log (s t)}.
	      \end{equation}
\end{enumerate}

\subsection{Facial structure}\label{subsec:facial-structure}
In general, we are more interested in nontrivial faces, especially nontrivial exposed faces.
Recall that if there exists \( \bm{n} := ( \bm{n}_x, \bm{n}_y, \bm{n}_{Z} ) \in \partial\Kld^{*} \setminus \{\bm{0}\}\) such that \( \mathcal{F} = \Kld \cap \{\bm{n}\}^{\perp } \), then \( \mathcal{F} \) is a nontrivial exposed face of \( \Kld \).
Different nonzero \( \bm{n} \)'s along $\partial\Kld^{*}$ will induce different nontrivial exposed faces.

The next proposition completely characterizes the facial structure of the log-determinant cone.
\begin{proposition}[Facial structure of \( \Kld \)]\label{prop:facial-structure-log-det-cone}
	All nontrivial faces of the log-determinant cone can be classified into the following types:
%\footnote{With a slight abuse of notation, we will use the symbols $\Fr$, $\Fs$ and $\Fne$ to denote both the class of faces and a specific face, depending on the context.}
	\begin{enumerate}[(a)]
		\item infinitely many 1-dimensional faces exposed by \( \bm{n} = (\bm{n}_x, \bm{n}_x (\ldt (-\bm{n}_Z / \bm{n}_x) + d ), \bm{n}_Z) \) with \( \bm{n}_x < 0, \bm{n}_Z \succ 0 \),
		      \begin{equation}\label{eq:1d-face-F-r}
			      \Fr := \left\{ (y \ldt (-\bm{n}_x \bm{n}_Z^{-1}), y, -y\bm{n}_x \bm{n}_Z^{-1}) : y \in \R_+ \right\} = \{y \bm{f}_{{\rm r}}: y \in \R_{+}\},
		      \end{equation}
		      where
		      \begin{equation}
			      \label{eq:def-f-r}
			      \bm{f}_{{\rm r}} = (\ldt ( -\bm{n}_x \bm{n}_Z ^{-1} ), 1, -\bm{n}_x \bm{n}_Z^{-1}).
		      \end{equation}
		\item a single \( ({\rm sd}(d) + 1) \)-dimensional exposed face exposed by \( \bm{n} = (0, \bm{n}_y, \bm{0}) \) with \( \bm{n}_y > 0 \),
		      \begin{equation}\label{eq:d-dim-face-F-d}
			      \Fd :=  \R_{-} \times \{0\} \times \psdc = \Kldii.
		      \end{equation}
		\item infinitely many \( ({\rm sd}(d - \rank(\bm{n}_Z)) + 1) \)-dimensional exposed faces given by
		      \begin{equation}\label{eq:n-dim-faces-F-n}
			      \Fs :=  \R_{-} \times \{0\} \times (\psdc \cap \{\bm{n}_Z\}^{\perp }),
		      \end{equation}
		      which are exposed by
		      \begin{equation}\label{eq:def-n-F-3}
			      \bm{n} = (0, \bm{n}_y, \bm{n}_Z) \text{ with } \bm{n}_y \geq 0, \bm{n}_Z \succeq 0, 0 < \rank(\bm{n}_Z) < d.
		      \end{equation}
		\item a single \( 1 \)-dimensional exposed face exposed by
		      \[ \bm{n} = (0, \bm{n}_y, \bm{n}_Z) \text{ with } \bm{n}_y \geq 0, \bm{n}_Z \succ 0, \]
		      that is, \( \rank(\bm{n}_Z) = d \),
		      \begin{equation}\label{eq:1-dim-face-F-inf}
			      \Ff :=  \R_{-} \times \{0\} \times \{\bm{0}\} .
		      \end{equation}
		\item infinitely many  non-exposed faces defined by
		      \begin{equation}\label{eq:ne-faces}
			      \Fne := \{0\} \times \{0\} \times (\psdc \cap \{\bm{n}_Z\}^{\perp }) ,
		      \end{equation}
		      which are proper subfaces of exposed faces of the form \( \Fs \) or \( \Fd \) (see~\eqref{eq:d-dim-face-F-d} and \eqref{eq:n-dim-faces-F-n}), and \( \bm{n}_Z \) comes from the \( \bm{n} \) that exposes \( \Fs \) or \( \Fd \), i.e., \( 0 \leq \rank(\bm{n}_Z) < d \).
	\end{enumerate}
\end{proposition}

\begin{proof}
	Let \( \bm{n} := (\bm{n}_x, \bm{n}_y, \bm{n}_Z) \in \Kld^{*} \) be such that \( \{\bm{n}\}^{\perp } \cap \Kld \) is a nontrivial face of \( \Kld \).
	Recall that \( \Kld \) is pointed, so \( \bm{n} \in \partial \Kld^{*} \setminus \{\bm{0}\} \).
	By~\eqref{eq:def-parts-cone-dual}, \( \bm{n}_x \leq 0 \) and we can determine whether \( \bm{n} \in \Kldsi \) or \( \bm{n} \in \Kldsii \) by checking whether \( \bm{n}_x < 0 \) or not. %, which will lead to different properties of \( \bm{n} \).
	Therefore, we shall consider the following cases.

	\noindent\uline{\( \bm{n}_x < 0 \):} \( \bm{n}_x < 0 \) indicates that \( \bm{n} \in \Kldsie \), then we must have
	\[ \bm{n} = (\bm{n}_x, \bm{n}_x (\ldt (-\bm{n}_Z / \bm{n}_x) + d ), \bm{n}_Z ) \text{ with } \bm{n}_x < 0,\, \bm{n}_Z \succ 0. \]
	For any \( \bm{q} := ( \bm{q}_x, \bm{q}_y, \bm{q}_Z ) \in \partial \Kld \), since \( \bm{n}_x < 0 \), we can see that \( \bm{q} \in \{\bm{n}\}^{\perp } \) if and only if
	\begin{equation}
    \label{eq:for-facial-stru-nx<0}
    \bm{q}_x + \bm{q}_y ( \ldt (-\bm{n}_Z / \bm{n}_x) + d ) + \tr(\bm{n}_Z \bm{q}_Z) / \bm{n}_x  = 0.
  \end{equation}

	If \( \bm{q}_y = 0 \), then \( \bm{q} \in \Kldii \), and so \( \bm{q}_x \leq 0, \bm{q}_Z \succeq 0 \).
	This together with \( \bm{n}_Z \succ 0 \) and \eqref{eqn:traceproperties} imply that \( \tr (\bm{n}_Z \bm{q}_Z) \geq 0 \).
	Since \( \bm{n}_x < 0 \), we observe that
	\[ 0 \leq -\bm{q}_x =  \tr(\bm{n}_Z \bm{q}_Z) / \bm{n}_x \leq 0. \]
	Thus, \( \bm{q}_x = 0 \) and \( \tr(\bm{n}_Z \bm{q}_Z) = 0 \).
  The latter relation leads to \( \bm{q}_Z = \bm{0} \).
  Consequently, \( \bm{q} = \bm{0} \).

	If \( \bm{q}_y \neq 0 \), then \( \bm{q}_y > 0 \) by the definition of the log-determinant cone and hence \( \bm{q} \in \Kldie \).
	Then, we know that \( \bm{q}_x = \bm{q}_y \ldt (\bm{q}_Z / \bm{q}_y),\, \bm{q}_Z \succ 0 \) and hence \eqref{eq:for-facial-stru-nx<0} becomes
	\[ \ldt \left(\frac{\bm{q}_Z}{\bm{q}_y}\right) + \ldt \left(-\frac{\bm{n}_Z}{\bm{n}_x}\right) + d + \tr\left(\frac{\bm{n}_Z \bm{q}_Z}{\bm{n}_x \bm{q}_y}\right) = 0. \]
	After rearranging terms, we have
	\begin{equation}
		\label{eq:orthogonal-qy>0-single-terms}
		\ldt \left(-\frac{\bm{n}_Z \bm{q}_Z}{\bm{n}_x \bm{q}_y}\right) + d + \tr\left(\frac{\bm{n}_Z \bm{q}_Z}{\bm{n}_x \bm{q}_y}\right) = 0.
	\end{equation}
	Note also that
	\[ \det \left(-\frac{\bm{n}_Z \bm{q}_Z}{\bm{n}_x \bm{q}_y}\right) =  \det \left(-\frac{\bm{n}_Z^{\frac{1}{2}} \bm{q}_Z \bm{n}_Z^{\frac{1}{2}} }{\bm{n}_x \bm{q}_y}\right) \quad \text{and} \quad \tr\left(\frac{\bm{n}_Z \bm{q}_Z}{\bm{n}_x \bm{q}_y}\right) = \tr\left(\frac{\bm{n}_Z^{\frac{1}{2}}\bm{q}_Z \bm{n}_Z^{\frac{1}{2}}}{\bm{n}_x \bm{q}_y}\right), \]
	where \( \bm{n}_Z^{\frac{1}{2}}\bm{q}_Z \bm{n}_Z^{\frac{1}{2}} \succ 0 \) and \( -\frac{\bm{n}_Z^{\frac{1}{2}}\bm{q}_Z \bm{n}_Z^{\frac{1}{2}}}{\bm{n}_x \bm{q}_y} \succ 0 \).

	Let \( f(x) = \log (x) - x + 1 \), we can rewrite~(\ref{eq:orthogonal-qy>0-single-terms}) as follows,
	\begin{equation}
		\label{eq:orthogonal-qy>0-sum}
    \begin{split}
      & \sum_{i=1}^d f\left(\lambda_i\left(-\frac{\bm{n}_Z^{\frac{1}{2}} \bm{q}_Z \bm{n}_Z^{\frac{1}{2}}}{\bm{n}_x\bm{q}_y}\right)\right) \\
      = & \sum_{i=1}^d \left(\log \left(\lambda_i\left(-\frac{\bm{n}_Z^{\frac{1}{2}} \bm{q}_Z \bm{n}_Z^{\frac{1}{2}}}{\bm{n}_x\bm{q}_y}\right)\right) + 1 - \lambda_i\left(-\frac{\bm{n}_Z^{\frac{1}{2}} \bm{q}_Z \bm{n}_Z^{\frac{1}{2}}}{\bm{n}_x\bm{q}_y}\right) \right) = 0.
    \end{split}
	\end{equation}
	Since \( f(x) \leq 0 \) for all \( x > 0 \) and \( f(x) = 0 \) if and only if \( x = 1 \), (\ref{eq:orthogonal-qy>0-sum}) holds if and only if
	\[ \lambda_i\left(-\frac{\bm{n}_Z^{\frac{1}{2}} \bm{q}_Z \bm{n}_Z^{\frac{1}{2}}}{\bm{n}_x\bm{q}_y}\right) = 1 \quad \forall i \in \{1, 2, \dots , d\}. \]
	This illustrates that all the eigenvalues of \( -\frac{\bm{n}_Z^{\frac{1}{2}} \bm{q}_Z \bm{n}_Z^{\frac{1}{2}} }{\bm{n}_x \bm{q}_y} \) are \( 1 \).
	Hence, one can immediately see \( \bm{n}_Z^{\frac{1}{2}} \bm{q}_Z \bm{n}_Z^{\frac{1}{2}} = -\bm{n}_x\bm{q}_y I_d \) and so \( \bm{q}_Z = -\bm{q}_y \bm{n}_x \bm{n}_Z^{-1} \).
	By substituting this expression of \( \bm{q}_Z \) into \( \bm{q} = (\bm{q}_y \ldt (\bm{q}_Z / \bm{q}_y), \bm{q}_y, \bm{q}_Z) \), we obtain~(\ref{eq:1d-face-F-r}).

	\uline{\( \bm{n}_x = 0 \):} \( \bm{n}_x = 0 \) indicates that \( \bm{n} \in \Kldsii \), then \( \bm{n}_y \geq 0 \) and \( \bm{n}_Z \succeq 0 \).
	Now, for any \( \bm{q} \in \partial \Kld \), we have \( \bm{q} \in \{\bm{n}\}^\perp \) if and only if
	\begin{equation}\label{eq:orthogonal-nx=0}
		\bm{n}_y \bm{q}_y + \tr(\bm{n}_Z \bm{q}_Z) = 0.
	\end{equation}
	Since \( \bm{n}_y \geq 0, \bm{q}_y \geq 0, \bm{n}_Z \succeq 0 \) and \( \bm{q}_Z \succeq 0 \), we observe that both summands on the left hand side of~(\ref{eq:orthogonal-nx=0}) are nonnegative.
	Therefore,~(\ref{eq:orthogonal-nx=0}) holds if and only if
	\begin{equation}\label{eq:orthogonal-complementary}
		\bm{n}_y \bm{q}_y = 0, \quad \tr(\bm{n}_Z \bm{q}_Z) = 0.
	\end{equation}
	These together with \eqref{eqn:traceproperties} make it clear the cases we need to consider.

	Specifically, if \( \bm{n}_x = 0 \), we consider the following four cases.
	\begin{enumerate}
		\item If \( \rank(\bm{n}_Z) = 0 \) and \( \bm{n}_y = 0 \), then \( \bm{n} = \bm{0} \), which contradicts our assumption.
		      This case is hence impossible.
		\item If \( \rank(\bm{n}_Z) = 0 \) and \( \bm{n}_y > 0 \), then by~(\ref{eq:orthogonal-complementary}), \( \bm{q}_y = 0 \).
          This corresponds to \eqref{eq:d-dim-face-F-d}.
		\item If \( 0 < \rank(\bm{n}_Z) < d \), then \( \bm{q}_Z \succeq 0 \) but \( \bm{q}_Z \) is not definite, so $(\bm{q}_x,\bm{q}_y,\bm{q}_Z) \in \Kldii$. Since \( \bm{q}_Z \in \{\bm{n}_Z\}^{\perp } \) holds, this corresponds to \eqref{eq:n-dim-faces-F-n}.
		\item If \( \rank(\bm{n}_Z) = d \), i.e., \( \bm{n}_Z \succ 0 \), then \( \bm{q}_Z = \bm{0} \).
          This corresponds to \eqref{eq:1-dim-face-F-inf}.
	\end{enumerate}
	Therefore, we obtain the exposed faces defined as in~(\ref{eq:d-dim-face-F-d}),~(\ref{eq:n-dim-faces-F-n}) and~(\ref{eq:1-dim-face-F-inf}).

	We now show that all nontrivial faces of \( \Kld \) were accounted for~\eqref{eq:1d-face-F-r},~\eqref{eq:d-dim-face-F-d},~\eqref{eq:n-dim-faces-F-n},~\eqref{eq:1-dim-face-F-inf} and~\eqref{eq:ne-faces}.
	First of all, by the previous discussion, all nontrivial exposed faces must be among the ones in \eqref{eq:1d-face-F-r}, \eqref{eq:d-dim-face-F-d}, \eqref{eq:n-dim-faces-F-n}, and \eqref{eq:1-dim-face-F-inf}.
	Suppose \( \mathcal{F} \) is a non-exposed face of \( \Kld  \).
	Then it must be contained in a nontrivial exposed face $\widehat\stdFace$ of \( \Kld  \), e.g., \cite[Proposition~3.6]{BW81} or \cite[Proposition~2.1]{LRS20}.
	The faces in \eqref{eq:1d-face-F-r} and \eqref{eq:1-dim-face-F-inf} are one-dimensional, so the only candidates for $\widehat\stdFace$ are the faces as in
	\eqref{eq:d-dim-face-F-d} and~\eqref{eq:n-dim-faces-F-n}.
	
	So suppose that $\widehat\stdFace$ is as in ~\eqref{eq:d-dim-face-F-d} or ~\eqref{eq:n-dim-faces-F-n}. Recalling the list of nontrivial exposed faces described so far, the only
	nontrival faces of $\widehat\stdFace$ that have not appeared yet
	are the ones of the form \( \Fne  \) (as in \eqref{eq:ne-faces}) for some \( \bm{n}_Z \) with \( 0 \leq \rank(\bm{n}_Z) < d \).
	This shows the completeness of the classification.
\end{proof}

It is worth noting that when \( d = 1 \), the case corresponding to \( \Fs \) does not occur.
We also have the following relationships between these nontrivial faces.
Let \( \bm{n} \neq \bm{0} \) with \( 0 \leq \rank(\bm{n}_Z) < d \) be given.
If \( \rank(\bm{n}_Z) > 0 \), then the corresponding faces \( \Fs \) and \( \Fne\) satisfy the following inclusion
\begin{equation}\label{eq:relationships-faces}
	\Fne \properideal \Fs \properideal \Fd \quad \text{and} \quad \Ff \properideal \Fs \properideal \Fd.
\end{equation}
If \( \rank(\bm{n}_Z) = 0 \), then we have
\begin{equation}\label{eq:relationships-faces-2}
	\Fne \properideal \Fd.
\end{equation}
For distinct \( \bm{n}^1 := (\bm{n}_x^1, \bm{n}_y^1, \bm{n}_Z^1) \) and \( \bm{n}^2 := (\bm{n}_x^2, \bm{n}_y^2, \bm{n}_Z^2) \) with \( 0 < \rank(\bm{n}_Z^1) < d \) and \( 0 < \rank(\bm{n}_Z^2) < d \), suppose \( \bm{n}^1 \) and \( \bm{n}^2 \) expose \( \Fs^1 \) and \( \Fs^2 \), respectively.
If \( \text{range}( \bm{n}_Z^1 ) \supsetneq \text{range}( \bm{n}_Z^2 ) \), then \( \Fs^1 \properideal \Fs^2 \) (see, e.g., \cite[Section~6]{BC75}).
% This property is true because if \( \text{range}( \bm{n}_Z^1 ) \supsetneq \text{range}( \bm{n}_Z^2 ) \), then \( \text{null}( \bm{n}_Z^1 ) \subsetneq \text{null}( \bm{n}_Z^2 ) \) and \( \psdc \cap \{\bm{n}_Z^1\}^{\perp } \subsetneq \psdc \cap \{\bm{n}_Z^2\}^{\perp } \).
A similar result also holds for non-exposed faces, that is, denote the non-exposed faces by \( \mathcal{F}_{{\rm ne}}^{\#1}  \) and \( \mathcal{F}_{{\rm ne}}^{\#2} \), respectively, with respect to \( \bm{n}^1 \) and \( \bm{n}^2 \), if \( \text{range}(\bm{n}_Z^1) \supsetneq \text{range}(\bm{n}_Z^2) \), then \( \mathcal{F}_{{\rm ne}}^{\#1} \properideal \mathcal{F}_{{\rm ne}}^{\#2} \).

\subsection{One-step facial residual functions}\label{subsec:facial-residual-functions}
In this subsection, we shall apply the strategy in \cite[Section 3.1]{LiLoPo20} to compute the corresponding one-step facial residual functions for nontrivial exposed faces of the log-determinant cone.
Put concretely, consider \( \mathcal{F} = \Kld \cap \{\bm{n}\}^{\perp} \) with \( \bm{n} \in \partial \Kld^{*} \setminus \{\bm{0}\} \).
For \( \eta > 0 \) and some nondecreasing function \( \mathfrak{g} : \R_+ \to \R_+ \) with \( \mathfrak{g}(0) = 0 \) and \( \mathfrak{g} \geq | \cdot |^{\alpha} \) for some \( \alpha \in (0, 1] \), we define
\begin{equation}
	\label{eq:def-gamma}
	\!\gamma_{\bm{n},\eta} \!:=\! \inf_{\bm{v}}\left\{\frac{\mathfrak{g}(\|\bm{v} - \bm{w}\|)}{\|\bm{u} - \bm{w}\|}\,\bigg| \,
	\begin{array}{c}
		\bm{v}\in \partial \Kld \cap B(\eta)\setminus \mathcal{F},\,\bm{w} = P_{\{\bm{n}\}^\perp}(\bm{v}), \\
		\bm{u} = P_{\mathcal{F}}(\bm{w}),\,\bm{u} \neq \bm{w}
	\end{array}\right\}.
\end{equation}
In view of \cite[Theorem 3.10]{LiLoPo20} and \cite[Lemma 3.9]{LiLoPo20}, if \( \gamma_{\bm{n}, \eta } \in (0,\infty]\) then we can use $\gamma_{\bm{n}, \eta }$ and $\mathfrak{g}$ to
construct a one-step facial residual function for $\Kld$ and $\bm{n}$.
In \cite{LiLoPo20}, the positivity of \( \gamma_{\bm{n}, \eta } \) (with the exponential cone in place of \( \Kld \) and some properly selected \( \mathfrak{g} \)) was shown by contradiction.
Here, we will follow a similar strategy and make extensive use of the following fact from \cite[Lemma 3.12]{LiLoPo20}: if \( \gamma_{\bm{n}, \eta } = 0 \), then there exist \( \widehat{\bm{v}}\in \mathcal{F} \) and a sequence \( \{\bm{v}^k\}\subset \partial \Kld \cap B(\eta)\setminus \mathcal{F} \) such that
\begin{equation}\label{eq:for-contradiction}
	\lim_{k\to \infty}\bm{v}^k = \lim_{k\to \infty}\bm{w}^k = \widehat{\bm{v}}\,\,{\rm and}\,\,\lim_{k\to\infty}\frac{\mathfrak{g}(\|\bm{w}^k - \bm{v}^k\|)}{\|\bm{u}^k - \bm{w}^k\|} = 0,
\end{equation}
where \( \bm{w}^k = P_{\{\bm{n}\}^\perp}(\bm{v}^k) \), \( \bm{u}^k = P_{\mathcal{F}}(\bm{w}^k) \) and \( \bm{u}^k\neq \bm{w}^k \).

\subsubsection{\( \Fd  \): the unique \( ({\rm sd}(d)+1) \)-dimensional faces}\label{subsubsec:FRFs-F-d}
We define the piecewise modified Boltzmann-Shannon entropy \( \mathfrak{g}_{{\rm d}} :\R_{+} \to \R_{+} \) as follows:
\begin{equation}\label{eq:BS-entropy}
	\mathfrak{g}_{{\rm d}}(t) := \left\{
	\begin{array}{ll}
		0                 & \text{ if } t = 0,                    \\
		-t \log (t)       & \text{ if } 0 < t \leq \frac{1}{e^2}, \\
		t + \frac{1}{e^2} & \text{ if } t > \frac{1}{e^2}.
	\end{array} \right.
\end{equation}
Note that \( \mathfrak{g}_{{\rm d}} \) is nondecreasing with \( \mathfrak{g}_{{\rm d}}(0) = 0 \) and \( | t | \leq \mathfrak{g}_{{\rm d}} (t) \) for any \( t \in \R_+ \).
% and there exists \( L \geq 1 \) such that the following inequalities hold for every \( t \in \R_+ \) and \( M > 0 \):
% \begin{equation}
% 	\label{eq:g-d-L-M}
% 	| t |  \leq \mathfrak{g}_{{\rm d}} (t), \,\,\mathfrak{g}_{{\rm d}} (2t) \leq L \mathfrak{g}_{{\rm d}} (t), \,\,\mathfrak{g}_{{\rm d}} (Mt) \leq L^{1+|\log _2(M)|} \mathfrak{g}_{{\rm d}} (t).
% \end{equation}

The next theorem shows that \( \gamma _{\bm{n}, \eta } \in (0, \infty] \) for \( \Fd  \), which implies that an entropic error bound holds.

\begin{theorem}[Entropic error bound concerning \( \Fd  \)]\label{thm:entropic-error-bound-F-d}
	Let \( \bm{n} = (0, \bm{n}_y, \bm{0}) \in \partial \Kld^{*} \) with \( \bm{n}_y > 0 \) such that \( \Fd = \Kld \cap \{\bm{n}\}^{\perp } \).
	Let \( \eta > 0 \) and let \( \gamma _{\bm{n}, \eta } \) be defined as in \eqref{eq:def-gamma} with \( \mathcal{F} = \Fd \) and \( \mathfrak{g} = \mathfrak{g}_{{\rm d}} \).
	Then \( \gamma _{\bm{n}, \eta } \in (0, \infty] \) and
	\begin{equation}\label{eq:entropic-error-bound-F-d}
		\dist(\bm{q}, \Fd) \leq \max \{2, 2\gamma _{\bm{n}, \eta }^{-1}\}\cdot \mathfrak{g}_{{\rm d}}(\dist(\bm{q}, \Kld)) \quad\quad \forall \bm{q} \in \{\bm{n}\}^{\perp } \cap B(\eta).
	\end{equation}
\end{theorem}

\begin{proof}
	If \( \gamma _{\bm{n}, \eta } = 0 \), in view of~\cite[Lemma 3.12]{LiLoPo20}, there exist \( \widehat{\bm{v}}\in \Fd \) and a sequence \( \{\bm{v}^k\}\subset \partial \Kld \cap B(\eta)\setminus \Fd \) such that \eqref{eq:for-contradiction} holds with \( \mathfrak{g} = \mathfrak{g}_{{\rm d}} \) and $\stdFace = \Fd$.

	By~\eqref{eq:d-dim-face-F-d}, \( \widehat{\bm{v}} = (\widehat{\bm{v}}_x, 0, \widehat{\bm{v}}_Z)  \) with \( \widehat{\bm{v}}_Z \succeq 0 \).
	Since \( \bm{v}^k \in \partial \Kld \cap B(\eta) \setminus \Fd \) for all \( k \), we have \( \bm{v}_y^k > 0 \) and \( \bm{v}^k \in \Kldie \) for all \( k \).
	Hence, \( \bm{v}^k = (\bm{v}_y^k \ldt (\bm{v}_Z^k / \bm{v}_y^k ), \bm{v}_y^k, \bm{v}_Z^k) \text{ with } \bm{v}_y^k > 0, \bm{v}_Z^k \succ 0 \) for all \( k \).

	Recall that \( \bm{n} = (0, \bm{n}_y, \bm{0}) \) with \( \bm{n}_y > 0 \), then \( \| \bm{n} \| = \bm{n}_y  \text{ and }  \langle \bm{n}, \bm{v}^k \rangle = \bm{n}_y \bm{v}_y^k > 0. \)
	Since \( \bm{w}^k = P_{\{\bm{n}\}^{\perp } } (\bm{v}^k) \) and \( \{\bm{n}\}^{\perp } \) is a hyperplane, one can immediately see that for all \( k \),
	\[ \bm{w}^k = \bm{v}^k - \frac{\langle \bm{n}, \bm{v}^k \rangle }{\| \bm{n} \|^2 } \bm{n} = (\bm{v}_y^k \ldt (\bm{v}_Z^k / \bm{v}_y^k), 0, \bm{v}_Z^k) \quad \text{and} \quad \| \bm{w}^k - \bm{v}^k \| = \frac{|\langle \bm{n}, \bm{v}^k \rangle |}{\| \bm{n} \| } = \bm{v}_y^k.  \]
	Using~\eqref{eq:d-dim-face-F-d}, \( \bm{u}^k = P_{\Fd} (\bm{w}^k) \) and \( \bm{u}^k \neq \bm{w}^k \), we  see that \( \bm{v}_y^k \ldt (\bm{v}_Z^k / \bm{v}_y^k) > 0 \) and \( \bm{u}^k = (0, 0, \bm{v}_Z^k) \).
	We thus obtain that for all \( k \),
	\[ \| \bm{w}^k - \bm{u}^k \| = \bm{v}_y^k \ldt (\bm{v}_Z^k / \bm{v}_y^k). \]
	Because \( \lim_{k \to \infty} \bm{v}_y^k = 0 \), for sufficiently large \( k \), we have \( 0 < \bm{v}_y^k < \frac{1}{e^2} \).
	Hence,
	\begin{align*}
		\lim_{k \to \infty} \frac{\mathfrak{g}_{{\rm d}}(\| \bm{w}^k - \bm{v}^k \| )}{\| \bm{w}^k - \bm{u}^k \| } \overset{(\text{a})}{\geq} \lim_{k \to \infty} \frac{-\bm{v}_y^k \log (\bm{v}_y^k) }{d\bm{v}_y^k |\log (\eta )| - d\bm{v}_y^k \log (\bm{v}_y^k)} = \lim_{k \to \infty} \frac{1}{d - d \frac{|\log(\eta)|}{\log (\bm{v}_y^k)}} = \frac{1}{d} > 0,
	\end{align*}
	where (a) comes from the fact \( \bm{v}^k \in B(\eta) \) and~\eqref{eq:vky-ldt-vkz/vky-upper-bound-by-eta}.
	This contradicts~\eqref{eq:for-contradiction} with \( \mathfrak{g}_{{\rm d}} \) in place of \( \mathfrak{g} \) and hence this case cannot happen.
	Therefore, we  conclude that \( \gamma _{\bm{n}, \eta } \in (0,\infty] \), with which and~\cite[Theorem 3.10]{LiLoPo20},~\eqref{eq:entropic-error-bound-F-d} holds.
\end{proof}

\begin{remark}[Tightness of~\eqref{eq:entropic-error-bound-F-d}]\label{remark:tightness-F-d}
	We claim that for \( \Fd \), there is a specific choice of sequence \( \{\bm{w}^k\} \) in \( \{\bm{n}\}^{\perp} \) with \( \dist (\bm{w}^k, \Kld) \to 0 \) along which both sides of~\eqref{eq:entropic-error-bound-F-d} vanish at the same order of magnitude.
	Recall that we assumed that \( d > 1 \); see the discussions following \eqref{eq:exponential-form-log-det-cone-dual}.\footnote{When \( d = 1 \), the log-determinant cone reduces to the exponential cone studied in \cite{LiLoPo20}, where the tightness of the corresponding error bounds was shown in Remark 4.14 therein.}
	Let \( \bm{n} = (0, \bm{n}_y, \bm{0}) \) with \( \bm{n}_y > 0 \) so that \( \{\bm{n}\}^{\perp} \cap \Kld = \Fd \).
    Define \( \bm{w}^k = (d\log (k) / k, 0, I_d) \) for every $k\in \mathbb{N}$. %where \( I_d \) is the identity matrix.
	Then \( \{\bm{w}^k\} \subseteq \{\bm{n}\}^{\perp} \).
	Since \( \log (k) / k > 0 \) for any \( k \geq 2 \) and \( \log (k) / k \to 0 \) as \( k \to \infty \), there exists \( \eta > 0 \) such that \( \left\{ \bm{w}^k \right\} \subseteq B(\eta) \).
	Thus, applying~\eqref{eq:entropic-error-bound-F-d}, there exists \( \kappa _B > 0\) such that
	\[ \dist\left(\bm{w}^k, \Fd\right) \leq \kappa _B \mathfrak{g}_{{\rm d}} \left( \dist\left(\bm{w}^k, \Kld\right) \right) \quad \text{for all sufficiently large } k. \]
	Noticing that the projection of \( \bm{w}^k \) onto \( \Fd \) (see \eqref{eq:d-dim-face-F-d}) is given by \( (0, 0, I_d) \), we obtain
	\[ \frac{d\log (k)}{k} = \dist(\bm{w}^k, \Fd) \leq \kappa _B \mathfrak{g}_{{\rm d}} \left( \dist(\bm{w}^k, \Kld) \right).  \]
	Let \( \bm{v}^k = (d\log(k) / k, 1 / k, I_d) \) for every \( k \).
	Then \( \dist(\bm{w}^k, \Kld) \leq 1 / k \) since \( \bm{v}^k \in \Kld\).
	In view of the definition of \( \mathfrak{g}_{{\rm d}} \) (see~\eqref{eq:BS-entropy}) and its monotonicity, we conclude that for large enough \( k \) we have
	\[ \frac{d\log (k)}{k} = \dist(\bm{w}^k, \Fd) \leq \kappa _B \mathfrak{g}_{{\rm d}} (\dist(\bm{w}^k, \Kld)) \leq \kappa_B \frac{\log (k)}{k}. \]
	That means it holds that for all sufficiently large \( k \),
	\[ d \leq \frac{\dist(\bm{w}^k, \Fd)}{\mathfrak{g}_{{\rm d}} (\dist(\bm{w}^k, \Kld))} \leq \kappa _B. \]
	
	Consequently, for any given nonnegative function \( \mathfrak{g}: \R_+ \to \R_+ \) such that \( \lim_{t \downarrow 0} \frac{\mathfrak{g}(t)}{\mathfrak{g}_{{\rm d}} (t)} = 0 \), we have upon noting \( \dist(\bm{w}^k, \Kld) \to 0 \) that
	\[ \frac{\dist(\bm{w}^k, \Fd)}{\mathfrak{g}(\dist(\bm{w}^k, \Kld))} = \frac{\dist(\bm{w}^k, \Fd)}{\mathfrak{g}_{{\rm d}} (\dist(\bm{w}^k, \Kld))} \frac{\mathfrak{g}_{{\rm d}} (\dist(\bm{w}^k, \Kld))}{\mathfrak{g}(\dist(\bm{w}^k, \Kld))} \to \infty , \]
	which shows that the choice of \( \mathfrak{g}_{{\rm d}} \) in~\eqref{eq:entropic-error-bound-F-d} is tight.
\end{remark}

Upon invoking Theorem~\ref{thm:entropic-error-bound-F-d} and \cite[Lemma 3.9]{LiLoPo20}, we  obtain the following one-step facial residual function for \( \Kld \) and \( \bm{n} \).

\begin{corollary}\label{corollary:FRFs-F-d}
	Let \( \bm{n} = (0, \bm{n}_y, \bm{0}) \in \partial \Kld^{*} \) with \( \bm{n}_y > 0 \) such that \( \Fd = \Kld \cap \{\bm{n}\}^{\perp }  \).
	Let \( \gamma _{\bm{n}, t } \) be defined as in~\eqref{eq:def-gamma} with $\stdFace = \Fd$ and \( \mathfrak{g} = \mathfrak{g}_{{\rm d}} \) in~\eqref{eq:BS-entropy}.
	Then the function \( \psi _{\mathcal{K}, \bm{n}} : \R_{+} \times \R_{+} \to \R_{+} \) defined by
	\[ \psi _{\mathcal{K}, \bm{n}} (\epsilon, t) := \max \left\{\epsilon, \epsilon / \| \bm{n} \| \right\} + \max \left\{2, 2\gamma _{\bm{n}, t}^{-1}\right\} \mathfrak{g}_{{\rm d}}\left(\epsilon  + \max\left\{\epsilon , \epsilon / \| \bm{n} \| \right\}\right) \]
	is a one-step facial residual function for \( \Kld  \) and \( \bm{n} \).
\end{corollary}

\subsubsection{\( \Fs \): the family of \( \texorpdfstring{({\rm sd}(d - \rank(\bm{n}_Z)) + 1) }{d-rank(n_Z)+1}\)-dimensional faces}\label{subsubsec:FRFs-F-3}

Let \( \eta > 0 \) and let \( \bm{n} \in \partial \Kld^{*} \) be such that \( \Fs = \Kld \cap \{\bm{n}\}^{\perp } \).
Let \( \gamma_{\bm{n}, \eta } \) be defined as in \eqref{eq:def-gamma} with \( \mathcal{F} = \Fs \) and some nondecreasing function \( \mathfrak{g} : \R_+ \to \R_+ \) with \( \mathfrak{g}(0) = 0 \) and \( \mathfrak{g} \geq | \cdot |^{\alpha} \) for some \( \alpha \in (0, 1] \).
If \( \gamma _{\bm{n}, \eta } = 0 \), in view of~\cite[Lemma 3.12]{LiLoPo20}, there exists \( \widehat{\bm{v}} \in \Fs \) and a sequence \( \{\bm{v}^k\} \subset \partial \Kld \cap B(\eta) \setminus \Fs \) such that \eqref{eq:for-contradiction} holds.
As we will see later in the proofs of Theorem~\ref{thm:Holderian-error-bound-F-3-ny>0} and Theorem~\ref{thm:log-type-error-bound-F-3} below, we will encounter the following three cases:
\begin{enumerate}[(I)]
	\item\label{item:ny>=0-vky=0} \( \bm{n}_y \geq 0 \) and \( \bm{v}^k \in \Fd \cap B(\eta) \setminus \Fs \) for all large $k$;
	\item\label{item:ny>0-vky>0} \( \bm{n}_y > 0 \) and \( \bm{v}^k \in \partial \Kld \cap B(\eta) \setminus \Fd \) infinitely often;
	\item\label{item:ny=0-vky>0} \( \bm{n}_y = 0 \) and \( \bm{v}^k \in \partial \Kld \cap B(\eta) \setminus \Fd \) infinitely often.
\end{enumerate}

For case (\ref{item:ny>=0-vky=0}), we have the following lemma which will aid in our further analysis.
One should notice that this lemma holds for both \( \Fs \) and \( \Ff \).

\begin{lemma}\label{lemma:ny>=0-vky=0-Holderian}
	Let \( \bm{n} = (0, \bm{n}_y, \bm{n}_Z) \in \partial \Kld^{*} \setminus \{\bm{0}\} \) with \( \bm{n}_y \geq 0 \) and \( \bm{n}_Z \succeq 0 \) such that \( \mathcal{F} = \Kld \cap \{\bm{n}\}^{\perp } \) with \( \mathcal{F} = \Fs \) or \( \Ff \).
	Let \( \overline{\bm{v}} \in \mathcal{F} \) be arbitrary and \( \{\bm{v}^k\} \subset \Fd \cap B(\eta) \setminus \mathcal{F} \) be such that
	\[ \lim_{k \to \infty} \bm{v}^k = \lim_{k \to \infty} \bm{w}^k = \overline{\bm{v}}, \]
	where \( \bm{w}^k = P_{\{\bm{n}\}^{\perp } } (\bm{v}^k), \bm{u}^k = P_{\mathcal{F}} (\bm{w}^k) \) and \( \bm{w}^k \neq \bm{u}^k \).
	Then
	\[ \liminf_{k \to \infty} \frac{\| \bm{w}^k - \bm{v}^k \| ^{\alpha }}{\| \bm{w}^k - \bm{u}^k \| } \in (0,\infty], \]
	where \( \alpha  \) is defined as in~\eqref{eq:alpha-psd-cone-error-bound} with \( Z \) being \( \bm{n}_Z \).
\end{lemma}

\begin{proof}
	Note that \( \{\bm{v}^k\} \subset \Fd \cap B(\eta) \setminus \mathcal{F}  \) implies \( \bm{v}^k = (\bm{v}_x^k, 0, \bm{v}_Z^k) \) with \( \bm{v}_x^k \leq 0 \) and \( \bm{v}_Z^k \in \psdc \) for all \( k \).
	Then, \( \langle \bm{n}, \bm{v}^k \rangle = \tr(\bm{v}_Z^k \bm{n}_Z)\), which is nonnegative since both \( \bm{v}_Z^k  \) and \( \bm{n}_Z  \) are positive semidefinite.
	Because \( \bm{w}^k = P_{\{\bm{n}\}^{\perp } }(\bm{v}^k) \) and \( \{\bm{n}\}^{\perp }  \) is a hyperplane, one can immediately see that for all \( k \),
	\[ \| \bm{w}^k - \bm{v}^k \| = \frac{| \langle \bm{n}, \bm{v}^k \rangle  | }{\| \bm{n} \| } = \frac{\tr(\bm{v}_Z^k \bm{n}_Z)}{\| \bm{n} \| }. \]
	On the other hand, by Lemma~\ref{lemma:norm-wk-uk} and the formula of \( \mathcal{F} \), we obtain that for all \( k \),
	\begin{align*}
		\| \bm{w}^k - \bm{u}^k \| \leq \dist(\bm{v}^k, \mathcal{F}) = \dist(\bm{v}_Z^k, \psdc \cap \{\bm{n}_Z\}^{\perp }) \leq  C_P \tr(\bm{v}_Z^k \bm{n}_Z)^{\alpha },
	\end{align*}
	where the final inequality comes from Proposition~\ref{prop:psd-error-bound} and \( \alpha  \) is defined as in~\eqref{eq:alpha-psd-cone-error-bound} with \( Z \) being \( \bm{n}_Z \).

	Now, we can conclude that
	\[ \liminf_{k \to \infty} \frac{\| \bm{w}^k - \bm{v}^k \| ^{\alpha }}{\| \bm{w}^k - \bm{u}^k \| } \geq \frac{1}{C_P \| \bm{n} \|^{\alpha }} > 0. \]
	This completes the proof.
\end{proof}

Now, we are ready to show the error bound concerning \( \Fs  \).
We first show that we have a H\"olderian error bound concerning \( \Fs  \) when \( \bm{n}_y > 0 \).
\begin{theorem}[H\"olderian error bound concerning \( \Fs  \) if \( \bm{n}_y > 0 \)]\label{thm:Holderian-error-bound-F-3-ny>0}
	Let \( \bm{n} = (0, \bm{n}_y, \bm{n}_Z) \in \partial \Kld^{*} \) with \( \bm{n}_y > 0\), \(\bm{n}_Z \succeq 0 \) and \( 0 < \rank(\bm{n}_Z) < d \) such that \( \Fs = \Kld \cap \{\bm{n}\}^{\perp } \).
	Let \( \eta > 0 \) and let \( \gamma _{\bm{n}, \eta } \) be defined as in~\eqref{eq:def-gamma} with \( \mathcal{F} = \Fs \) and \( \mathfrak{g} = | \cdot | ^{\frac{1}{2}} \).
	Then \( \gamma _{\bm{n}, \eta } \in (0, \infty] \) and
	\begin{equation}
		\label{eq:Holderian-error-bound-F-3}
		\dist(\bm{q}, \Fs) \leq \max \{2\eta ^{\frac{1}{2}}, 2\gamma _{\bm{n}, \eta }^{-1}\} \cdot (\dist(\bm{q}, \Kld))^{\frac{1}{2}} \quad\quad \forall \bm{q} \in \{\bm{n}\}^{\perp } \cap B(\eta).
	\end{equation}
\end{theorem}

\begin{proof}
	If \( \gamma _{\bm{n}, \eta } = 0 \), in view of~\cite[Lemma 3.12]{LiLoPo20}, there exist \( \widehat{\bm{v}}\in \Fs \) and a sequence \( \{\bm{v}^k\}\subset \partial \Kld \cap B(\eta)\setminus \Fs \) such that \eqref{eq:for-contradiction} holds with \( \mathfrak{g} = | \cdot |^{\frac{1}{2}} \) and $\stdFace = \Fs$.
	Since \( \{\bm{v}^k\} \subset \partial \Kld \cap B(\eta) \setminus \Fs \), the equation for the boundary of $\Kld$ (see \eqref{eq:kldboundary} and \eqref{eq:relationships-faces}) implies that we have the following two cases:
	\begin{enumerate}[(i)]
		\item\label{item:F-3-ny>0-vk-not-F-d} \( \bm{v}^k \in \partial \Kld \cap B(\eta) \setminus \Fd \) infinitely often;
		\item\label{item:F-3-ny>0-vk-F-d} \( \bm{v}^k \in \Fd \cap B(\eta) \setminus \Fs \) for all large \( k \).
	\end{enumerate}

	(\ref{item:F-3-ny>0-vk-not-F-d})
	 Passing to a subsequence if necessary, we can assume that \( \bm{v}^k \in \partial \Kld \cap B(\eta) \setminus \Fd \) for all \( k \), that is,
	\[ \bm{v}^k = (\bm{v}_y^k \ldt (\bm{v}_Z^k / \bm{v}_y^k ), \bm{v}_y^k, \bm{v}_Z^k) \text{ with }  \bm{v}_y^k > 0, \bm{v}_Z^k \succ 0, \quad \text{for all } k. \]
	Then, \( \langle \bm{n}, \bm{v}^k \rangle = \bm{n}_y \bm{v}_y^k + \tr(\bm{v}_Z^k \bm{n}_Z) \), which is positive since \( \bm{n}_y > 0, \bm{v}_y^k > 0 \) and both \( \bm{v}_Z^k , \bm{n}_Z  \) are positive semidefinite.

	Now, one can check that
	\begin{equation}
    \label{eq:F-3-case-1}
    \| \bm{w}^k - \bm{v}^k \| = \frac{\langle \bm{n}, \bm{v}^k \rangle }{\| \bm{n} \| } = \frac{\bm{n}_y \bm{v}_y^k + \tr(\bm{v}_Z^k \bm{n}_Z)}{\| \bm{n} \| } .
  \end{equation}
	On the other hand, by Lemma~\ref{lemma:norm-wk-uk}, the formula of \( \Fs \) and Proposition~\ref{prop:psd-error-bound}, we  obtain the following inequality for all \( k \),
	\begin{equation}\label{eq:up-bd-wk-uk-F3}
      \| \bm{w}^k - \bm{u}^k \| \leq \dist(\bm{v}^k, \Fs) \leq ( \bm{v}_y^k \ldt (\bm{v}_Z^k / \bm{v}_y^k) )_+ + \bm{v}_y^k + C_P \tr(\bm{v}_Z^k \bm{n}_Z)^{\frac{1}{2}}.
  \end{equation}

	Let \( \tau ^k := \tr(\bm{v}_Z^k \bm{n}_Z) \) and \( \mathsf{r} := \rank(\bm{n}_Z) \).

	If \( \bm{v}_y^k \ldt (\bm{v}_Z^k / \bm{v}_y^k) \geq 0 \) infinitely often, then by extracting a subsequence if necessary, we may assume that \( \bm{v}_y^k \ldt (\bm{v}_Z^k / \bm{v}_y^k) \geq 0 \) for all \( k \).
  Then we have from \eqref{eq:up-bd-wk-uk-F3} and \eqref{eq:vky-ldt-vkz/vky-upper-bound-by-eta} that for all large \( k \),
	\begin{align*}
		\| \bm{w}^k - \bm{u}^k \| & \leq d |\log(\eta)|\bm{v}_y^k - d\bm{v}_y^k \log (\bm{v}_y^k) + \bm{v}_y^k + C_P(\tau^k)^{\frac{1}{2}}                                                                                                        \\
		                          & \overset{(\text{a})}{\leq } (d |\log(\eta)| + 1) (\bm{v}_y^k)^{\frac{1}{2}} + d (\bm{v}_y^k)^{\frac{1}{2}} + C_P(\tau^k)^{\frac{1}{2}}                                                                                               \\
		                          & \overset{(\text{b})}{\leq } \frac{(d |\log(\eta)| + d + 1)\| \bm{n} \| ^{\frac{1}{2}}}{(\bm{n}_y)^{\frac{1}{2}}} \| \bm{w}^k - \bm{v}^k \| ^{\frac{1}{2}} + C_P \| \bm{n} \| ^{\frac{1}{2}} \| \bm{w}^k - \bm{v}^k \| ^{\frac{1}{2}} \\
		                          & = \left[\frac{(d |\log(\eta)| + d + 1)\| \bm{n} \| ^{\frac{1}{2}}}{(\bm{n}_y)^{\frac{1}{2}}} + C_P \| \bm{n} \| ^{\frac{1}{2}}\right] \| \bm{w}^k - \bm{v}^k \| ^{\frac{1}{2}},
	\end{align*}
	where (a) holds by~\eqref{eq:relationships-t-a-tlogt-1/logt}  with \( \alpha=1 \) and the fact that \( \bm{v}_y^k \to 0 \) (since ${\bm v^k} \to \widehat{\bm v}\in \Fs$), (b) is true since \( \| \bm{w}^k - \bm{v}^k \|^{\frac{1}{2}} \geq (\bm{n}_y \bm{v}_y^k)^{\frac{1}{2}} / (\| \bm{n} \| )^{\frac{1}{2}} \) and \( \| \bm{w}^k - \bm{v}^k \| ^{\frac{1}{2}} \geq (\tau^k)^{\frac{1}{2}} / (\| \bm{n} \| )^{\frac{1}{2}} \) for all \( k \) thanks to \eqref{eq:F-3-case-1}.

	% Therefore, we can conclude that
	% \[ \lim_{k \to \infty} \frac{\| \bm{w}^k - \bm{v}^k \|^{\frac{1}{2}} }{\| \bm{w}^k - \bm{u}^k \| } \geq  \left[\frac{(d \log(\eta) + d + 1)\| \bm{n} \| ^{\frac{1}{2}}}{(\bm{n}_y)^{\frac{1}{2}}} + C_P \| \bm{n} \| ^{\frac{1}{2}}\right]^{-1} > 0.\]
	This contradicts~\eqref{eq:for-contradiction} with \( | \cdot |^{\frac{1}{2}} \) in place of \( \mathfrak{g} \) and hence this case cannot happen.

	If \( \bm{v}_y^k \ldt (\bm{v}_Z^k / \bm{v}_y^k) < 0 \) infinitely often, then by extracting a subsequence if necessary, we may assume that \( \bm{v}_y^k \ldt (\bm{v}_Z^k / \bm{v}_y^k) < 0 \) for all \( k \).
  Similar to the previous analysis, we have from \eqref{eq:up-bd-wk-uk-F3}, \eqref{eq:relationships-t-a-tlogt-1/logt} and \eqref{eq:F-3-case-1} that for all large \( k \),
	\[ \| \bm{w}^k - \bm{u}^k \|  \leq  \bm{v}_y^k + C_P (\tau^k)^{\frac{1}{2}} \leq (\bm{v}_y^k)^{\frac{1}{2}} + C_P (\tau^k)^{\frac{1}{2}} \leq \left[ (\| \bm{n} \| / \bm{n}_y)^{\frac{1}{2}} + C_P \| \bm{n} \| ^{\frac{1}{2}} \right] \| \bm{w}^k - \bm{v}^k \|^{\frac{1}{2}}. \]
	% Therefore,
	% \[ \lim_{k \to \infty} \frac{\| \bm{w}^k - \bm{v}^k \| ^{\frac{1}{2}}}{ \| \bm{w}^k - \bm{u}^k \| } \geq \left[(\| \bm{n} \| / \bm{n}_y)^{\frac{1}{2}} + C_P \| \bm{n} \| ^{\frac{1}{2}}\right]^{-1} > 0. \]
	The above display contradicts~\eqref{eq:for-contradiction} with \( | \cdot |^{\frac{1}{2}} \) in place of \( \mathfrak{g} \) and hence this case cannot happen.

	(\ref{item:F-3-ny>0-vk-F-d}) By Lemma~\ref{lemma:ny>=0-vky=0-Holderian}, case (\ref{item:F-3-ny>0-vk-F-d}) also cannot happen.

	Hence, we conclude that \( \gamma _{\bm{n}, \eta } \in (0,\infty] \).
	In view of \cite[Theorem 3.10]{LiLoPo20}, we deduce that~\eqref{eq:Holderian-error-bound-F-3} holds.
\end{proof}

\begin{remark}[Tightness of~\eqref{eq:Holderian-error-bound-F-3}]\label{remark:tightness-Holderian-F-3}
  Fix any $0<{\sf r}<d$ (recall that we assumed $d\ge 2$; see the discussions following \eqref{eq:exponential-form-log-det-cone-dual}).
  Let \( \bm{n} = (0, \bm{n}_y, \bm{n}_Z) \) with \( \bm{n}_y > 0 \), \( \bm{n}_Z \succeq 0 \) and \( \rank(\bm{n}_Z) = {\sf r} \).
  Then, we have \( \Fs = \Kld \cap \{\bm{n}\}^{\perp} \) from \eqref{eq:n-dim-faces-F-n}.
  Let \( R \in \R^{d \times d} \) be such that \( \bm{n}_Z = R
  \begin{bmatrix}
    \bm{0} & \bm{0} \\
    \bm{0} & \Sigma_{\mathsf{r}}
  \end{bmatrix}R^{\top}
  \) where \( \Sigma_{\mathsf{r}}\in {\cal S}^{{\sf r}} \) is diagonal, \( \Sigma_{\mathsf{r}} \succ 0 \) and \( RR^{\top} = I_d \).
Then
  \begin{equation}\label{Fsharpexample}
   \Fs = \R_- \times \{0\} \times (\psdc \cap \{\bm{n}_Z\}^{\perp}) = \R_- \times \{0\} \times \left\{ R
     \begin{bmatrix}
       A & \bm{0} \\
       \bm{0} & \bm{0}
     \end{bmatrix}R^{\top} \, : \, A \in \mathcal{S}_+^{d-\mathsf{r}}
   \right\}.
  \end{equation}

  Fix a $Q\in \R^{{\sf r}\times(d-{\sf r})}$ with $0 < \lambda_{\max}(Q^\top Q)\le 1$. For every \( k > 0 \), we define
  \[
   \bm{w}^k = \left( -1, 0, R
     \begin{bmatrix}
       I_{d-\mathsf{r}} & \frac{Q^\top}{k} \\
       \frac{Q}{k} & \bm{0}
     \end{bmatrix}R^{\top}
   \right) \text{ and } \bm{v}^k = \left( -1, 0, R
     \begin{bmatrix}
       I_{d-\mathsf{r}} & \frac{Q^\top}{k} \\
       \frac{Q}{k} & \frac{I_{\mathsf{r}}}{k^2}
     \end{bmatrix}R^{\top}
    \right).
  \]
  Then there exists \( \eta > 0 \) such that \( \{\bm{w}^k\} \subset \{\bm{n}\}^{\perp} \cap B(\eta) \).
  We also observe that \( R
     \begin{bmatrix}
       I_{d-\mathsf{r}} & \frac{Q^\top}{k} \\
       \frac{Q}{k} & \frac{I_{\mathsf{r}}}{k^2}
     \end{bmatrix}R^{\top} \succeq 0 \) for all \( k \) based on standard arguments involving the Schur complement.
  Then \( \{\bm{v}^k\} \subset \Fd \subset \Kld \).
  With that, we have
  \[ \dist(\bm{w}^k, \Kld) \leq \| \bm{w}^k - \bm{v}^k \| = \left\| \frac{I_{\mathsf{r}}}{k^2} \right\|_F = \frac{\sqrt{\mathsf{r}}}{k^2}. \]

  Therefore, by applying \eqref{eq:Holderian-error-bound-F-3} and using \eqref{Fsharpexample}, there exists \( \kappa_B > 0 \) such that
  \[
   0 < \frac{\sqrt{2}\|Q\|_F}{k} = \dist(\bm{w}^k, \Fs) \leq \kappa_B \dist(\bm{w}^k, \Kld)^{\frac{1}{2}} \leq \frac{\kappa_B\mathsf{r}^{\frac{1}{4}}}{k}.
  \]
  Consequently, for all \( k \), we have
  \[
   0 < \frac{\sqrt{2}\|Q\|_F}{\mathsf{r}^{\frac{1}{4}}} \leq \frac{\dist(\bm{w}^k, \Fs)}{\dist(\bm{w}^k, \Kld)^{\frac{1}{2}}} \leq \kappa_B.
  \]
  Similar to the argument in Remark \ref{remark:tightness-F-d}, we conclude that the choice of \( | \cdot |^{\frac{1}{2}} \) is tight.
\end{remark}
Next, we consider the case where \( \bm{n}_y = 0 \).
Define \( \mathfrak{g}_{\log} \) as follows
\begin{equation}\label{eq:g-log}
	\mathfrak{g}_{\log}(t) := \left\{
	\begin{array}{ll}
		0                             & \text{ if } t = 0,                    \\
		-\frac{1}{\log (t)}           & \text{ if } 0 < t \leq \frac{1}{e^2}, \\
		\frac{1}{4} + \frac{1}{4}e^2t & \text{ if } t > \frac{1}{e^2}.
	\end{array}\right.
\end{equation}
We note that \( \mathfrak{g}_{\log} \) is increasing with \( \mathfrak{g}_{\log} (0) = 0 \) and \( | t | \leq \mathfrak{g}_{\log} (t) \) for all \( t \in \R_{+} \).
% there exists a constant \( \widehat{L} \geq 1 \) such that the following inequalities hold for every \( t \in \R_+ \) and \( M > 0 \):
% \begin{equation}
% 	\label{eq:g-log-L-M}
% 	| t |  \leq \mathfrak{g}_{\log} (t), \,\,\mathfrak{g}_{\log} (2t) \leq \widehat{L} \mathfrak{g}_{\log} (t), \,\,\color{red}\mathfrak{g}_{\log} (Mt) \leq \widehat{L}^{1+ | \log _2(M) | }\mathfrak{g}_{\log} (t).\color{black}
% \end{equation}
Moreover, \( \mathfrak{g}_{\log} (t) > \mathfrak{g}_{{\rm d}} (t) \) for any \( t \in (0, \frac{1}{e^2}) \).
With \( \mathfrak{g}_{\log} \), the next theorem shows that \( \gamma _{\bm{n}, \eta } \in (0, \infty] \) for \( \Fs \), which implies that a log-type error bound holds.

\begin{theorem}[Log-type error bound concerning \( \Fs \) if \( \bm{n}_y = 0 \)]\label{thm:log-type-error-bound-F-3}
	Let \( \bm{n} = (0, 0, \bm{n}_Z) \in \partial \Kld^{*} \) with \( \bm{n}_Z \succeq 0 \) and \( 0 < \rank(\bm{n}_Z) < d \) such that \( \Fs = \Kld \cap \{\bm{n}\}^{\perp } \).
	Let \( \eta > 0 \) and let \( \gamma _{\bm{n}, \eta } \) be defined as in~\eqref{eq:def-gamma} with \( \mathcal{F} = \Fs \) and \( \mathfrak{g} = \mathfrak{g}_{\log} \) in~\eqref{eq:g-log}.
	Then \( \gamma _{\bm{n}, \eta } \in (0, \infty] \) and
	\begin{equation}
		\label{eq:log-type-error-bound-F-3}
		\dist(\bm{q}, \Fs) \leq \max \{2, 2\gamma _{\bm{n}, \eta }^{-1}\} \cdot \mathfrak{g}_{\log} (\dist(\bm{q}, \Kld)) \quad\quad \forall \bm{q} \in \{\bm{n}\}^{\perp } \cap B(\eta).
	\end{equation}
\end{theorem}

\begin{proof}
	If \( \gamma _{\bm{n}, \eta } = 0 \), in view of~\cite[Lemma 3.12]{LiLoPo20}, there exists \( \widehat{\bm{v}} \in \Fs \) and sequences \( \{\bm{v}^k\}, \{\bm{w}^k\}, \{\bm{u}^k\} \) being defined as those therein, with the cone being \( \Kld \) and the face being \( \Fs \), such that~\eqref{eq:for-contradiction} holds with \( \mathfrak{g} = \mathfrak{g}_{\log} \) as in~\eqref{eq:g-log}.
	As in the proof of Theorem~\ref{thm:Holderian-error-bound-F-3-ny>0}, the condition \( \{\bm{v}^k\} \subset \partial \Kld \cap B(\eta) \setminus \Fs \) means that we need to consider the following two cases:
	\begin{enumerate}[(i)]
		\item\label{item:F-3-ny=0-vk-not-F-d} \( \bm{v}^k \in \partial \Kld \cap B(\eta) \setminus \Fd \) infinitely often;
		\item\label{item:F-3-ny=0-vk-F-d} \( \bm{v}^k \in \Fd \cap B(\eta) \setminus \Fs \) for all large \( k \).
	\end{enumerate}

	(\ref{item:F-3-ny=0-vk-not-F-d}) Passing to a subsequence if necessary, we can assume that \( \bm{v}^k \in \partial \Kld \cap B(\eta) \setminus \Fd \) for all \( k \), that is,
	\[ \bm{v}^k = (\bm{v}_y^k \ldt (\bm{v}_Z^k / \bm{v}_y^k ), \bm{v}_y^k, \bm{v}_Z^k) \text{ with }  \bm{v}_y^k > 0, \bm{v}_Z^k \succ 0, \quad \text{for all }k. \]
	Then \( \langle \bm{n}, \bm{v}^k \rangle = \tr(\bm{v}_Z^k \bm{n}_Z) \), which is nonnegative since \( \bm{n}_Z \succeq 0, \bm{v}_Z^k \succ 0 \).

	Now, one can check that for all \( k \),
	\begin{equation} \label{eqn:toptop}
		\| \bm{w}^k - \bm{v}^k \| = \frac{\langle \bm{n}, \bm{v}^k \rangle }{\| \bm{n} \| } = \frac{\tr(\bm{v}_Z^k \bm{n}_Z)}{\| \bm{n} \| }.
	\end{equation}
	On the other hand, by Lemma~\ref{lemma:norm-wk-uk}, the formula of \( \Fs \) and Proposition~\ref{prop:psd-error-bound}, we  obtain that for all \( k \),
	\begin{equation}\label{eq:norm-wk-uk-log-type-F-3}
		\| \bm{w}^k - \bm{u}^k \| \leq \dist(\bm{v}^k, \Fs) \leq ( \bm{v}_y^k \ldt (\bm{v}_Z^k / \bm{v}_y^k) )_+ + \bm{v}_y^k + C_P \tr(\bm{v}_Z^k \bm{n}_Z)^{\frac{1}{2}}.
	\end{equation}

	Let \( \tau ^k := \tr(\bm{v}_Z^k \bm{n}_Z) \) and \( \mathsf{r} := \rank(\bm{n}_Z) \).

	If \( \bm{v}_y^k \ldt (\bm{v}_Z^k / \bm{v}_y^k) \geq 0 \) infinitely often, then,
	by passing to a subsequence if necessary, we may assume that \( \det (\bm{v}_Z^k / \bm{v}_y^k) \geq 1 \) for all \( k \), and hence \( (\bm{v}_y^k)^d \leq \det(\bm{v}_Z^k) \) for all \( k \).
	Thus, upon invoking Lemma~\ref{lemma:rel-det-tr-Z-rank}, we obtain that for all \( k \),
	\begin{equation}\label{eq:rel-vky-det-ldt>0}
		\bm{v}_y^k \leq (\det(\bm{v}_Z^k))^{\frac{1}{d}} \leq C(\tau^k)^{\frac{\mathsf{r}}{d}}.
	\end{equation}
	Then, for all sufficiently large \( k \),
	\begin{align*}
		\| \bm{w}^k - \bm{u}^k \| & \overset{(\text{a})}{\leq } d\bm{v}_y^k |\log(\eta)| - d\bm{v}_y^k \log (\bm{v}_y^k) + \bm{v}_y^k + C_P( \tau^k)^{\frac{1}{2}}                                                           \\
		                          & \overset{(\text{b})}{\leq } (d|\log(\eta)| + 1)C(\tau^k)^{\frac{\mathsf{r}}{d}} - dC(\tau^k)^{\frac{\mathsf{r}}{d}} \log (C(\tau^k)^{\frac{\mathsf{r}}{d}}) + C_P (\tau^k)^{\frac{1}{2}} \\
		                          & = (d|\log(\eta)| + 1)C(\tau^k)^{\frac{\mathsf{r}}{d}} - Cd\log(C)(\tau^k)^{\frac{\mathsf{r}}{d}} - C\mathsf{r}(\tau^k)^{\frac{\mathsf{r}}{d}} \log (\tau^k) + C_P (\tau^k)^{\frac{1}{2}}  \\
		                          & \overset{(\text{c})}{\leq } (Cd|\log(\eta)| + C - Cd\log(C))(\tau^k)^{\frac{\mathsf{r}}{d}} + C\mathsf{r} (\tau^k)^{\frac{r}{2d}} + C_P(\tau^k)^{\frac{1}{2}}                            \\
		                          & \leq \Big|Cd|\log(\eta)| + C - Cd\log(C)\Big|(\tau^k)^{\rho } + C\mathsf{r}(\tau^k)^{\rho } + C_P(\tau^k)^{\rho }                                                                                \\
		                          & = C_{\#}(\tau^k)^{\rho },
	\end{align*}
	where \( \rho = \min \{\frac{\mathsf{r}}{2d}, \frac{1}{2}\} \) and \( C_{\#} := \Big|Cd|\log(\eta)|  + C - Cd\log(C)\Big| + C\mathsf{r} + C_P > 0 \), (a) comes from~\eqref{eq:norm-wk-uk-log-type-F-3} and~\eqref{eq:vky-ldt-vkz/vky-upper-bound-by-eta}, (b) holds because of \eqref{eq:rel-vky-det-ldt>0} and  the fact that \( x \mapsto -x\log (x) \) is increasing for all sufficiently small positive \( x \), (c) is true by~\eqref{eq:relationships-t-a-tlogt-1/logt} (with \( \alpha = \mathsf{r}/d>0 \)).

	Therefore, we conclude that
	\[ \lim_{k \to \infty} \frac{\mathfrak{g}_{\log}(\| \bm{w}^k - \bm{v}^k \| )}{\| \bm{w}^k - \bm{u}^k \| } \geq \liminf_{k \to \infty} \frac{\| \bm{w}^k - \bm{v}^k \| ^{\rho }}{\| \bm{w}^k - \bm{u}^k \| } \geq \lim_{k \to \infty} \frac{(\tau^k)^{\rho }}{\| \bm{n} \| ^{\rho } C_{\#} (\tau^k)^{\rho }} = \frac{1}{\| \bm{n} \| ^{\rho } C_{\#}} > 0. \]
	This contradicts~\eqref{eq:for-contradiction} with \( \mathfrak{g}_{\log} \) in place of \( \mathfrak{g} \) and hence this case cannot happen.

	If \( \bm{v}_y^k \ldt ( \bm{v}_Z^k  / \bm{v}_y^k  ) < 0 \) infinitely often, then by passing to a subsequence if necessary, we may assume that \( \bm{v}_y^k \ldt ( \bm{v}_Z^k  / \bm{v}_y^k  ) < 0 \) for all large \( k \).
  Moreover, recalling the exponential form of \( \Kld \) in~\eqref{eq:exponential-form-log-det-cone}, we have \( ( \bm{v}_y^k  )^d e^{\bm{v}_x^k / \bm{v}_y^k } = \det( \bm{v}_Z^k  ) \) for all \( k \).
	Invoking Lemma~\ref{lemma:rel-det-tr-Z-rank}, we then see that for all \( k \),
	\[ \bm{v}_y^k e^{\bm{v}_x^k / (d\bm{v}_y^k) } = (\det(\bm{v}_Z^k))^{\frac{1}{d}} \leq C ( \tau ^k )^{\frac{\mathsf{r}}{d}}. \]
	Thus, by taking logarithm on both sides, the above inequality becomes
	\[ \log (\bm{v}_y^k) + \frac{\bm{v}_x^k}{d \bm{v}_y^k} \leq  \log (C) + \frac{\mathsf{r}}{d} \log (\tau^k). \]
	Since \( \bm{v}_y^k \to 0\), \( \tau ^k \to 0 \), and both sequences are positive, we note that \( -\bm{v}_y^k \log (\tau^k) > 0 \) for all large \( k \).
	After multiplying \( -\bm{v}_y^k  \) on both sides of the above display and rearranging terms, we  see that for all large \( k \),
	\[ 0 < -\bm{v}_y^k \log (\tau^k) \leq \frac{d\log (C)\bm{v}_y^k }{\mathsf{r}} - \frac{d \bm{v}_y^k \log (\bm{v}_y^k)}{\mathsf{r}} - \frac{\bm{v}_x^k}{\mathsf{r}}. \]
	Then, by passing to the limit on both sides of the above display, we obtain that
  \begin{align}
    0 & \leq \limsup_{k \to \infty} - \bm{v}_y^k \log (\tau^k) \leq \limsup_{k \to \infty}\frac{d\log (C)\bm{v}_y^k }{\mathsf{r}} - \frac{d \bm{v}_y^k \log (\bm{v}_y^k)}{\mathsf{r}} - \frac{\bm{v}_x^k}{\mathsf{r}} \notag \\
    & = -\lim_{k \to \infty} \frac{\bm{v}_x^k}{\mathsf{r}} = -\frac{\widehat{\bm{v}}_x}{\mathsf{r}}. \label{eq:lim-vky-log(tauk)}
  \end{align}
	Therefore, we  conclude that
  \begin{align*}
			 & \quad \lim_{k \to \infty} \frac{\mathfrak{g}_{\log}(\| \bm{w}^k - \bm{v}^k \| )}{\| \bm{w}^k - \bm{u}^k \| } \overset{(\text{a})}{\geq} \liminf_{k \to \infty} -\frac{1}{\log (\tau^k) - \log (\| \bm{n} \| )}\frac{1}{\bm{v}_y^k + C_P(\tau^k)^{\frac{1}{2}}}       \\
			 & = \liminf_{k \to \infty} \frac{1}{\log (\| \bm{n} \| )(\bm{v}_y^k + C_P(\tau^k)^{\frac{1}{2}}) - \bm{v}_y^k \log (\tau^k) - C_P(\tau^k)^{\frac{1}{2}}\log (\tau^k)} \\
      & \overset{(\text{b})}{\geq } \lim_{k \to \infty} \frac{-\mathsf{r}}{\bm{v}_x^k } \in (0,\infty],
  \end{align*}
	where (a) is true owing to~\eqref{eqn:toptop} and~\eqref{eq:norm-wk-uk-log-type-F-3}, (b) comes from~\eqref{eq:lim-vky-log(tauk)} and the fact $\bm{v}_y^k \to 0, {\tau^k} \to 0$,  %,	\( \bm{v}_y^k + C_P(\tau^k)^{\frac{1}{2}} \downarrow 0, \bm{v}_y^k \to 0,  \),
	the last inequality holds because \( \widehat{\bm{v}}_x \leq 0 \) thanks to \( \bm{v}_y^k \ldt ( \bm{v}_Z^k  / \bm{v}_y^k  ) < 0 \) for all large \( k \).
	The above display contradicts~\eqref{eq:for-contradiction} with \( \mathfrak{g}_{\log} \) in place of \( \mathfrak{g} \) and so this case cannot happen.

	(\ref{item:F-3-ny=0-vk-F-d}) In this case, we have from Lemma~\ref{lemma:ny>=0-vky=0-Holderian} that
	$
	 \liminf_{k \to \infty} \frac{\| \bm{w}^k - \bm{v}^k \| ^{1/2 }}{\| \bm{w}^k - \bm{u}^k \| } \in (0,\infty]
	$, which implies
	 that
	 \[
	 \lim_{k \to \infty} \frac{\mathfrak{g}_{\log}(\| \bm{w}^k - \bm{v}^k \|)}{\| \bm{w}^k - \bm{u}^k \| } \geq \liminf_{k \to \infty} \frac{\| \bm{w}^k - \bm{v}^k \|^{1/2}}{\| \bm{w}^k - \bm{u}^k \| } \in (0,\infty],
	 \]
	 where we recall that $|t|^{1/2} \leq \mathfrak{g}_{\log}(t)$ for $t$ sufficiently small. In view of the definition of
	 $\gamma_{\bm{n},\eta} $, 	case (\ref{item:F-3-ny=0-vk-F-d}) also cannot happen.
	
	Hence, we  conclude that \( \gamma _{\bm{n}, \eta } \in (0,\infty] \). Using this together with \cite[Theorem 3.10]{LiLoPo20}, we deduce that~\eqref{eq:log-type-error-bound-F-3} holds.
\end{proof}

\begin{remark}[Tightness of~\eqref{eq:log-type-error-bound-F-3}]\label{remark:tightness-log-type-F-3}
	Let \( \bm{n} = (0, 0, \bm{n}_Z) \) with \( \bm{n}_Z \succeq 0\), \(0 < \rank(\bm{n}_Z) < d \).
	Then, we have \( \Fs = \{\bm{n}\}^{\perp} \cap \Kld \) from~\eqref{eq:n-dim-faces-F-n}.
	Consider the sequence \( \bm{w}^k = (-1, 1 / k, \bm{0}) \), \( \bm{v}^k = (-1, 1 / k, I_d / (ke^{\frac{k}{d}})) \) and \( \bm{u}^k = (-1, 0, \bm{0}) \) for every \( k \), we note that \( \bm{w}^k \in \{\bm{n}\}^{\perp}, \bm{v}^k \in \Kld \) and \( \bm{u}^k = P_{\Fs}(\bm{w}^k) \) for every \( k \).
	Moreover, there exists \( \eta > 0 \) such that \( \{\bm{w}^k\} \subseteq B(\eta) \).
	Therefore, applying~\eqref{eq:log-type-error-bound-F-3}, there exists \( \kappa _B > 0 \) such that
	\[ \frac{1}{k} = \dist(\bm{w}^k, \Fs) \leq \kappa _B \mathfrak{g}_{\log} (\dist(\bm{w}^k, \Kld)) \leq \kappa _B \mathfrak{g}_{\log} \left( \frac{\sqrt{d}}{ke^{\frac{k}{d}}} \right) \quad \forall k \in \mathbb{N}.  \]
	In view of the definition of \( \mathfrak{g}_{\log} \) (see~\eqref{eq:g-log}) and its monotonicity, for large enough \( k \) we have
	\[ \frac{1}{k} = \dist(\bm{w}^k, \Fs) \leq \kappa _B \mathfrak{g}_{\log} (\dist(\bm{w}^k, \Kld)) \le \frac{\kappa_B}{\log k + (k/d) - \log\sqrt{d}} \leq \kappa_B\frac{2d}{k}. \]
	Consequently, it holds that for all sufficiently large \( k \),
	\[ \frac{1}{2d} \leq \frac{\dist(\bm{w}^k, \Fs)}{\mathfrak{g}_{\log} (\dist(\bm{w}^k, \Kld))} \leq \kappa _B. \]
	Similar to the argument in Remark~\ref{remark:tightness-F-d}, we conclude that the choice of \( \mathfrak{g}_{\log}  \) is tight.
\end{remark}

Using Theorems~\ref{thm:Holderian-error-bound-F-3-ny>0} and~\ref{thm:log-type-error-bound-F-3} in combination with
\cite[Lemma~3.9]{LiLoPo20}, we obtain the following one-step facial residual functions for \( \Kld \) and \( \bm{n} \).

\begin{corollary}\label{corollary:FRFs-F-3}
	Let \( \bm{n} = (0, \bm{n}_y, \bm{n}_Z) \in \partial \Kld^{*} \) with \( \bm{n}_y \geq 0\), \(\bm{n}_Z \succeq 0 \) and \( 0 < \rank(\bm{n}_Z) < d \) such that \( \Fs = \Kld \cap \{\bm{n}\}^{\perp }  \).
	\begin{enumerate}[(i)]
		\item If \( \bm{n}_y > 0 \), let \( \gamma _{\bm{n}, t} \) be as in~\eqref{eq:def-gamma} with $\stdFace = \Fs$ and \( \mathfrak{g} = | \cdot  | ^{\frac{1}{2}} \).
		      Then the function \( \psi _{\mathcal{K}, \bm{n}} : \R_{+} \times \R_{+} \to \R_{+} \) defined by
		      \[ \psi _{\mathcal{K}, \bm{n}} (\epsilon, t) := \max \left\{\epsilon, \epsilon / \| \bm{n} \| \right\} + \max \left\{2t ^{\frac{1}{2}}, 2\gamma _{\bm{n}, t}^{-1}\right\}\left(\epsilon  + \max\left\{\epsilon , \epsilon / \| \bm{n} \| \right\}\right)^{\frac{1}{2}} \]
		      is a one-step facial residual function for \( \Kld  \) and \( \bm{n} \).
		\item If \( \bm{n}_y = 0 \), let \( \gamma _{\bm{n}, t } \) be as in~\eqref{eq:def-gamma} with $\stdFace = \Fs$ and \( \mathfrak{g} = \mathfrak{g}_{\log} \) in~\eqref{eq:g-log}.
		      Then the function \( \psi _{\mathcal{K}, \bm{n}} : \R_{+} \times \R_{+} \to \R_{+} \) defined by
		      \[ \psi _{\mathcal{K}, \bm{n}} (\epsilon, t) := \max \left\{\epsilon, \epsilon / \| \bm{n} \| \right\} + \max \left\{2, 2\gamma _{\bm{n}, t}^{-1}\right\} \mathfrak{g}_{\log}\left(\epsilon  + \max\left\{\epsilon , \epsilon / \| \bm{n} \| \right\}\right) \]
		      is a one-step facial residual function for \( \Kld  \) and \( \bm{n} \).
	\end{enumerate}
\end{corollary}

\subsubsection{\( \Ff  \): the exceptional 1-dimensional face}\label{subsubsec:FRFs-F-inf}
We first show a Lipschitz error bound concerning \( \Ff  \) if \( \bm{n}_y > 0 \).
\begin{theorem}[Lipschitz error bound concerning \( \Ff  \) if \( \bm{n}_y > 0 \)]\label{thm:Lipschitz-error-bound-F-inf-n-y>0}
	Let \( \bm{n} = (0, \bm{n}_y, \bm{n}_Z) \in \partial\Kld^{*}  \) with \( \bm{n}_y > 0\) and \(\bm{n}_Z \succ 0 \) such that \( \Ff = \Kld \cap \{\bm{n}\}^{\perp }  \).
	Let \( \eta > 0 \) and let \( \gamma _{\bm{n}, \eta } \) be defined as in~\eqref{eq:def-gamma} with $\stdFace = \Ff$ and \( \mathfrak{g}= | \cdot  | \).
	Then \( \gamma _{\bm{n}, \eta } \in (0, \infty] \) and
	\begin{equation}\label{eq:Lipschitz-error-bound-F-inf-n-y>0}
		\dist(\bm{q}, \Ff) \leq \max \{2, 2\gamma _{\bm{n}, \eta }^{-1}\}\cdot \dist(\bm{q}, \Kld) \quad\quad \forall \bm{q} \in \{\bm{n}\}^{\perp } \cap B(\eta).
	\end{equation}
\end{theorem}

\begin{proof}
	If \( \gamma _{\bm{n}, \eta } = 0 \), in view of~\cite[Lemma 3.12]{LiLoPo20}, there exists \( \widehat{\bm{v}} \in \Ff \) and sequences \( \{\bm{v}^k\}, \{\bm{w}^k\}, \{\bm{u}^k\} \) being defined as those therein, with the cone being \( \Kld \) and the face being \( \Ff \), such that~\eqref{eq:for-contradiction} holds with \( \mathfrak{g} = | \cdot | \).
	Note that \( \{\bm{v}^k\} \subset \partial \Kld \cap B(\eta) \setminus \Ff \) means that we need to consider the following two cases:
	\begin{enumerate}[(i)]
		\item\label{item:F-inf-ny>0-vk-not-F-d} \( \bm{v}^k \in \partial \Kld \cap B(\eta) \setminus \Fd \) infinitely often;
		\item\label{item:F-inf-ny>0-vk-F-d} \( \bm{v}^k \in \Fd \cap B(\eta) \setminus \Ff \) for all large \( k \).
	\end{enumerate}

	(\ref{item:F-inf-ny>0-vk-not-F-d}) Without loss of generality, we  assume that \( \bm{v}^k \in \partial \Kld \cap B(\eta) \setminus \Fd \) for all \( k \) by passing to a subsequence if necessary, that is,
	\[ \bm{v}^k = (\bm{v}_y^k \ldt (\bm{v}_Z^k / \bm{v}_y^k), \bm{v}_y^k , \bm{v}_Z^k ) \text{ with } \bm{v}_y^k > 0, \bm{v}_Z^k \succ 0 \quad \text{for all }k. \]
	Then, \( \langle \bm{n}, \bm{v}^k \rangle = \bm{n}_y \bm{v}_y^k + \tr(\bm{v}_Z^k \bm{n}_Z) > 0 \) and
	\[ \| \bm{w}^k - \bm{v}^k \| = \frac{\bm{n}_y \bm{v}_y^k + \tr(\bm{v}_Z^k \bm{n}_Z)}{\| \bm{n} \| }. \]
	On the other hand, by Lemma~\ref{lemma:norm-wk-uk}, we  obtain that for all \( k \),
	\begin{equation}
    \label{eq:F-inf-case-1}
    \| \bm{w}^k - \bm{u}^k \| \leq \dist(\bm{v}^k, \Ff) \leq (\bm{v}_y^k \ldt (\bm{v}_Z^k / \bm{v}_y^k))_{+} + \bm{v}_y^k + \| \bm{v}_Z^k  \| _F.
  \end{equation}
	If \( \bm{v}_y^k \ldt (\bm{v}_Z^k / \bm{v}_y^k) \geq 0 \) infinitely often, by passing to a subsequence if necessary, we may assume that \( \bm{v}_y^k \ldt (\bm{v}_Z^k / \bm{v}_y^k) \geq 0 \) for all large \( k \) and hence, recalling
	that $\| \bm{v}_Z^k  \| _F \leq \tr(\bm{v}_Z^k)$ (since $\bm{v}_Z^k \succ 0$), we obtain
  \begin{align*}
    \| \bm{w}^k - \bm{u}^k \| & \leq \bm{v}_y^k \ldt (\bm{v}_Z^k / \bm{v}_y^k) + \bm{v}_y^k + \tr(\bm{v}_Z^k) = \bm{v}_y^k +\tr(\bm{v}_Z^k) + \bm{v}_y^k \log (\prod_{i=1}^d \lambda _i(\bm{v}_Z^k) / \bm{v}_y^k  )                                                 \\
                            & = \bm{v}_y^k + \tr(\bm{v}_Z^k) + \sum_{i=1}^d \bm{v}_y^k \log (\lambda _i(\bm{v}_Z^k) / \bm{v}_y^k ) \\
                            & \overset{(\text{a})}{\leq } \bm{v}_y^k + \tr(\bm{v}_Z^k) + \sum_{i=1}^d \bm{v}_y^k (\lambda _i(\bm{v}_Z^k) / \bm{v}_y^k  +  1) \\
                              & = \bm{v}_y^k + \tr(\bm{v}_Z^k) + \tr(\bm{v}_Z^k) + d\bm{v}_y^k  = (1 + d)\bm{v}_y^k + 2 \tr(\bm{v}_Z^k),
  \end{align*}
	where (a) holds because \( \log (x) \leq x + 1 \) for all \( x > 0 \).

	Combining these identities and using \eqref{eqn:traceproperties} yields:
	\[
    \begin{split}
      & \lim_{k \to \infty} \frac{\| \bm{w}^k - \bm{v}^k \| }{\| \bm{w}^k - \bm{u}^k \| } \geq \liminf_{k \to \infty} \frac{\frac{\bm{n}_y }{\| \bm{n} \| (1 + d)} (1 + d) \bm{v}_y^k + \frac{\lambda _{\min }(\bm{n}_Z)}{2 \| \bm{n} \| } 2 \tr(\bm{v}_Z^k) }{(1 + d)\bm{v}_y^k + 2 \tr(\bm{v}_Z^k)} \\
      \geq & \min \left\{\frac{\bm{n}_y }{\| \bm{n} \| (1 + d)}, \frac{\lambda _{\min }(\bm{n}_Z)}{2 \| \bm{n} \| }  \right\} > 0.
    \end{split}
  \]
	This contradicts~\eqref{eq:for-contradiction} with \( | \cdot | \) in place of \( \mathfrak{g} \) and hence this case cannot happen.

	If \( \bm{v}_y^k \ldt (\bm{v}_Z^k / \bm{v}_y^k) < 0 \) infinitely often, then by extracting a subsequence if necessary, we may assume that \( \bm{v}_y^k \ldt (\bm{v}_Z^k / \bm{v}_y^k) < 0 \) for all large \( k \) and hence \eqref{eq:F-inf-case-1} becomes
	\[ \| \bm{w}^k - \bm{u}^k \| \leq \bm{v}_y^k +\tr(\bm{v}_Z^k). \]
	Therefore,
	\[ \lim_{k \to \infty} \frac{\| \bm{w}^k - \bm{v}^k \| }{\| \bm{w}^k - \bm{u}^k \| } \geq \liminf_{k \to \infty} \frac{\frac{\bm{n}_y }{\| \bm{n} \| } \bm{v}_y^k + \frac{\lambda _{\min }(\bm{n}_Z)}{\| \bm{n} \| } \tr(\bm{v}_Z^k) }{\bm{v}_y^k + \tr(\bm{v}_Z^k)} \geq \min \left\{\frac{\bm{n}_y }{\| \bm{n} \| }, \frac{\lambda _{\min }(\bm{n}_Z) }{\| \bm{n} \| } \right\} > 0.  \]
	The above inequality contradicts~\eqref{eq:for-contradiction} with \( | \cdot | \) in place of \( \mathfrak{g} \) and hence this case cannot happen.

	(\ref{item:F-inf-ny>0-vk-F-d}) By Lemma~\ref{lemma:ny>=0-vky=0-Holderian}, case (\ref{item:F-inf-ny>0-vk-F-d}) also cannot happen.

	Overall, we conclude that \( \gamma _{\bm{n}, \eta } \in (0,\infty] \), and so by ~\cite[Theorem 3.10]{LiLoPo20},~\eqref{eq:Lipschitz-error-bound-F-inf-n-y>0} holds.
\end{proof}

Note that a Lipschitz error bound is always tight up to a constant, so~\eqref{eq:Lipschitz-error-bound-F-inf-n-y>0} is tight.

If \( \bm{n}_y = 0 \), we have the following Log-type error bound for \( \Kld \).

\begin{theorem}[Log-type error bound concerning \( \Ff \) if \( \bm{n}_y = 0 \)]\label{thm:log-type-error-bound-F-inf}
	Let \( \bm{n} = (0, 0, \bm{n}_Z) \in \partial\Kld^{*}  \) with \( \bm{n}_Z \succ 0 \) such that \( \Ff = \Kld \cap \{\bm{n}\}^{\perp }   \).
	Let \( \eta > 0 \) and let \( \gamma _{\bm{n}, \eta } \) be defined as in~\eqref{eq:def-gamma} with $\stdFace = \Ff$ and \( \mathfrak{g} = \mathfrak{g}_{\log} \) in \eqref{eq:g-log}.
	Then \( \gamma _{\bm{n}, \eta } \in (0, \infty] \) and
	\begin{equation}\label{eq:log-type-error-bound-F-inf}
		\dist(\bm{q}, \Ff) \leq \max \{2, 2\gamma _{\bm{n}, \eta }^{-1}\} \cdot \mathfrak{g}_{\log}(\dist(\bm{q}, \Kld)) \quad\quad \forall \bm{q} \in \{\bm{n}\}^{\perp } \cap B(\eta).
	\end{equation}
\end{theorem}

\begin{proof}
	If \( \gamma _{\bm{n}, \eta } = 0 \), in view of~\cite[Lemma 3.12]{LiLoPo20}, there exists \( \widehat{\bm{v}} \in \Ff \) and sequences \( \{\bm{v}^k\}, \{\bm{w}^k\}, \{\bm{u}^k\} \) being defined as those therein, with the cone being \( \Kld \) and the face being \( \Ff \), such that~\eqref{eq:for-contradiction} holds with \( \mathfrak{g} = \mathfrak{g}_{\log} \) as in~\eqref{eq:g-log}.
	Note that \( \{\bm{v}^k\} \subset \partial \Kld \cap B(\eta) \setminus \Ff \) means that we need to consider the following two cases:
	\begin{enumerate}[(i)]
		\item\label{item:F-inf-ny=0-vk-not-F-d} \( \bm{v}^k \in \partial \Kld \cap B(\eta) \setminus \Fd \) infinitely often;
		\item\label{item:F-inf-ny=0-vk-F-d} \( \bm{v}^k \in \Fd \cap B(\eta) \setminus \Ff \) for all large \( k \).
	\end{enumerate}

	(\ref{item:F-inf-ny=0-vk-not-F-d}) Without loss of generality, we  assume that \( \bm{v}^k \in \partial \Kld \cap B(\eta) \setminus \Fd \) for all \( k \) by passing to a subsequence if necessary, that is,
	\[ \bm{v}^k = (\bm{v}_y^k \ldt (\bm{v}_Z^k / \bm{v}_y^k), \bm{v}_y^k , \bm{v}_Z^k ) \text{ with } \bm{v}_y^k > 0, \bm{v}_Z^k \succ 0 \quad \text{for all }k. \]
	Then \( \langle \bm{n}, \bm{v}^k \rangle = \tr(\bm{v}_Z^k \bm{n}_Z) \geq 0 \) and
	\begin{equation}
    \label{eq:F-d-case-1}
    \| \bm{w}^k - \bm{v}^k \| = \frac{\langle \bm{n}, \bm{v}^k \rangle }{\| \bm{n} \| } = \frac{\tr(\bm{v}_Z^k \bm{n}_Z)}{\| \bm{n} \| }.
  \end{equation}
	In addition, by Lemma~\ref{lemma:norm-wk-uk}, we obtain that for all \( k \),
	\begin{equation}\label{eq:norm-wk-uk-log-type-F-f}
		\| \bm{w}^k - \bm{u}^k \| \leq \dist(\bm{v}^k, \Ff) \leq ( \bm{v}_y^k \ldt (\bm{v}_Z^k / \bm{v}_y^k) )_+ + \bm{v}_y^k + \| \bm{v}_Z^k  \| _F.
	\end{equation}

	Let \( \tau ^k := \tr(\bm{v}_Z^k \bm{n}_Z) \).

	If \( \bm{v}_y^k \ldt (\bm{v}_Z^k / \bm{v}_y^k) \geq 0 \) infinitely often, then by passing to a subsequence if necessary, we may assume that \( \det (\bm{v}_Z^k / \bm{v}_y^k) \geq 1 \) for all \( k \).
  Hence we have \( (\bm{v}_y^k)^d \leq \det(\bm{v}_Z^k) \).
	Thus, combining Lemma~\ref{lemma:rel-det-tr-Z-rank} with \( \rank(\bm{n}_Z) = d \), we obtain that for all \( k \),
	\begin{equation}\label{eq:rel-vky-det-ldt>0-d}
		\bm{v}_y^k \leq (\det(\bm{v}_Z^k))^{\frac{1}{d}} \leq C\tau^k.
	\end{equation}
	Then, for sufficiently large \( k \),
	\begin{equation}\label{eq:up-bd-wk-uk-F-inf}
		\begin{aligned}
			&\| \bm{w}^k - \bm{u}^k \|  \overset{(\text{a})}{\leq }d|\log(\eta)| \bm{v}_y^k - d \bm{v}_y^k \log (\bm{v}_y^k) + \bm{v}_y^k  + \tr(\bm{v}_Z^k)                                  \\
			                          & =(d|\log(\eta)| + 1)\bm{v}_y^k - d\bm{v}_y^k \log (\bm{v}_y^k) + \frac{1}{\lambda _{\min }(\bm{n}_Z)} \lambda _{\min }(\bm{n}_Z) \tr(\bm{v}_Z^k)      \\
			                          & \overset{(\text{b})}{\leq } (Cd|\log(\eta)| + C)\tau ^k - Cd\tau ^k\log (C\tau^k) + \frac{1}{\lambda _{\min }(\bm{n}_Z)}\tau ^k                       \\
			                          & = (Cd|\log(\eta)| + C - Cd\log(C))\tau ^k - Cd\tau ^k\log (\tau^k) + \frac{1}{\lambda _{\min }(\bm{n}_Z)}\tau ^k                                      \\
			                          & \overset{(\text{c})}{\leq } \left(\Big|Cd|\log(\eta)|  + C - Cd\log(C)\Big| + Cd + \frac{1}{\lambda _{\min }(\bm{n}_Z)}\right) (-\tau^k \log(\tau^k)) \\
			                          & = C_{\infty } (-\tau^k \log(\tau^k)),
		\end{aligned}
  \end{equation}
	where \( C_{\infty } := \Big|Cd|\log(\eta)|  + C - Cd\log(C)\Big| + Cd + \frac{1}{\lambda _{\min }(\bm{n}_Z)} > 0 \), (a) comes from~\eqref{eq:norm-wk-uk-log-type-F-f} and~\eqref{eq:vky-ldt-vkz/vky-upper-bound-by-eta}, (b) holds because of \eqref{eqn:traceproperties}, ~\eqref{eq:rel-vky-det-ldt>0-d} and the fact that \( x \mapsto -x\log (x) \) is increasing for all sufficiently small positive \( x \), (c) is true because \( x\leq -x\log(x) \) for sufficiently small \( x \) and \( \tau^k \rightarrow 0 \) because \( \bm{v}_Z^k \rightarrow 0 \).

	Hence,
	\[
    \begin{split}
      & \lim_{k \to \infty} \frac{\mathfrak{g}_{\log}(\| \bm{w}^k - \bm{v}^k \| )}{\| \bm{w}^k - \bm{u}^k \| } \geq \liminf_{k \to \infty} -\frac{1}{\log \left( \frac{\tau^k}{\| \bm{n} \|} \right)} \frac{1}{C_{\infty }(-\tau^k \log(\tau^k))} \\
      \geq & \lim_{k \to \infty} \frac{- \tau ^k \log (\tau^k)}{C_{\infty } (-\tau^k \log(\tau^k))} = \frac{1}{C_{\infty}} > 0,
    \end{split}
  \]
	where the first inequality comes from \eqref{eq:F-d-case-1} and \eqref{eq:up-bd-wk-uk-F-inf}, the second inequality comes from~\eqref{eq:relationships-t-a-tlogt-1/logt} (with \( \alpha = 1 \) and \( s = \frac{1}{\|\bm{n} \|} \)).
	This contradicts~\eqref{eq:for-contradiction} with \( \mathfrak{g}_{\log} \) in place of \( \mathfrak{g} \) and hence this case cannot happen.

	If \( \bm{v}_y^k \ldt ( \bm{v}_Z^k  / \bm{v}_y^k  ) < 0 \) infinitely often, then by passing to a subsequence if necessary, we may assume that that \( \bm{v}_y^k \ldt ( \bm{v}_Z^k  / \bm{v}_y^k  ) < 0 \) for all \( k \).
  Moreover, recalling the exponential form of \( \Kld \) in~\eqref{eq:exponential-form-log-det-cone}, we have \( ( \bm{v}_y^k  )^d e^{\bm{v}_x^k / \bm{v}_y^k } = \det( \bm{v}_Z^k  ) \) for all \( k \).
	Upon invoking Lemma~\ref{lemma:rel-det-tr-Z-rank} with \( \rank(\bm{n}_Z) = d \), we  then see that for all \( k \),
	we have \[ \bm{v}_y^k e^{\bm{v}_x^k / (d\bm{v}_y^k) } = (\det(\bm{v}_Z^k))^{\frac{1}{d}} \leq C  \tau ^k . \]
	Thus, by taking the logarithm on both sides, the above inequality becomes
	\[ \log (\bm{v}_y^k) + \frac{\bm{v}_x^k}{d \bm{v}_y^k} \leq  \log (C) + \log (\tau^k). \]
	Since \( \bm{v}_y^k \to 0, \,\tau ^k \to 0 \) and $\{\bm{v}_y^k\}$, $\{\tau^k\}$ are positive sequences, we note that \( -\bm{v}_y^k \log (\tau^k) > 0 \) for all large \( k \).
	After multiplying \( -\bm{v}_y^k  \) on both sides of the above display and rearranging terms, we see that for all large \( k \),
	\[ 0 < -\bm{v}_y^k \log (\tau^k) \leq -\bm{v}_y^k \log (\bm{v}_y^k) - \frac{\bm{v}_x^k}{d} + \log (C) \bm{v}_y^k. \]
	Then, by passing to the limit on both sides of the above display, we obtain that
	\begin{equation}\label{eq:lim-vky-log(tauk)-d}
    \begin{split}
      0 & \leq \limsup_{k \to \infty} -\bm{v}_y^k \log (\tau^k) \leq \limsup_{k \to \infty} -\bm{v}_y^k \log (\bm{v}_y^k) - \frac{\bm{v}_x^k}{d} + \log (C) \bm{v}_y^k \\
      & = -\lim_{k \to \infty} \frac{\bm{v}_x^k}{d} = -\frac{\widehat{\bm{v}}_x}{d}.
    \end{split}
	\end{equation}
	Note also that since \( \bm{n}_Z \) has full rank, we have upon invoking the equivalence in \eqref{eqn:traceproperties} that \( \{\bm{n}_Z\}^\perp \cap \mathcal{S}_+^d = \{\bm{0} \} \).
  Then Proposition~\ref{prop:psd-error-bound} guarantees that \( C_P \tau^k \geq \|\bm{v}_Z^k\|_F \).

	Therefore, altogether we  conclude that
	\[
		\begin{aligned}
			 & \quad \lim_{k \to \infty} \frac{\mathfrak{g}_{\log}(\| \bm{w}^k - \bm{v}^k \| )}{\| \bm{w}^k - \bm{u}^k \| } \overset{(\text{a})}{\geq} \liminf_{k \to \infty}  -\frac{1}{\log (\tau^k) - \log (\| \bm{n} \| )}\frac{1}{\bm{v}_y^k + C_P\tau^k} \\
			 & \geq \liminf_{k \to \infty} \frac{1}{\log (\| \bm{n} \| )(\bm{v}_y^k + C_P\tau^k) - \bm{v}_y^k \log (\tau^k) - C_P\tau^k\log (\tau^k)} \overset{(\text{b})}{\geq } \lim_{k \to \infty} \frac{-d}{\bm{v}_x^k } \in (0,\infty],
		\end{aligned}
	\]
	where (a) is true owing to~\eqref{eq:F-d-case-1}, \eqref{eq:norm-wk-uk-log-type-F-f}, (b) comes from~\eqref{eq:lim-vky-log(tauk)-d}, \( \tau^k \log (\tau^k) \to 0 \) and \( \bm{v}_y^k + C_P\tau^k \to 0 \), the last inequality holds because \( \widehat{\bm{v}}_x \leq 0 \).
	The above display contradicts~\eqref{eq:for-contradiction} with \( \mathfrak{g}_{\log} \) in place of \( \mathfrak{g} \) and hence this case cannot happen.

	(\ref{item:F-inf-ny=0-vk-F-d}) Analogously to the proof of Theorem~\ref{thm:log-type-error-bound-F-3}, by Lemma~\ref{lemma:ny>=0-vky=0-Holderian}, case (\ref{item:F-inf-ny=0-vk-F-d}) cannot happen.

	Therefore, we  obtain that \( \gamma _{\bm{n}, \eta } \in (0,\infty] \). Using this together with~\cite[Theorem 3.10]{LiLoPo20}, we deduce that~\eqref{eq:log-type-error-bound-F-3} holds.
\end{proof}

\begin{remark}[Tightness of~\eqref{eq:log-type-error-bound-F-inf}]\label{remark:tightness-F-f}
	Let \( \bm{n} = (0, 0, \bm{n}_Z) \) with \( \bm{n}_Z \succ 0 \).
	Then, \( \Ff = \{\bm{n}\}^{\perp} \cap \Kld \).
	Consider the same sequences \( \{\bm{v}^k\}, \{\bm{w}^k\}, \{\bm{u}^k\} \) in Remark~\ref{remark:tightness-log-type-F-3}, i.e., for every \( k \),
	\[ \bm{v}^k = (-1, 1 / k, I_d / (k e^{\frac{k}{d}})), \quad \bm{w}^k = (-1, 1 / k, \bm{0}), \quad \bm{u}^k = (-1, 0, \bm{0}). \]
	Note that there exists \( \eta > 0 \) such that \( \bm{w}^k \in \{\bm{n}\}^{\perp} \cap B(\eta), \, \bm{v}^k \in \Kld \) and \( \bm{u}^k = P_{\Ff}(\bm{w}^k) \) for any \( k \).
	Therefore, applying~\eqref{eq:log-type-error-bound-F-inf}, there exists \( \kappa _B > 0 \) such that
	\[ \frac{1}{k} = \dist(\bm{w}^k, \Ff) \leq \kappa _B \mathfrak{g}_{\log } (\dist(\bm{w}^k, \Kld)) \leq \kappa _B \mathfrak{g}_{\log} \left( \frac{\sqrt{d}}{ke^{\frac{k}{d}}} \right) \quad \forall k \in \mathbb{N}. \]
	In view of the definition of \( \mathfrak{g}_{\log} \) (see~\eqref{eq:g-log}) and its monotonicity, for large enough \( k \) we have
	\[ \frac{1}{k} = \dist(\bm{w}^k, \Ff) \leq \kappa _B \mathfrak{g}_{\log} (\dist(\bm{w}^k, \Kld)) \le \frac{\kappa_B}{\log k + (k/d) - \log\sqrt{d}}\leq \kappa_B\frac{2d}{k}. \]
	Consequently, it holds that for all sufficiently large \( k \),
	\[ \frac{1}{2d} \leq \frac{\dist(\bm{w}^k, \Ff)}{\mathfrak{g}_{\log} (\dist(\bm{w}^k, \Kld))} \leq \kappa _B. \]
	Similar to the argument in Remark~\ref{remark:tightness-F-d}, we conclude that the choice of \( \mathfrak{g}_{\log}  \) is tight.
\end{remark}

Using Theorem~\ref{thm:Lipschitz-error-bound-F-inf-n-y>0} and Theorem~\ref{thm:log-type-error-bound-F-inf} in combination with
\cite[Lemma~3.9]{LiLoPo20}, we deduce the following one-step facial residual function for \( \Kld \) and \( \bm{n} \).

\begin{corollary}\label{corollary:FRFs-F-inf}
	Let \( \bm{n} = (0, \bm{n}_y, \bm{n}_Z) \in \partial \Kld^{*} \) with \( \bm{n}_y \geq 0\) and \(\bm{n}_Z \succ 0 \) such that \( \Ff = \Kld \cap \{\bm{n}\}^{\perp }  \).
	\begin{enumerate}[(i)]
		\item If \( \bm{n}_y > 0 \), let \( \gamma _{\bm{n}, t } \) be as in~\eqref{eq:def-gamma} with $\stdFace = \Ff$ and  \( \mathfrak{g} = | \cdot  | \).
		      Then the function
		      \[ \psi _{\mathcal{K}, \bm{n}}(\epsilon, t) := \max \left\{\epsilon , \epsilon / \| \bm{n} \| \right\} + \max \left\{2, 2\gamma _{\bm{n}, t}^{-1}\right\} \left(\epsilon + \max\left\{\epsilon , \epsilon / \| \bm{n} \| \right\}\right) \]
		      is a one-step facial residual function for \( \Kld  \) and \( \bm{n} \).
		\item If \( \bm{n}_y = 0 \), let \( \gamma _{\bm{n}, t } \) be as in~\eqref{eq:def-gamma} with $\stdFace = \Ff$ and  \( \mathfrak{g}_{\log} \) defined in~\eqref{eq:g-log}.
		      Then the function
		      \[ \psi _{\mathcal{K}, \bm{n}}(\epsilon, t) := \max \left\{\epsilon , \epsilon / \| \bm{n} \| \right\} + \max \left\{2, 2\gamma _{\bm{n}, t}^{-1}\right\} \mathfrak{g}_{\log}\left(\epsilon + \max\left\{\epsilon , \epsilon / \| \bm{n} \| \right\}\right) \]
		      is a one-step facial residual function for \( \Kld  \) and \( \bm{n} \).
	\end{enumerate}
\end{corollary}

\subsubsection{\( \Fr  \): the family of 1-dimensional faces}\label{subsubsec:FRFs-F-r}
\begin{theorem}[H\"olderian error bound concerning \( \Fr  \)]\label{thm:Holderian-error-bound-F-r}
	Let \( \bm{n} = (\bm{n}_x, \bm{n}_x(\ldt (-\bm{n}_Z / \bm{n}_x ) + d ), \bm{n}_Z) \in \partial\Kld^{*} \) with \( \bm{n}_x < 0\) and \(\bm{n}_Z \succ 0 \) such that \( \Fr = \Kld \cap \{\bm{n}\}^{\perp } \).
	Let \( \eta > 0 \) and let \( \gamma _{\bm{n}, \eta } \) be defined as~\eqref{eq:def-gamma} with $\stdFace = \Fr$ and  \( \mathfrak{g} = | \cdot  | ^{\frac{1}{2}} \).
	Then \( \gamma _{\bm{n}, \eta } \in (0, \infty] \) and
	\begin{equation}\label{eq:Holderian-error-bound-F-r}
		\dist(\bm{q}, \Fr) \leq \max \{2\eta ^{\frac{1}{2}}, 2\gamma _{\bm{n}, \eta }^{-1}\}\cdot (\dist(\bm{q}, \Kld))^{\frac{1}{2}} \quad\quad \forall \bm{q} \in \{\bm{n}\}^{\perp } \cap B(\eta).
	\end{equation}
\end{theorem}

\begin{proof}
	If \( \gamma _{\bm{n}, \eta } = 0 \), in view of~\cite[Lemma 3.12]{LiLoPo20}, there exists \( \widehat{\bm{v}} \in \Fr \) and sequences \( \{\bm{v}^k\}, \{\bm{w}^k\}, \{\bm{u}^k\} \) being defined as those therein, with the cone being \( \Kld \) and the face being \( \Fr \), such that~\eqref{eq:for-contradiction} holds with \( \mathfrak{g} = | \cdot | ^{\frac{1}{2}} \).
	We consider two different cases.
	\begin{enumerate}[(i)]
		\item\label{item:vky=0} \( \bm{v}^k \in \Fd \) infinitely often, i.e., \( \bm{v}_y^k = 0 \) infinitely often (wherefore $\widehat{\bm{v}}=\mathbf{0}$);
		\item\label{item:vky>0} \( \bm{v}^k \notin \Fd \) for all large \( k \), i.e., \( \bm{v}_y^k > 0 \) for all large \( k \).
	\end{enumerate}

	(\ref{item:vky=0}) If \( \bm{v}_y^k = 0 \) infinitely often, by extracting a subsequence if necessary, we may assume that
	\[ \bm{v}^k = (\bm{v}_x^k, 0, \bm{v}_Z^k) \text{ with } \bm{v}_x^k \leq 0, \bm{v}_Z^k \succeq 0 \quad \text{for all }k. \]

	Combining this with the definition of \( \bm{n} \), we have
	\begin{align*}
		| \langle \bm{n}, \bm{v}^k \rangle | = & | \bm{n}_x \bm{v}_x^k + \tr(\bm{n}_Z\bm{v}_Z^k) | = -\bm{n}_x | \bm{v}_x^k | + \tr(\bm{n}_Z \bm{v}_Z^k) \geq -\bm{n}_x | \bm{v}_x^k | + \lambda _{\min }(\bm{n}_Z) \tr(\bm{v}_Z^k) \\
		\geq                                   & \min \{-\bm{n}_x , \lambda _{\min }(\bm{n}_Z)\} (| \bm{v}_x^k | + \tr(\bm{v}_Z^k) ) \geq \min \{-\bm{n}_x , \lambda _{\min }(\bm{n}_Z)\} \| \bm{v}^k \|.
	\end{align*}
	Here, we recall that \( \tr(\bm{n}_Z \bm{v}_Z^k) \geq 0, \lambda _{\min }(\bm{n}_Z) > 0, \tr(\bm{v}_Z^k) \geq 0 \).

	Since projections are non-expansive, we have \( \| \bm{w}^k \| \leq \| \bm{v}^k \|  \).
	Moreover, since \( \bm{0} \in \Fr \), we have \( \dist(\cdot, \Fr) \leq \| \cdot  \| \).
	Thus,
	\begin{align*}
		\| \bm{w}^k - \bm{u}^k \| = & \dist(\bm{w}^k, \Fr) \leq \| \bm{w}^k \| \leq \| \bm{v}^k \|                                                                                                                          \\
		\leq                        & \frac{1}{\min \{-\bm{n}_x , \lambda _{\min }(\bm{n}_Z) \}} | \langle \bm{n}, \bm{v}^k \rangle  | = \frac{\| \bm{n} \| }{\min \{-\bm{n}_x , \lambda _{\min }(\bm{n}_Z)\}}  \| \bm{w}^k - \bm{v}^k \|.
	\end{align*}
	This display shows that~\eqref{eq:for-contradiction} for \( \mathfrak{g} = | \cdot | \) does not hold in this case.
	Since $|t|^{1/2} \geq |t|$ holds for small $t > 0$,
	we conclude that \eqref{eq:for-contradiction} for \( \mathfrak{g} = | \cdot |^{1/2} \) does not hold as well.

	(\ref{item:vky>0}) If \( \bm{v}_y^k > 0 \) for all large \( k \), by passing to a subsequence if necessary, we can assume that
	\[ \bm{v}^k = (\bm{v}_y^k \ldt (\bm{v}_Z^k / \bm{v}_y^k), \bm{v}_y^k, \bm{v}_Z^k) \text{ with } \bm{v}_y^k > 0, \bm{v}_Z^k \succ 0, \quad \text{for all }k. \]
	Thus, we have
	\begin{equation}
		\label{eq:norm-wk-vk-F-r}
		\| \bm{w}^k - \bm{v}^k \| = \frac{| \langle \bm{n}, \bm{v}^k \rangle | }{\| \bm{n} \| },
	\end{equation}
	and
  \begin{align}
      & \quad \langle \bm{n}, \bm{v}^k \rangle = \bm{n}_x \bm{v}_y^k \ldt (\bm{v}_Z^k / \bm{v}_y^k) + \bm{n}_x \bm{v}_y^k (\ldt (-\bm{n}_Z / \bm{n}_x ) + d) + \tr(\bm{n}_Z \bm{v}_Z^k)                                                                                                                         \notag \\
      & = \bm{n}_x \bm{v}_y^k \left(\ldt \left( -\frac{\bm{v}_Z^k \bm{n}_Z }{\bm{v}_y^k \bm{n}_x }\right) + d + \tr\left(\frac{\bm{v}_Z^k \bm{n}_Z }{\bm{v}_y^k \bm{n}_x } \right) \right)                                                                                                                      \notag \\
      & = \bm{n}_x \bm{v}_y^k \left(\ldt \left( -\frac{\bm{n}_Z ^{\frac{1}{2}} \bm{v}_Z^k \bm{n}_Z^{\frac{1}{2}} }{\bm{v}_y^k \bm{n}_x }\right) + d + \tr\left(\frac{\bm{n}_Z ^{\frac{1}{2}} \bm{v}_Z^k \bm{n}_Z^{\frac{1}{2}} }{\bm{v}_y^k \bm{n}_x } \right) \right)                                          \notag \\
      & = \bm{n}_x \bm{v}_y^k \sum_{i=1}^d \left(\log\left(\lambda_i \left( -\frac{\bm{n}_Z ^{\frac{1}{2}} \bm{v}_Z^k \bm{n}_Z^{\frac{1}{2}} }{\bm{v}_y^k \bm{n}_x }\right)\right) + 1 + \lambda_i \left( \frac{\bm{n}_Z ^{\frac{1}{2}} \bm{v}_Z^k \bm{n}_Z^{\frac{1}{2}} }{\bm{v}_y^k \bm{n}_x }\right)\right) \notag \\
      & = \bm{n}_x \bm{v}_y^k \sum_{i=1}^d \left( t_i^k + 1 - e^{t_i^k}\right) \geq 0, \label{eq:form-nvkip-F-r}
  \end{align}
	where \( t_i^k := \log \left(\lambda_i \left( -\frac{\bm{n}_Z ^{\frac{1}{2}} \bm{v}_Z^k \bm{n}_Z^{\frac{1}{2}} }{\bm{v}_y^k \bm{n}_x }\right)\right) \) for \( i = 1, 2, \dots, d  \) and \( k \geq 1 \), and the nonnegativity comes from the observation that \( t + 1 - e^t \leq 0 \) for all \( t \in \R \) and  the facts that \( \bm{n}_x < 0 \) and \( \bm{v}_y^k > 0 \); recall that here \( \bm{v}_Z^k \succ 0, \bm{n}_Z \succ 0, \bm{n}_x < 0, \) and \( \bm{v}_y^k > 0\), then \( \lambda_i ( -\frac{\bm{n}_Z ^{\frac{1}{2}} \bm{v}_Z^k \bm{n}_Z^{\frac{1}{2}} }{\bm{v}_y^k \bm{n}_x }) > 0 \) for all \( i \), and hence \( t_i^k \) is well-defined.

	Next, we turn to compute \( \| \bm{w}^k - \bm{u}^k \|  \).
	Using Lemma~\ref{lemma:norm-wk-uk},~\eqref{eq:1d-face-F-r} and~\eqref{eq:def-f-r}, one can see for all \( k \),
	\begin{align*}
		 & \,\quad \| \bm{w}^k - \bm{u}^k \| \leq \dist( \bm{v}^k , \Fr) \overset{(\text{a})}{\leq } \| \bm{v}^k - \bm{v}_y^k \bm{f}_{{\rm r}} \|                       \\
		 & = \| ( \bm{v}_y^k \ldt (\bm{v}_Z^k / \bm{v}_y^k) - \bm{v}_y^k \ldt (- \bm{n}_x \bm{n}_Z^{-1}), 0, \bm{v}_Z^k + \bm{v}_y^k \bm{n}_x \bm{n}_Z ^{-1} )  \|                        \\
		 & \leq \bm{v}_y^k (|\ldt (- (\bm{n}_Z ^{\frac{1}{2}}\bm{v}_Z^k \bm{n}_Z^{\frac{1}{2}}) / (\bm{v}_y^k \bm{n}_x) )| + \| \bm{v}_Z^k / \bm{v}_y^k  + \bm{n}_x \bm{n}_Z ^{-1} \|_F ) \\
		 & \leq \bm{v}_y^k \left( \left(\sum_{i=1}^d | t_i^k |\right) + \| \bm{v}_Z^k / \bm{v}_y^k  + \bm{n}_x \bm{n}_Z ^{-1} \|_F \right),
	\end{align*}
	where (a) holds because \( \bm{v}_y^k \bm{f}_{{\rm r}} \in \Fr \). We remark that \( \bm{v}_Z^k / \bm{v}_y^k  + \bm{n}_x \bm{n}_Z^{-1} \) is a symmetric matrix.
	Let
	\[ A_k := \frac{\bm{v}_Z^k}{\bm{v}_y^k } + \bm{n}_x \bm{n}_Z^{-1}, \quad B := -\bm{n}_x \bm{n}_Z ^{-1}, \quad D_k := \frac{\bm{v}_Z^k \bm{n}_Z}{\bm{v}_y^k \bm{n}_x }, \quad \widehat{D}_k := \frac{\bm{n}_Z^{\frac{1}{2}} \bm{v}_Z^k \bm{n}_Z^{\frac{1}{2}}}{\bm{v}_y^k \bm{n}_x }. \]
	We notice that \( D_k = \bm{n}_Z ^{-\frac{1}{2}} \widehat{D}_k \bm{n}_Z ^{\frac{1}{2}} \) and \( e^{t_i^k} = \lambda _i(-\widehat{D}_k) \) for \( i = 1, 2, \dots, d \) and \( k \geq 1 \).
	Then, we have for all \( k \),\footnote{For \( X \in \R^{n \times n} \), we denote the nuclear norm and spectral norm of \( X \) by \( \| X \|_{*} \) and \( \| X \|_{2} \), respectively.}
	\begin{align*}
		 & \,\quad\| A_k \| _F \leq \| A_k \| _{*} = \| A_k (\bm{n}_Z / \bm{n}_x) (\bm{n}_x \bm{n}_Z^{-1}) \|_{*} = \| (D_k+I)B \| _{*}                                                                                                                                  \\
		 & \overset{(\text{a})}{=} \sup_{\| W \| _{2} \leq 1} \tr(W(D_k+I)B) = \sup_{\| W \| _{2} \leq 1} \tr\left(BW\left(\bm{n}_Z^{-\frac{1}{2}}\widehat{D}_k\bm{n}_Z^{\frac{1}{2}} + I\right)\right)                                                                              \\
		 & = \sup_{\| W \| _{2} \leq 1} \tr\left(\bm{n}_Z^{\frac{1}{2}} BW\bm{n}_Z^{-\frac{1}{2}}(\widehat{D}_k+I)\right) \overset{(\text{b})}{\leq} \| \widehat{D}_k+I \| _{*} \sup_{\| W\| _{2} \leq 1} \| \bm{n}_Z ^{\frac{1}{2}} BW \bm{n}_Z ^{-\frac{1}{2}} \| _{2} \\
		 & = \beta \sum_{i=1}^d | \lambda _i(\widehat{D}_k + I) |   = \beta \sum_{i=1}^d | \lambda _i(\widehat{D}_k) + 1 |   = \beta \sum_{i=1}^d | \lambda _i(-\widehat{D}_k) - 1 | = \beta \sum_{i=1}^d | e^{t_i^k} - 1 |,
	\end{align*}
	where \( \beta := \sup_{\| W \| _{2} \leq 1} \| \bm{n}_Z ^{\frac{1}{2}}BW\bm{n}_Z ^{-\frac{1}{2}} \|_{2} \in (0, \infty) \), and (a) and (b) hold since the dual norm of nuclear norm \( \| \cdot \| _{*}\) is the spectral norm \( \| \cdot \| _{2}\).
	Hence, we obtain that for all \( k \),
	\begin{equation}
    \label{eq:F-r-case-II}
    \| \bm{w}^k - \bm{u}^k \| \leq \bm{v}_y^k \left(\sum_{i=1}^d | t_i^k | + \beta | e^{t_i^k} - 1 |  \right).
  \end{equation}

  Before moving on, we define two auxiliary functions and discuss some useful properties.
  Define
  \begin{equation}
    \label{eq:def-h-t}
    h(t) :=  t + 1 - e^t \quad \text{ and } \quad  g(t) := | t | + \beta | e^t - 1 |.
  \end{equation}
  We observe that
  \begin{equation}
    \label{eq:property-h}
    \begin{aligned}
    h(t) = 0 \Longleftrightarrow t = 0,\\
     \lim_{t \to \infty} h(t)=\lim_{t \to -\infty} h(t) = -\infty,\\
     h'(t) = 1 - e^t,\quad h'(0) = 0,\\
      h''(t) = -e^t,\quad h''(0) = -1.
    \end{aligned}
  \end{equation}
  In addition, \( g(t) \geq 0 \) for all \( t \in \R \) and \( g(t) = 0 \) if and only if \( t = 0 \).

	Now, recall from the setting of \( \{\bm{v}^k\} \) that \( \bm{v}^k \to \widehat{\bm{v}} \) and \( \langle \bm{n}, \bm{v}^k \rangle \to 0 \).
	This and the formula of \( \langle \bm{n}, \bm{v}^k \rangle \) in \eqref{eq:form-nvkip-F-r} reveal that we need to consider the following two cases:
	\begin{enumerate}[(I)]
		\item\label{item:v-bar-y>0} \( \liminf _{k \to \infty } \sum_{i=1}^d h(t_i^k) = 0 \);
		\item\label{item:v-bar-y=0} \( \liminf _{k \to \infty } \sum_{i=1}^d h(t_i^k) \in [-\infty, 0) \).
	\end{enumerate}

  For notational simplicity, we define \( \bm{t}^k := (t_i^k)_{i=1}^d \) for all \( k \).

	(\ref{item:v-bar-y>0})
  Without loss of generality, by passing to a further subsequence, we assume that \( \lim_{k \to \infty} \sum_{i=1}^d h(t_i^k) = 0 \).
	Combining this assumption and the fact that \( h(t) \leq 0 \) for all \( t \in \R \) with \eqref{eq:property-h}, we know that \( \bm{t}^k \to \bm{0} \).
	Now, consider the Taylor expansion of \( h(t) \) at \( t = 0 \), that is,
	\[ h(t) = -0.5t^2 + O(|t|^3), \quad t \to 0. \]
	It follows that there exists \( \epsilon > 0 \) such that for any \( t \) satisfying \( | t | < \epsilon \), \( h(t) \leq -0.25t^2 \leq 0 \).
	Thus, we have for all large \( k \) that,
	\begin{equation}\label{eq:Taylor-expansion-inequality}
		0 \leq \sum_{i=1}^d 0.25(t_i^k)^2 \leq \sum_{i=1}^d | h(t_i^k) |.
	\end{equation}

  We can deduce the lower bound of \( \| \bm{w}^k - \bm{v}^k \| ^{\frac{1}{2}} \) for sufficiently large \( k \) as follows:
	\begin{align*}
		 & \,\quad \| \bm{w}^k - \bm{v}^k \| ^{\frac{1}{2}} \overset{(\text{a})}{=} \frac{| \langle \bm{n}, \bm{v}^k \rangle  | ^{\frac{1}{2}}}{\| \bm{n} \| ^{\frac{1}{2}}} \overset{(\text{b})}{=} \frac{| \bm{n}_x  | ^{\frac{1}{2}} | \bm{v}_y^k  | ^{\frac{1}{2}}}{\| \bm{n} \| ^{\frac{1}{2}}} \left( \sum_{i=1}^d | h(t_i^k) | \right)^{\frac{1}{2}}  \\
		 & \overset{(\text{c})}{\geq } \frac{| \bm{n}_x  | ^{\frac{1}{2}} | \bm{v}_y^k  | ^{\frac{1}{2}}}{2\| \bm{n} \| ^{\frac{1}{2}}} \left(\sum_{i=1}^d (t_i^k)^2 \right)^{\frac{1}{2}} \overset{(\text{d})}{\geq } \frac{(| \bm{n}_x  | | \bm{v}_y^k |) ^{\frac{1}{2}}}{2(d\| \bm{n} \|) ^{\frac{1}{2}}} \left(\sum_{i=1}^d | t_i^k| \right),
	\end{align*}
	where (a) comes from~\eqref{eq:norm-wk-vk-F-r}, (b) comes from~\eqref{eq:form-nvkip-F-r} and~\eqref{eq:def-h-t}, (c) holds by~\eqref{eq:Taylor-expansion-inequality}, (d) comes from the root-mean inequality.

	Next, to derive a bound for \( \| \bm{w}^k - \bm{u}^k \| \), we shall relate \( |e^{t_i^k} - 1| \) to \( |t_i^k| \).
	To this end, notice that \( \lim_{t \to  0} (e^t - 1) / t = 1. \)
	Then, there exists \( C_1 > 0 \) such that for any \( i = 1, 2, \dots , d \),
	\[ | e^{t_i^k} - 1 | \leq C_1 | t_i^k | \quad \text{ for sufficiently large } k. \]
	Therefore, by \eqref{eq:F-r-case-II}, for all sufficiently large \( k \),
	\[ \| \bm{w}^k - \bm{u}^k \| \leq \bm{v}_y^k (\beta C_1 + 1) \sum_{i=1}^d | t_i^k |.  \]
	We thus conclude that
	\[
    \begin{split}
      \lim_{k \to \infty} \frac{\| \bm{w}^k - \bm{v}^k \| ^{\frac{1}{2}}}{\| \bm{w}^k - \bm{u}^k \| } \geq & \liminf_{k \to \infty} \frac{(| \bm{n}_x  | | \bm{v}_y^k  |) ^{\frac{1}{2}}}{ 2(d\| \bm{n} \| )^{\frac{1}{2}}} \frac{(\sum_{i=1}^d | t_i^k | )}{\bm{v}_y^k (\beta C_1 + 1) (\sum_{i=1}^d | t_i^k |) } \\
      \overset{(\text{a})}{\geq} & \frac{| \bm{n}_x  | ^{\frac{1}{2}}}{2(d\| \bm{n} \| \eta )^{\frac{1}{2}} (\beta C_1 + 1)} > 0,
    \end{split}
	\]
  where (a) holds since \( 0 < \bm{v}_y^k \leq \| \bm{v}^k \| \leq \eta \).
	This contradicts~\eqref{eq:for-contradiction} with \( | \cdot |^{\frac{1}{2}} \) in place of \( \mathfrak{g} \) and hence this case cannot happen.

	(\ref{item:v-bar-y=0})
	In this case, in view of \eqref{eq:property-h}, by passing to a further subsequence if necessary, we can assume that there exist \( \epsilon > 0 \) and \( i_0 \) such that \( |t_{i_0}^k| \geq \epsilon \) for all large \( k \), that is, \( g(t_{i_0}^k) > 0 \) for all large \( k \).
	Then, \( \| \bm{t}^k \| _{\infty } \geq \epsilon  \) for all large \( k \).
	Now, consider the following function
	\[ H(\bm{t}) :=
		\begin{cases}
			\frac{\sum_{i=1}^d | h(t_i) |  }{\sum_{i=1}^d g(t_i) } & \text{ if } \| \bm{t} \| _{\infty } \geq \epsilon, \\
			\infty                                                 & \text{otherwise},
		\end{cases}
	\]
	where \( h \) is defined as in~\eqref{eq:def-h-t}.
	Since \( \| \bm{t} \| _{\infty } \geq \epsilon  \) implies \( g(t_i) > 0 \) for some \( i \), we see that \( H \) is well-defined.
	Moreover, one can check that \( H \) is lower semi-continuous and never zero.

	We claim that \( \inf H > 0 \).
	Granting this, we have
	\begin{align*}
		 & \,\quad \lim_{k \to \infty} \frac{\| \bm{w}^k - \bm{v}^k \|^{\frac{1}{2}} }{\| \bm{w}^k - \bm{u}^k \| } \geq \liminf_{k \to \infty} \frac{\| \bm{w}^k - \bm{v}^k \| }{\| \bm{w}^k - \bm{u}^k \| } \overset{(\text{a})}{\ge} \liminf_{k \to \infty} \frac{| \bm{n}_x  | }{\| \bm{n} \| } \frac{\sum_{i=1}^d | h(t_i^k) |  }{\sum_{i=1}^d g(t_i^k)} \\
		 & \overset{(\text{b})}{=} \liminf_{k\to \infty } \frac{| \bm{n}_x  | }{\| \bm{n} \| } H(\bm{t}^k) \geq\frac{| \bm{n}_x  | }{\| \bm{n} \| } \inf H > 0,
	\end{align*}
  where (a) comes from \eqref{eq:norm-wk-vk-F-r}, \eqref{eq:form-nvkip-F-r}, \eqref{eq:F-r-case-II} and the definition of \( h \) and \( g \) in \eqref{eq:def-h-t}, (b) holds thanks to the definition of \( H \).
	The above display contradicts~\eqref{eq:for-contradiction} with \( | \cdot |^{\frac{1}{2}} \) in place of \( \mathfrak{g} \) and hence this case cannot happen.
	Therefore, we obtain that \( \gamma _{\bm{n}, \eta } \in (0,\infty] \) with \( \mathfrak{g} = | \cdot  | ^{\frac{1}{2}} \).
	Together with \cite[Theorem 3.10]{LiLoPo20}, we deduce that~\eqref{eq:Holderian-error-bound-F-r} holds.

	Now it remains to show that \( \inf H > 0 \).
	Since \( H \) is lower semi-continuous and never zero, it suffices to show that \( \liminf_{\|\bm{t}\| \to \infty} H(\bm{t}) > 0 \); see this footnote.\footnote{Suppose that $\inf H = 0$. Then, there exists a sequence $\{\bm{\zeta}^l\}$ such that $ H(\bm{\zeta}^l) \to 0$. If $\{\bm{\zeta}^l\}$ is unbounded, we can find a subsequence $\{\bm{\zeta}^{l_k}\}$ such that $\|\bm{\zeta}^{l_k}\|\to \infty$ and $ H(\bm{\zeta}^{l_k}) \to 0$ holds, which would contradict $\liminf_{\|\bm{t}\| \to \infty} H(\bm{t}) > 0$. So $\{\bm{\zeta}^l\}$ must be bounded and passing to a subsequence we may assume it converges to some $\bar{\bm{\zeta}}$. By lower semicontinuity, we have $H(\bar{\bm{\zeta}}) \leq \lim \inf_{\bm{t} \to \bar{\bm{\zeta}}} H(\bm{t}) \leq \lim _{l \to \infty}H(\bm{\zeta}^l)=  0$. However, $H$ is always positive, so this cannot happen either. Therefore, $ \liminf_{\|\bm{t}\| \to \infty} H(\bm{t}) > 0 $ implies $\inf H > 0$.}

	To this end, consider a sequence \( \{\bm{\zeta}^l\} \) such that \( \| \bm{\zeta}^l \| \to \infty  \) and
	\[ \lim_{l \to \infty} H(\bm{\zeta}^l) = \liminf_{\| \bm{t} \| \to \infty } H(\bm{t}), \]
	then there exists at least one \( i_0 \in \{1, 2, \dots ,d\} \) such that \( |\zeta_{i_0}^l| \to \infty \). Consequently, \( | h(\zeta_{i_0}^l) | \to \infty\) and \( g(\zeta_{i_0}^l) \to \infty  \), and so both \( \sum_{i=1}^d | h(\zeta_i^l) | \) and \( \sum_{i=1}^d g(\zeta_i^l)  \) tend to \( \infty  \).
	Passing to a subsequence, we can assume that for each \( i \), \( \lim_{l \to \infty} \zeta_i^l \in  [-\infty, \infty] \) exists and we can split \( \bm{\zeta}^l \) into three parts:
	\begin{enumerate}[(1)]
		\item \( \zeta_i^l \to \overline{\zeta}_i \in \R \setminus \{0\} \), then \( | h(\zeta_i^l) | \to | h(\overline{\zeta}_i) | \neq 0 \), \( g(\zeta_i^l) \to g(\overline{\zeta}_i) \neq 0 \).
		      Denote the set of indices of these components by \( \mathcal{I}_{\bm{\zeta}}^C \) where \( C \) refers to ``constant''.

		      For any \( i \in \mathcal{I}_{\bm{\zeta}}^C \), we have
		      \[ \lim_{l \to \infty} \frac{| h(\zeta_i^l) | }{g(\zeta_i^l)} = \frac{| h(\overline{\zeta}_i)| }{g(\overline{\zeta}_i)} > 0. \]
		      Thus, there exists a constant \( C_C > 0 \) such that for all sufficiently large \( l \) and all \( i \in \mathcal{I}_{\bm{\zeta}}^C \),
		      \[ | h(\zeta_i^l) | \geq C_C g(\zeta_i^l). \]
		\item \( \zeta_i^l \to 0 \), then \( | h(\zeta_i^l) | \to 0 \), \( g(\zeta_i^l) \to 0 \).
		      Denote the set of indices of these components by \( \mathcal{I}_{\bm{\zeta}}^0 \).
		\item \( |\zeta_i^l| \to \infty  \), then \( | h(\zeta_i^l) | \to \infty \), \( g(\zeta_i^l) \to \infty  \).
		      Denote the set of these components by \( \mathcal{I}_{\bm{\zeta}}^{\infty } \).
		      We have \( \mathcal{I}_{\bm{\zeta}}^{\infty } \neq \emptyset \), since otherwise \( \| \bm{\zeta}^l\|  \not\to \infty  \).

		      For any \( i \in \mathcal{I}_{\bm{\zeta}}^{\infty } \), we  notice that
		      \[
\liminf_{l \to \infty} \frac{| h(\zeta_i^l)| }{g(\zeta_i^l)}\ge \min\left\{\liminf_{t\to-\infty}\frac{|h(t)|}{g(t)},\liminf_{t\to\infty}\frac{|h(t)|}{g(t)}\right\}
= \min \left\{1, \frac{1}{\beta }\right\} := \widehat{\beta } > 0.
\]
		      Thus, for all sufficiently large \( l \) and all \( i \in \mathcal{I}_{\bm{\zeta}}^C \),
		      \[ | h(\zeta_i^l) | \geq \frac{\widehat{\beta }}{2} g(\zeta_i^l). \]
	\end{enumerate}
	Combining the above three cases, we obtain
	\begin{align}
			 & \,\quad \liminf_{\| \bm{t} \|  \to \infty } H(\bm{t}) = \lim_{l \to \infty} H(\bm{\zeta}^l)                                                                                                                                                                                                                                                                                \notag\\
			 & \geq \lim_{l \to \infty } \frac{C_C\sum_{i \in \mathcal{I}_{\bm{\zeta}}^C} g(\zeta_i^l) + \frac{\widehat{\beta }}{2} \sum_{i \in \mathcal{I}_{\bm{\zeta}}^{\infty }} g(\zeta_i^l) + \sum_{i \in \mathcal{I}_{\bm{\zeta}}^0} g(\zeta_i^l)}{\sum_{i=1}^d g(\zeta_i^l) } + \frac{\sum_{i \in \mathcal{I}_{\bm{\zeta}}^0} | h(\zeta_i^l) | - \sum_{i \in \mathcal{I}_{\bm{\zeta}}^0} g(\zeta_i^l) }{\sum_{i=1}^d g(\zeta_i^l)} \notag\\
			 & \overset{(\text{a})}{\geq } \min \left\{C_C, \frac{\widehat{\beta }}{2}, 1\right\} > 0,\label{eq:lim-H>0}
  \end{align}
	where (a) comes from the fact that
	\[ \lim_{l \to \infty} \frac{\sum_{i \in \mathcal{I}_{\bm{\zeta}}^0} | h(\zeta_i^l) | - \sum_{i \in \mathcal{I}_{\bm{\zeta}}^0} g(\zeta_i^l) }{\sum_{i=1}^d g(\zeta_i^l)} = 0, \]
	which holds because the numerator tends to 0 while the denominator tends to infinity.
%  Finally, since \( H \) is lower semicontinuous and is never zero, we conclude from \eqref{eq:lim-H>0} that \( \inf H > 0 \).
\end{proof}

\begin{remark}[Tightness of~\eqref{eq:Holderian-error-bound-F-r}]\label{remark:tightness-F-r}
	Let \( \bm{n} \) be defined as in Proposition \ref{prop:facial-structure-log-det-cone}.(a)) and
	\[ \bm{v}^k = \left( \ldt (-\bm{n}_x \bm{n}_Z^{-1}) + \frac{1}{k}, 1, e^{\frac{1}{dk}} (-\bm{n}_x \bm{n}_Z^{-1}) \right), \quad \bm{w}^k = P_{\{\bm{n}\}^{\perp } } (\bm{v}^k), \quad \bm{u}^k = P_{\Fr}(\bm{w}^k), \]
	so that \( \Fr = \Kld \cap \{\bm{n}\}^{\perp} \), \( \{\bm{v}^k\} \subset \Kld \) and there exists \( \eta > 0 \) such that \( \{\bm{w}^k\} \subseteq  B(\eta) \).
	Then we have
	\[ \| \bm{w}^k - \bm{v}^k \| = \frac{| \langle \bm{n}, \bm{v}^k \rangle | }{\| \bm{n} \| } = \frac{-\bm{n}_x | \frac{1}{k} + d - de^{\frac{1}{dk}} | }{\| \bm{n} \| }. \]
  For the sake of notational simplicity, we denote \( \xi := \frac{1}{k} \), then
  \begin{equation}\label{eq:form-wk-vk-F-r}
    \| \bm{w}^k - \bm{v}^k \| = \frac{-\bm{n}_x| \xi + d - de^{\frac{\xi}{d}} |}{\| \bm{n} \|}.
  \end{equation}
  Consider the Taylor expansion of \( \xi + d - de^{\frac{\xi}{d}} \) with respect to \( \xi \) at \( 0 \), we have
  \begin{equation}
    \label{eq:Taylor-w-v-xi}
    \xi + d - de^{\frac{\xi}{d}} = \xi + d - d \left( 1 + \frac{\xi}{d} + \frac{\xi^2}{2d^2} \right) + o(\xi^2) = -\frac{\xi^2}{2d} + o(\xi^2), \text{ as } \xi \to 0.
  \end{equation}

	Next, upon invoking the definitions of \( \Fr  \) and \( \bm{f}_{{\rm r}} \) (see~\eqref{eq:1d-face-F-r} and~\eqref{eq:def-f-r}, respectively), we can see that
	\[
		\begin{split}
			&\,\quad \| \bm{v}^k - \bm{u}^k \| ^2 = \dist^2(\bm{v}^k, \Fr) = \min _{y\geq 0} \| \bm{v}^k - y\bm{f}_{{\rm r}} \|^2 \\
			&=\min _{y\geq 0} \left\| \left( (1-y) \ldt (-\bm{n}_x \bm{n}_Z^{-1}) + \xi, 1-y, -(e^{\frac{\xi}{d}} -y)\bm{n}_x  \bm{n}_Z ^{-1} \right)  \right\| ^2\\
      &=\min _{y\geq 0} \Bigg\{ \underbrace{ \left[ (1 - y)\log \det (-\bm{n}_x \bm{n}_Z^{-1}) + \xi \right]^2 + (1 - y)^2 + (e^{\frac{\xi}{d}} - y)^2 \bm{n}_x^2 \| \bm{n}_Z^{-1} \|_F^2}_{F(y)} \Bigg\} \\
		\end{split}
	\]
	For the sake of brevity, we denote \( \mu := \ldt (-\bm{n}_x \bm{n}_Z^{-1})\) and \( \nu  := \bm{n}_x ^2 \| \bm{n}_Z^{-1}  \|_F ^2 \).
  Then we have
  \[
   F(y) = \left( \mu (y - 1) - \xi  \right)^2 + (y - 1)^2 + \nu ( y - e^{\frac{\xi}{d}})^2.
  \]
  Noting
  \[
    \begin{split}
      F'(y) &= 2\mu \left( \mu (y - 1) - \xi  \right) + 2(y - 1) + 2\nu (y - e^{\frac{\xi}{d}}) \\
      & = (2\mu^2 + 2 + 2\nu )y - (2\mu^2 + 2 + 2\mu \xi  + 2\nu e^{\xi / d}),
    \end{split}
  \]
	we know that \( F \) attains its minimum at \( \bar{y} = \frac{\mu^2 + 1 + \mu \xi  + \nu e^{\xi / d}}{\mu^2 + \nu  + 1} \), which is larger than 0 for sufficiently large \( k \) (or, equivalently, sufficiently small \( \xi \)).

  Next we move towards the analysis of \( \| \bm{v}^k - \bm{u}^k \|^2 = F(\bar{y}) \).
  Consider the Taylor expansion of \( \bar{y} - 1 \) and \( \bar{y} - e^{\frac{\xi}{d}} \) with respect to \( \xi \) at \( 0 \), we have that
  \[
   \bar{y} - 1 = \frac{\mu^2 + 1 + \mu \xi  + \nu e^{\xi / d}}{\mu^2 + \nu  + 1} - 1 = \frac{\mu \xi + \nu (e^{\frac{\xi}{d}} - 1)}{\mu ^2 + \nu + 1} = \frac{\mu \xi + \nu \frac{\xi}{d} }{\mu ^2 + \nu + 1} + o(\xi),
  \]
  and
  \[
   \bar{y} - e^{\frac{\xi}{d}} = \bar{y} - 1 - \frac{\xi}{d} + o(\xi) = \left( \frac{\mu + \frac{\nu}{d}}{\mu^2 + \nu + 1} - \frac{1}{d} \right)\xi + o(\xi) = -\left( \frac{\frac{\mu^2}{d} - \mu + \frac{1}{d}}{\mu^2 + \nu + 1} \right)\xi + o(\xi).
  \]
	Then, for all sufficiently large \( k \),
  \begin{align}
    & \| \bm{v}^k - \bm{u}^k \|^2 = F(\bar{y}) \notag\\
    = & \left( \mu (\bar{y} - 1) - \xi  \right)^2 + (\bar{y} - 1)^2 + \nu ( \bar{y} - e^{\frac{\xi}{d}})^2 \notag\\
    = & \left( \frac{\mu^2 + \mu\frac{\nu}{d}}{\mu^2 + \nu + 1} - 1 \right)^2 \xi^2 + \left( \frac{\mu + \frac{\nu}{d}}{\mu^2 + \nu + 1} \right)^2\xi^2 + \nu \left( \frac{\frac{\mu^2}{d} - \mu + \frac{1}{d}}{\mu^2 + \nu + 1} \right)^2 \xi^2 + o(\xi^2) \notag\\
    = & \Bigg[ \underbrace{\left( \frac{\nu(1 - \frac{\mu}{d}) + 1}{\mu^2 + \nu + 1} \right)^2 + \left( \frac{\mu + \frac{\nu}{d}}{\mu^2 + \nu + 1} \right)^2 + \nu \left( \frac{\frac{\mu^2}{d} - \mu + \frac{1}{d}}{\mu^2 + \nu + 1} \right)^2}_{C_{{\rm r}} \geq 0} \Bigg]\xi^2 + o(\xi^2).  \label{eq:F-r-dist-lower-bound}
  \end{align}
  Next we show \( C_{{\rm r}} > 0 \).
  Suppose that \( \frac{\mu^2}{d} - \mu + \frac{1}{d} = 0 \), then\footnote{Note that this quadratic in $\mu$ has real roots because $d\ge 2$; see the discussions following \eqref{eq:exponential-form-log-det-cone-dual}.} \( \mu = \frac{d \pm \sqrt{d^2 - 4}}{2} > 0 \) and \( \mu (1 - \frac{\mu}{d}) = \frac{1}{d} \).
  This implies that \( 1 - \frac{\mu}{d} > 0 \) and hence \( \nu (1 - \frac{\mu}{d}) + 1 > 0 \) thanks to \( \nu > 0 \).
  Therefore, we can see that \( C_{{\rm r}} > 0 \) because either \( \frac{\mu^2}{d} - \mu + \frac{1}{d} \neq 0 \) or \( \nu (1 - \frac{\mu}{d}) + 1 \neq 0 \).

  Using  Lemma \ref{lemma:norm-wk-uk}, \eqref{eq:form-wk-vk-F-r}, \eqref{eq:Taylor-w-v-xi} and \eqref{eq:F-r-dist-lower-bound}, we deduce that
  \begin{align}
    &\,\quad L_{{\rm r}} := \lim_{k \to \infty} \frac{\| \bm{w}^k - \bm{u}^k \|}{\| \bm{w}^k - \bm{v}^k \|^{\frac{1}{2}}} = \lim_{k \to \infty} \frac{\sqrt{\| \bm{v}^k - \bm{u}^k \|^2 - \| \bm{w}^k - \bm{v}^k \|^2}}{\| \bm{w}^k - \bm{v}^k \|^{\frac{1}{2}}} \notag \\
    &= \lim_{k \to \infty} \frac{\| \bm{n} \| ^{\frac{1}{2}}}{\left(-\bm{n}_x| \frac{1}{k} + d - de^{\frac{1}{dk}} | \right)^{\frac{1}{2}}} \frac{\sqrt{\| \bm{n} \| ^2 \| \bm{v}^k - \bm{u}^k \| ^2 - ( \frac{1}{k} + d - de^{\frac{1}{dk}} )^2 \bm{n}_x ^2}}{\| \bm{n} \| }\notag \\
    &= \frac{1}{\| \bm{n} \| ^{\frac{1}{2}}} \lim_{k \to \infty} \sqrt{\frac{\| \bm{n} \|^2 }{-\bm{n}_x } \frac{\| \bm{v}^{k} - \bm{u}^k \|^2 }{ | \frac{1}{k} + d - d e^{\frac{1}{dk}} | } + \Big|\frac{1}{k} + d - de^{\frac{1}{dk}}\Big|\cdot \bm{n}_x } \notag \\
    &= \frac{1}{\| \bm{n} \|^{\frac{1}{2}}} \lim_{\xi \to 0} \sqrt{\frac{\| \bm{n} \|^2}{-\bm{n}_x} \frac{C_{{\rm r}}\xi^2 + o(\xi^2)}{\frac{\xi^2}{2d} + o(\xi^2)} + \bm{n}_x \left( \frac{\xi^2}{2d} + o(\xi^2) \right)} \notag\\
    &= \frac{1}{\| \bm{n} \|^{\frac{1}{2}}} \sqrt{\frac{2\| \bm{n} \|^2C_{{\rm r}}d}{-\bm{n}_x}} = \sqrt{\frac{2\| \bm{n} \|C_{{\rm r}}d}{-\bm{n}_x}} > 0. \label{eq:L-r-tightness-F-r}
  \end{align}
%  \todo[inline]{B:In the third display, can someone check whether it should be }

	By contrast, applying~\eqref{eq:Holderian-error-bound-F-r}, there exists \( \kappa _B > 0 \) such that
	\[ \| \bm{w}^k - \bm{u}^k \| = \dist(\bm{w}^k, \Fr) \leq \kappa _B \dist(\bm{w}^k, \Kld)^{\frac{1}{2}} \leq \kappa _B \| \bm{w}^k - \bm{v}^k \| ^{\frac{1}{2}}. \]
  This shows that \( L_{{\rm r}} \leq \kappa_B < \infty \).
	Moreover, from~\eqref{eq:L-r-tightness-F-r}, for large enough \( k \), we have \( \| \bm{w}^k - \bm{u}^k \| \geq \frac{L_{{\rm r}}}{2} \| \bm{w}^k - \bm{v}^k \| ^{\frac{1}{2}} \).
	Therefore, for sufficiently large \( k \), we have
	\[ \frac{L_{{\rm r}}}{2} \| \bm{w}^k - \bm{v}^k \| ^{\frac{1}{2}} \leq \dist(\bm{w}^k , \Fr) \leq \kappa _B \dist(\bm{w}^k, \Kld)^{\frac{1}{2}} \leq \kappa _B \| \bm{w}^k - \bm{v}^k \| ^{\frac{1}{2}}. \]
	Consequently, it holds that for all large enough \( k \),
	\[ \frac{L_{{\rm r}}}{2} \leq \frac{\dist(\bm{w}^k, \Fr)}{\dist(\bm{w}^k, \Kld)^{\frac{1}{2}}} \leq \kappa _B. \]
	Similar to the argument in Remark~\ref{remark:tightness-F-d}, we conclude that the choice of \( | \cdot  | ^{\frac{1}{2}} \) is tight.
\end{remark}

By Theorem~\ref{thm:Holderian-error-bound-F-r}, we have the following one-step facial residual function for \( \Kld \) and \( \bm{n} \).

\begin{corollary}\label{corollary:FRFs-F-r}
	Let \( \bm{n} = (\bm{n}_x, \bm{n}_x(\ldt (-\bm{n}_Z / \bm{n}_x ) + d ), \bm{n}_Z) \in \partial\Kld^{*} \) with \( \bm{n}_x < 0\) and \(\bm{n}_Z \succ 0 \) such that \( \Fr = \Kld \cap \{\bm{n}\}^{\perp }  \).
	Let \( \gamma _{\bm{n}, t } \) be as in~\eqref{eq:def-gamma} with $\stdFace = \Fr$ and \( \mathfrak{g} = | \cdot  | ^{\frac{1}{2}} \).
	Then the function \( \psi _{\mathcal{K}, \bm{n}} : \R_{+} \times \R_{+} \to \R_{+} \) defined by
	\[ \psi _{\mathcal{K}, \bm{n}} (\epsilon, t) := \max \left\{\epsilon, \epsilon / \| \bm{n} \| \right\} + \max \left\{2t ^{\frac{1}{2}}, 2\gamma _{\bm{n}, t}^{-1}\right\} \left(\epsilon  + \max\left\{\epsilon , \epsilon / \| \bm{n} \| \right\}\right)^{\frac{1}{2}} \]
	is a one-step facial residual function for \( \Kld  \) and \( \bm{n} \).
\end{corollary}

\subsection{Error bounds}\label{subsec:error-bounds}
In this subsection, we combine all the previous analysis to deduce the error bound concerning~\eqref{eq:conic-feasibility-problem} with \( \mathcal{K} = \Kld \).
We proceed as follows.

We consider ~\eqref{eq:conic-feasibility-problem} with \( \mathcal{K} = \Kld \) and we suppose~\eqref{eq:conic-feasibility-problem} is feasible.
We also let \( \mathfrak{d} := \dpps(\Kld, \mathcal{L} + \bm{a}) \), where we recall that $\dpps$ denotes the distance to the PPS condition, i.e., the minimum number of facial reduction steps necessary to find a face $\stdFace$ such that $\stdFace$ and ${\cal L}+{\bm a}$ satisfy the PPS condition; see \cite[Section~2.4.1]{L17}.

In particular,  invoking \cite[Proposition~5]{L17}, there exists a chain of faces
\begin{equation}\label{eq:chain}
\mathcal{F}_{\mathfrak{d} + 1} \subsetneq \mathcal{F}_{\mathfrak{d}} \subsetneq \dots \subsetneq \mathcal{F}_2 \subsetneq \mathcal{F}_1 = \Kld
\end{equation}
together with $\bm{n}^1, \ldots, \bm{n}^{\mathfrak{d}} $ satisfying the following properties:
\begin{enumerate}[{\rm (a)}]
	\item For all $i \in \{1,\ldots, \mathfrak{d}\}$ we have
	\begin{align*}
	\bm{n}^{i} \in \mathcal{F}_i^* \cap \mathcal{L}^\perp \cap \{\bm{a}\}^\perp\ \ {\rm and}\ \
	\mathcal{F}_{i+1} = \mathcal{F}_{i} \cap \{\bm{n}^i\}^\perp.
	\end{align*}
		\item $\mathcal{F}_{\mathfrak{d} + 1}\cap( \mathcal{L} + \bm{a}) = \Kld\cap ( \mathcal{L} + \bm{a})$ and
	$\mathcal{F}_{\mathfrak{d} + 1}$ and $\mathcal{L}+\bm{a}$ satisfy the PPS condition.
\end{enumerate}

In order to get the final error bound for \eqref{eq:conic-feasibility-problem} we aggregate the one-step facial residual functions for each of $\mathcal{F}_i$ and $\bm{n}^i$ using the recipe described in \cite[Theorem~3.8]{LiLoPo20}.

So far, we only computed facial residual functions for $\mathcal{F}_1=\Kld$ and
$\bm{n}^1 \in \Kld^*$, but we need the ones for the other $\mathcal{F}_i$ and $\bm{n}^i$.
Fortunately, thanks to the facial structure of
$\Kld$, if $\mathfrak{d} \geq 2$, then $\mathcal{F}_2$ must be a face of the form $\Fd$ or $\Fs$ (see \eqref{eq:d-dim-face-F-d} and \eqref{eq:n-dim-faces-F-n}).
This is because all other possibilities correspond to non-exposed faces or faces of dimension $1$ (for which the PPS condition is automatically satisfied).

$\Fs$ and $\Fd$ are symmetric cones \cite{FK94,FB08} since they are linearly isomorphic to a direct product of $\R_-$ and a face of a positive semidefinite cone (which are symmetric cones on their own right, e.g., \cite[Proposition~31]{L17}).
The conclusion is that for the faces ``down the chain'' we can compute the one-step facial residual functions using the general result for symmetric cones given in \cite[Theorem~35]{L17}.
We note this as a lemma.
\begin{lemma}\label{lem:symmetric}
Let $\bar{\mathcal{F}}$ be a face of $\Fd$.  Let
$\bm{n} \in \bar{\mathcal{F}}^*\cap{\mathcal{L}}^\perp \cap \{\bm{a}\}^\perp$. 	
Then, there exists a constant $\kappa > 0$ such that the function
\[ \psi _{\bar{\mathcal{F}}, \bm{n}} (\epsilon, t) \coloneqq
\kappa \epsilon + \kappa \sqrt{\epsilon t}
\]
is a one-step facial residual function for \( \bar{\mathcal{F}}  \) and \( \bm{n} \).
\end{lemma}
\begin{proof}
Follows by  invoking \cite[Theorem~35]{L17} with  $\stdCone \coloneqq \bar{\mathcal{F}}$, $\mathcal{F}\coloneqq \bar{\mathcal{F}}$ and
	$z \coloneqq \bm{n}$.
\end{proof}
%One final remark is that by invoking
%\cite[Theorem~3.8]{LiLoPo20}, we get an error bound of the following shape.

We are now positioned to prove our main result in this paper.
\begin{theorem}[Error bounds for~\eqref{eq:conic-feasibility-problem} with \( \mathcal{K} = \Kld \)]\label{thm:error-bounds-log-det-cone}
	Consider~\eqref{eq:conic-feasibility-problem} with \( \mathcal{K} = \Kld \).
	Suppose~\eqref{eq:conic-feasibility-problem} is feasible and let \( \mathfrak{d} := \dpps(\Kld, \mathcal{L} + \bm{a}) \) and consider a chain of faces as in \eqref{eq:chain}. Then \( \mathfrak{d} \leq \min \{d - 1, \dim (\mathcal{L}^{\perp } \cap \{\bm{a}\}^{\perp })\} + 1 \) and
	 the following items hold:
	\begin{enumerate}[(i)]
		%\item\label{item:upper-bound-distance-PPS-log-det} .
		\item If \( \mathfrak{d} = 0 \), then \eqref{eq:conic-feasibility-problem} satisfies a Lipschitzian error bound.
		\item If \( \mathfrak{d} = 1 \), we have \( \mathcal{F}_2 = \{\bm{0}\} \) or \( \mathcal{F}_2 = \Fd \) or \( \mathcal{F}_2 = \Fs \) or \( \mathcal{F}_2 = \Fr \) or \( \mathcal{F}_2 = \Ff \).
		      \begin{enumerate}[(a)]
			      \item If \( \mathcal{F}_2 = \{\bm{0}\} \), then \eqref{eq:conic-feasibility-problem} satisfies a Lipschitzian error bound.
			      \item If \( \mathcal{F}_2 = \Fd \), then \eqref{eq:conic-feasibility-problem} satisfies an entropic error bound.\footnote{\label{footnote:def-gs} An entropic error bound is an error bound with the residual function being \( \mathfrak{g}_{{\rm d}} \), see Definition~\ref{def:error-bounds}. A log-type error bound refers to an error bound with the residual function being \( \mathfrak{g}_{\log} \). See \eqref{eq:BS-entropy} and \eqref{eq:g-log} for the definitions of \( \mathfrak{g}_{{\rm d}} \) and \( \mathfrak{g}_{\log} \), respectively.}
			      \item If \( \mathcal{F}_2 = \Fs \) %let \( \bm{n} \in \Kld^{*} \cap \mathcal{L} ^{\perp } \cap \{\bm{a}\}^{\perp } \) be such that \( \mathcal{F}_2 = \Fs = \Kld \cap \{\bm{n}\}^{\perp } \).
			            and \( \bm{n}^1_y > 0 \), then \eqref{eq:conic-feasibility-problem} satisfies a H\"olderian error bound with exponent \( \frac{1}{2} \).
			            If  \( \mathcal{F}_2 = \Fs \) and \( \bm{n}^1_y = 0 \), then \eqref{eq:conic-feasibility-problem} satisfies a log-type error bound.\footref{footnote:def-gs}
			      \item If \( \mathcal{F}_2 = \Fr \), then \eqref{eq:conic-feasibility-problem} satisfies a H\"olderian error bound with exponent \( \frac{1}{2} \).
			      \item If \( \mathcal{F}_2 = \Ff \)%, let \( \bm{n} \in \Kld^{*} \cap \mathcal{L}^{\perp } \cap \{\bm{a}\}^{\perp } \) be such that \( \mathcal{F}_2 = \Ff = \Kld \cap \{\bm{n}\}^{\perp } \).
			            and \( \bm{n}^1_y > 0 \), then \eqref{eq:conic-feasibility-problem} satisfies a Lipschitzian error bound.
			            If \( \mathcal{F}_2 = \Ff \) and \( \bm{n}^1_y = 0 \), then \eqref{eq:conic-feasibility-problem} satisfies a log-type error bound.\footref{footnote:def-gs}
		      \end{enumerate}
		\item If \( \mathfrak{d} \geq 2 \) %, consider the chain of faces
		%      \[ \mathcal{F}_{\mathfrak{d} + 1} \subsetneq \mathcal{F}_{\mathfrak{d}} \subsetneq \dots \subsetneq \mathcal{F}_2 \subsetneq \mathcal{F}_1 = \Kld \]
		      %of length \( \mathfrak{d} + 1 \),
		       we have \( \mathcal{F}_2 = \Fd \) or \( \mathcal{F}_2 \) is of form \( \Fs \).
		      Then, an error bound with residual function \( \underbrace{\mathfrak{h} \circ \mathfrak{h} \circ \cdots \circ \mathfrak{h}}_{\mathfrak{d} - 1} \circ \, \bar{\mathfrak{g}} \) holds, where \( \mathfrak{h} = | \cdot | ^{\frac{1}{2}} \) and
		      \begin{equation}
			      \label{eq:def-g-in-error-bound-whole-log-det-cone}
			      \bar{\mathfrak{g}} =
			      			      \begin{cases}
			      \mathfrak{g}_{{\rm d}}  & \text{ if } \mathcal{F}_2 = \Fd,                                                                            \\[0.1cm]
			      \mathfrak{g}_{\log}     & \text{ if } \mathcal{F}_2 = \Fs \text{ and } \bm{n}^1_y = 0, \\[0.1cm]
			      | \cdot |^{\frac{1}{2}} & \text{ if } \mathcal{F}_2 = \Fs \text{ and }  \bm{n}^1_y > 0.
			      \end{cases}
%			      \begin{cases}
%				      \mathfrak{g}_{{\rm d}}  & \text{ if } \mathcal{F}_2 = \Fd,                                                                            \\[0.1cm]
%				      \mathfrak{g}_{\log}     & \text{ if } \mathcal{F}_2 = \Fs \text{ exposed by } \bm{n} = (0, \bm{n}_y, \bm{n}_Z) \text{ with } \bm{n}_y = 0, \\[0.1cm]
%				      | \cdot |^{\frac{1}{2}} & \text{ if } \mathcal{F}_2 = \Fs \text{ exposed by } \bm{n} = (0, \bm{n}_y, \bm{n}_Z) \text{ with } \bm{n}_y > 0.
%			      \end{cases}
		      \end{equation}
	\end{enumerate}
\end{theorem}

\begin{proof}
	Following the discussion so far,
	if $\mathfrak{d} \geq 2$, it is because \( \mathcal{F}_2 = \Fd \) or \( \mathcal{F}_2 \) is of the form  \( \Fs \). Also, as
	remarked previously, in this case, $\mathcal{F}_2$ is a symmetric cone that is a direct product of a polyhedral cone (of rank at most $1$) and a symmetric cone of rank at most $d$.
	Considering the conic feasibility problem with \( \mathcal{K} = \mathcal{F}_2 \), it follows from~\cite[Proposition 24, Remark 39]{L17} that
	\[ \dpps(\mathcal{F}_2, \mathcal{L} + \bm{a}) \leq \min \{d - 1, \dim (\mathcal{L}^{\perp } \cap \{\bm{a}\}^{\perp })\}. \]
	Hence, by adding the first facial reduction step to get \( \mathcal{F}_2 \), we obtain the bound on $\mathfrak{d}$. %item%~\eqref{item:upper-bound-distance-PPS-log-det}.
Next, we examine the possibilities for $\mathfrak{d}$.
	
%	So after the face $\mathcal{F}$

%	We first recall the fact that a subface is still a face of a cone.
%	Hence, we know that for any chain of faces of \( \Kld \), each face is among those faces discussed in Proposition \ref{prop:facial-structure-log-det-cone} and the two trivial faces ($\{0\}$ and $\Kld$).
%	Moreover,  \( \Fr \) is not a subface of any other face except $\Kld$.
%	Therefore, from the discussion in Subsection~\ref{subsec:facial-structure}, especially~\eqref{eq:relationships-faces} and~\eqref{eq:relationships-faces-2}, we have the following possibilities for the procedure of facial reduction:
	\begin{enumerate}[$(i)$]
		\item If \( \mathfrak{d} = 0 \), then \eqref{eq:conic-feasibility-problem} satisfies the PPS condition and so a Lipschitzian error bound holds because of \cite[Corollary~6]{BBL99}.
		\item If \( \mathfrak{d} = 1\),
		then the possibilities for $\mathcal{F}_2$ are \( \{\bm{0}\},\, \Fd\), \( \, \Fs \), \( \Ff \) or \( \Fr  \).
		      Then, except for the case $\{\bm{0}\}$, the  error bound then follows from \cite[Theorem 3.8]{LiLoPo20} and
		      % analysis in Subsection~\ref{subsec:facial-residual-functions}, particularly
		      the facial residual functions computed in Corollaries~\ref{corollary:FRFs-F-d}, \ref{corollary:FRFs-F-3}, \ref{corollary:FRFs-F-inf} and \ref{corollary:FRFs-F-r}. The case
		      $\mathcal{F}_2=\{\bm{0}\}$ follows from  \cite[Proposition~27]{L17}.
	%	\item %After one single facial reduction step, a non-polyhedral face satisfying the PPS condition is attained, then \( \mathfrak{d} = 1 \).
	%	      In this case, the face can be \( \Fd \) or \( \Fs  \).
		   %   Then the error bound results follow from %\cite[Theorem 3.8]{LiLoPo20} and Corollaries~\ref{corollary:FRFs-F-d} and~\ref{corollary:FRFs-F-3}.
		\item In this case, it must hold that \( \mathcal{F}_2 = \Fd \) or \( \mathcal{F}_2 \) is of form \( \Fs \).
		Both cases, as discussed previously, correspond to symmetric cones.
		    %  We should also notice that \( \mathcal{F}_2 \) is a product of two polyhedral cones ( \( \R_- \) and \( \{0\} \) ) and a positive semidefinite cone (a face of a positive semidefinite cone is also a positive semidefinite cone), and hence both of the two possibilities of \( \mathcal{F}_2 \) are symmetric cones; see~\cite[Section 4]{L17} and~\cite{FK94} for further details.
		   %   Now, consider the conic feasibility problem with \( \mathcal{K} = \mathcal{F}_2 \), it follows from~\cite[Proposition 24, Remark 39]{L17} that
		 %     \[ \dpps(\mathcal{F}_2, \mathcal{L} + \bm{a}) \leq \min \{d - 1, \dim (\mathcal{L}^{\perp } \cap \{\bm{a}\}^{\perp })\}. \]
		  %    Hence, by adding the first facial reduction step to get \( \mathcal{F}_2 \), we show item~\eqref{item:upper-bound-distance-PPS-log-det}.
		      The error bound is obtained by invoking \cite[Theorem 3.8]{LiLoPo20} and
		      using the facial residual functions constructed in  Corollaries~\ref{corollary:FRFs-F-d} and~\ref{corollary:FRFs-F-3} and Lemma~\ref{lem:symmetric}.
	\end{enumerate}
\end{proof}
From Theorem~\ref{thm:error-bounds-log-det-cone} we
see the presence of \textit{non-H\"olderian} behaviour in the cases of entropic and logarithmic error bounds.
A similar phenomenon was observed in the study of error bounds for the exponential cone, see \cite[Section~4.4]{LiLoPo20}.
The analysis of convergence rates of algorithms under non-H\"olderian error bounds is still a challenge (see \cite[Sections~5 and 6]{LL22}) and $\Kld$ is thus another interesting test bed for research ideas on this topic.

%\section*{Acknowledgements}
%The third author was supported partly by the JSPS Grant-in-Aid for Early-Career Scientists 	23K16844 and the JSPS Grant-in-Aid  for Scientific Research (B)	21H03398. The fourth author was supported partly by the Hong Kong Research Grants Council PolyU153001/22p.

\bibliographystyle{abbrvurl}
\bibliography{bib_plain}

\end{document}